\documentclass[12pt,reqno]{amsart}
\usepackage{amssymb}
\usepackage{amscd}
\usepackage{hyperref}

\theoremstyle{plain}
\newtheorem{theorem}{Theorem}[section]
\newtheorem{corollary}[theorem]{Corollary}
\newtheorem{lemma}[theorem]{Lemma}
\newtheorem{proposition}[theorem]{Proposition}
\theoremstyle{definition}
\newtheorem{definition}[theorem]{Definition}

\newtheorem{remark}[theorem]{Remark}

\newtheorem{example}[theorem]{Example}

\numberwithin{equation}{section}
\begin{document}
\title[The equivariant index]{The eta invariant and equivariant index of
transversally elliptic operators}
\author[J.~Br\"{u}ning]{Jochen Br\"{u}ning}
\address{Institut f\"{u}r Mathematik \\
Humboldt Universit\"{a}t zu Berlin \\
Unter den Linden 6 \\
D-10099 Berlin, Germany}
\email[J.~Br\"{u}ning]{bruening@mathematik.hu-berlin.de}
\author[F. W.~Kamber]{Franz W.~Kamber}
\address{Department of Mathematics, University of Illinois \\
1409 W. Green Street \\
Urbana, IL 61801, USA}
\email[F. W.~Kamber]{kamber@math.uiuc.edu}
\author[K.~Richardson]{Ken Richardson}
\address{Department of Mathematics \\
Texas Christian University \\
Fort Worth, Texas 76129, USA}
\email[K.~Richardson]{k.richardson@tcu.edu}
\thanks{Work of the first author was partly supported by the grant SFB647
''Space-Time-Matter''.}
\date{\today }

\begin{abstract}
We prove a formula for the multiplicities of the index of an equivariant
transversally elliptic operator on a $G$-manifold. The formula is a sum of
integrals over blowups of the strata of the group action and also involves
eta invariants of associated elliptic operators. Among the applications, we
obtain an index formula for basic Dirac operators on Riemannian foliations,
a problem that was open for many years.
\end{abstract}

\maketitle
\tableofcontents

\section{Introduction}

Suppose that a compact Lie group $G$ acts by isometries on a compact,
connected Riemannian manifold $M$, and let $E=E^{+}\oplus E^{-}$ be a
graded, $G$-equivariant Hermitian vector bundle over $M$. We consider a
first order $G$-equivariant differential operator $D=D^{+}:$ $\Gamma \left(
M,E^{+}\right) \rightarrow \Gamma \left( M,E^{-}\right) $ that is
transversally elliptic, and let $D^{-}$ be the formal adjoint of $D^{+}$.

The group $G$ acts on $\Gamma \left( M,E^{\pm }\right) $ by $\left(
gs\right) \left( x\right) =g\cdot s\left( g^{-1}x\right) $, and the
(possibly infinite-dimensional) subspaces $\ker \left( D^{+}\right) $ and $%
\ker \left( D^{-}\right) $ are $G$-invariant subspaces. Let $\rho
:G\rightarrow U\left( V_{\rho }\right) $ be an irreducible unitary
representation of $G$, and let $\chi _{\rho }=\mathrm{tr}\left( \rho \right) 
$ denote its character. Let $\Gamma \left( M,E^{\pm }\right) ^{\rho }$ be
the subspace of sections that is the direct sum of the irreducible $G$%
-representation subspaces of $\Gamma \left( M,E^{\pm }\right) $ that are
unitarily equivalent to the representation $\rho $. It can be shown that the
extended operators 
\begin{equation*}
\overline{D}_{\rho ,s}:H^{s}\left( \Gamma \left( M,E^{+}\right) ^{\rho
}\right) \rightarrow H^{s-1}\left( \Gamma \left( M,E^{-}\right) ^{\rho
}\right)
\end{equation*}%
are Fredholm and independent of $s$, so that each irreducible representation
of $G$ appears with finite multiplicity in $\ker D^{\pm }$ (see Corollary %
\ref{FredholmSobolev}). Let $a_{\rho }^{\pm }\in \mathbb{Z}_{\geq 0}$ be the
multiplicity of $\rho $ in $\ker \left( D^{\pm }\right) $.

The study of index theory for such transversally elliptic operators was
initiated by M. Atiyah and I. Singer in the early 1970s (\cite{A}). The
virtual representation-valued index of $D$ is given by 
\begin{equation*}
\mathrm{ind}^{G}\left( D\right) :=\sum_{\rho }\left( a_{\rho }^{+}-a_{\rho
}^{-}\right) \left[ \rho \right] ,
\end{equation*}%
where $\left[ \rho \right] $ denotes the equivalence class of the
irreducible representation $\rho $. The index multiplicity is 
\begin{equation*}
\mathrm{ind}^{\rho }\left( D\right) :=a_{\rho }^{+}-a_{\rho }^{-}=\frac{1}{%
\dim V_{\rho }}\mathrm{ind}\left( \left. D\right\vert _{\Gamma \left(
M,E^{+}\right) ^{\rho }\rightarrow \Gamma \left( M,E^{-}\right) ^{\rho
}}\right) .
\end{equation*}%
In particular, if $\rho _{0}$ is the trivial representation of $G$, then 
\begin{equation*}
\mathrm{ind}^{\rho _{0}}\left( D\right) =\mathrm{ind}\left( \left.
D\right\vert _{\Gamma \left( M,E^{+}\right) ^{G}\rightarrow \Gamma \left(
M,E^{-}\right) ^{G}}\right) ,
\end{equation*}%
where the superscript $G$ implies restriction to $G$-invariant sections.

There is a clear relationship between the index multiplicities and Atiyah's
equivariant distribution-valued index $\mathrm{ind}_{g}\left( D\right) $;
the multiplicities determine the distributional index, and vice versa. Let $%
\left\{ X_{1},...,X_{r}\right\} $ be an orthonormal basis of the Lie algebra
of $G$. Let $\mathcal{L}_{X_{j}}$ denote the induced Lie derivative with
respect to $X_{j}$ on sections of $E$, and let $C=\sum_{j}\mathcal{L}%
_{X_{j}}^{\ast }\mathcal{L}_{X_{j}}$ be the Casimir operator on sections of $%
E$. The space $\Gamma \left( M,E^{\pm }\right) ^{\rho }$ is a subspace of
the $\lambda _{\rho }$-eigenspace of $C$. The virtual character $\mathrm{ind}%
_{g}\left( D\right) $ is given by (see \cite{A}) 
\begin{eqnarray*}
\mathrm{ind}_{g}\left( D\right) &:&=\text{\textquotedblleft }\mathrm{tr}%
\left( \left. g\right\vert _{\ker D^{+}}\right) -\mathrm{tr}\left( \left.
g\right\vert _{\ker D^{-}}\right) \text{\textquotedblright } \\
&=&\sum_{\rho }\mathrm{ind}^{\rho }\left( D\right) \chi _{\rho }\left(
g\right) .
\end{eqnarray*}%
Note that the sum above does not in general converge, since $\ker D^{+}$ and 
$\ker D^{-}$ are in general infinite-dimensional, but it does make sense as
a distribution on $G$. That is, if $dg$ is the normalized, biinvariant Haar
measure on $G$, and if $\phi =\beta +\sum c_{\rho }\chi _{\rho }\in
C^{\infty }\left( G\right) $, with $\beta $ orthogonal to the subspace of
class functions on $G$, then 
\begin{eqnarray*}
\mathrm{ind}_{\ast }\left( D\right) \left( \phi \right) &=&\text{%
\textquotedblleft }\int_{G}\phi \left( g\right) ~\overline{\mathrm{ind}%
_{g}\left( D\right) }~dg\text{\textquotedblright } \\
&=&\sum_{\rho }\mathrm{ind}^{\rho }\left( D\right) \int \phi \left( g\right)
~\overline{\chi _{\rho }\left( g\right) }~dg=\sum_{\rho }\mathrm{ind}^{\rho
}\left( D\right) c_{\rho },
\end{eqnarray*}%
an expression which converges because $c_{\rho }$ is rapidly decreasing and $%
\mathrm{ind}^{\rho }\left( D\right) $ grows at most polynomially as $\rho $
varies over the irreducible representations of $G$. From this calculation,
we see that the multiplicities determine Atiyah's distributional index.
Conversely, let $\alpha :G\rightarrow U\left( V_{\alpha }\right) $ be an
irreducible unitary representation. Then 
\begin{equation*}
\mathrm{ind}_{\ast }\left( D\right) \left( \chi _{\alpha }\right)
=\sum_{\rho }\mathrm{ind}^{\rho }\left( D\right) \int \chi _{\alpha }\left(
g\right) \overline{\chi _{\rho }\left( g\right) }\,dg=\mathrm{ind}^{\alpha
}D,
\end{equation*}%
so that in principle complete knowledge of the equivariant distributional
index is equivalent to knowing all of the multiplicities $\mathrm{ind}^{\rho
}\left( D\right) $. Because the operator $\left. D\right\vert _{\Gamma
\left( M,E^{+}\right) ^{\rho }\rightarrow \Gamma \left( M,E^{-}\right)
^{\rho }}$ is Fredholm, all of the indices $\mathrm{ind}^{G}\left( D\right) $
, $\mathrm{ind}_{g}\left( D\right) $, and $\mathrm{ind}^{\rho }\left(
D\right) $ depend only on the stable homotopy class of the principal
transverse symbol of $D$.

Let us now consider the heat kernel expression for the index multiplicities.
The usual McKean-Singer argument shows that, in particular, for every $t>0$,
the index $\mathrm{ind}^{\rho }\left( D\right) $ may be expressed as the
following iterated integral: 
\begin{gather}
\mathrm{ind}^{\rho }\left( D\right) =\int_{x\in M}\int_{g\in G}\,\mathrm{str~%
}g\cdot K\left( t,g^{-1}x,x\right) ~\overline{\chi _{\rho }\left( g\right) }%
~dg~\left\vert dx\right\vert  \notag \\
=\int_{x\in M}\int_{g\in G}\left( \mathrm{tr~}g\cdot K^{+}\left(
t,g^{-1}x,x\right) -\mathrm{tr~}g\cdot K^{-}\left( t,g^{-1}x,x\right)
\right) ~\overline{\chi _{\rho }\left( g\right) }~dg~\left\vert dx\right\vert
\label{indexIntegralIntroduction}
\end{gather}%
where $K^{\pm }\left( t,\cdot ,\cdot \right) \in \Gamma \left( M\times
M,E^{\mp }\boxtimes \left( E^{\pm }\right) ^{\ast }\right) $ is the kernel
for $e^{-t\left( D^{\mp }D^{\pm }+C-\lambda _{\rho }\right) }$ on $\Gamma
\left( M,E^{\pm }\right) $, letting $\left\vert dx\right\vert $ denote the
Riemannian density over $M$.

A priori, the integral above is singular near sets of the form 
\begin{equation*}
\bigcup\limits_{G_{x}\in \left[ H\right] }x\times G_{x}\subset M\times G,
\end{equation*}%
where the isotropy subgroup $G_{x}$ is the subgroup of $G$ that fixes $x\in
M $, and $\left[ H\right] $ is a conjugacy class of isotropy subgroups.

A\ large body of work over the last twenty years has yielded theorems that
express $\mathrm{ind}_{g}\left( D\right) $ and $\int_{M}\,\left( \mathrm{tr~}%
g\cdot K^{+}\left( t,g^{-1}x,x\right) -\mathrm{tr~}g\cdot K^{-}\left(
t,g^{-1}x,x\right) \right) ~\left\vert dx\right\vert $ in terms of
topological and geometric quantities, as in the Atiyah-Segal-Singer index
theorem for elliptic operators \cite{ASe} or the Berline-Vergne Theorem for
transversally elliptic operators \cite{Be-V1},\cite{Be-V2}. However, until
now there has been very little known about the problem of expressing $%
\mathrm{ind}^{\rho }\left( D\right) $ in terms of topological or geometric
quantities which are determined at the different strata of the $G$-manifold $%
M$. The special case when all of the isotropy groups are the same dimension
was solved by M. Atiyah in \cite{A}, and this result was utilized by T.
Kawasaki to prove the Orbifold Index Theorem (see \cite{Kawas2}). Our
analysis is new in that the integral over the group in (\ref%
{indexIntegralIntroduction}) is performed first, before integration over the
manifold, and thus the invariants in our index theorem are very different
from those seen in other equivariant index formulas. Theorem \ref%
{MainTheorem} gives a formula for the Fourier coefficients of the virtual
character instead of the value of the character at a particular $g\in G$.

Our main theorem (Theorem \ref{MainTheorem}) expresses $\mathrm{ind}^{\rho
}\left( D\right) $ as a sum of integrals over the different strata of the
action of $G$ on $M$, and it involves the eta invariant of associated
equivariant elliptic operators on spheres normal to the strata. The result is%
\begin{eqnarray*}
\mathrm{ind}^{\rho }\left( D\right) &=&\int_{G\diagdown \widetilde{M_{0}}%
}A_{0}^{\rho }\left( x\right) ~\widetilde{\left\vert dx\right\vert }%
~+\sum_{j=1}^{r}\beta \left( \Sigma _{\alpha _{j}}\right) ~, \\
\beta \left( \Sigma _{\alpha _{j}}\right) &=&\frac{1}{2\dim V_{\rho }}%
\sum_{b\in B}\frac{1}{n_{b}\mathrm{rank~}W^{b}}\left( -\eta \left(
D_{j}^{S+,b}\right) +h\left( D_{j}^{S+,b}\right) \right) \int_{G\diagdown 
\widetilde{\Sigma _{\alpha _{j}}}}A_{j,b}^{\rho }\left( x\right) ~\widetilde{%
\left\vert dx\right\vert }~,
\end{eqnarray*}%
(The notation will be explained later; the integrands $A_{0}^{\rho }\left(
x\right) $ and $A_{j,b}^{\rho }\left( x\right) $ are the familar
Atiyah-Singer integrands corresponding to local heat kernel supertraces of
induced elliptic operators over closed manifolds.) Even in the case when the
operator $D$ is elliptic, this result was not known previously. Further, the
formula above gives a method for computing eta invariants of Dirac-type
operators on quotients of spheres by compact group actions; these have been
computed previously only in some special cases. We emphasize that every part
of the formula is explicitly computable from local information provided by
the operator and manifold. Even the eta invariant of the operator $%
D_{j}^{S+,b}$ on a sphere is calculated directly from the principal symbol
of the operator $D$ at one point of a singular stratum. The de Rham operator
provides an important example illustrating the computability of the formula,
yielding a new theorem expressing the equivariant Euler characteristic in
terms of ordinary Euler characteristics of the strata of the group action
(Theorem \ref{EulerCharacteristicTheorem}).

One of the primary motivations for obtaining an explicit formula for $%
\mathrm{ind}^{\rho }\left( D\right) $ was to use it to produce a basic index
theorem for Riemannian foliations, thereby solving a problem that has been
open since the 1980s (it is mentioned, for example, in \cite{EK}). In fact,
the basic index theorem is a consequence of the invariant index theorem
(Theorem \ref{InvariantIndexTheorem}), corresponding to the trivial
representation $\rho _{0}$. This theorem is stated in Section \ref%
{BasicIndexSubsection}. We note that a recent paper of Gorokhovsky and Lott
addresses this transverse index question on Riemannian foliations. Using a
different technique, they are able to prove a formula for the basic index of
a basic Dirac operator that is distinct from our formula, in the case where
all the infinitesimal holonomy groups of the foliation are connected tori
and if Molino's commuting sheaf is abelian and has trivial holonomy (see 
\cite{GLott}). Our result requires at most mild topological assumptions on
the transverse structure of the strata of the Riemannian foliation and has a
similar form to the formula above for $\mathrm{ind}^{\rho _{0}}\left(
D\right) $. In particular, the analogue for the Gauss-Bonnet Theorem for
Riemannian foliations (Theorem \ref{BasicGaussBonnet}) is a corollary and
requires no assumptions on the structure of the Riemannian foliation.

There are several new techniques in this paper that have not been explored
previously. First, the fact that $\mathrm{ind}^{\rho }\left( D\right) $ is
invariant under $G$-equivariant homotopies is used in a very specific way,
and we keep track of the effects of these homotopies so that the formula for
the index reflects data coming from the original operator and manifold. In
Section \ref{BlowupDoubleMainSection} we describe a process of blowing up,
cutting, and reassembling the $G$-manifold into what is called the
desingularization, which also involves modifying the operator and vector
bundles near the singular strata as well. The result is a $G$-manifold that
has less intricate structure and for which the heat kernels are easier to
evaluate. The key idea is to relate the local asymptotics of the equivariant
heat kernel of the original manifold to the desingularized manifold; at this
stage the eta invariant appears through a direct calculation on the normal
bundle to the singular stratum. We note that our desingularization process
and the equivariant index theorem were stated and announced in \cite{RiOber}
and \cite{RiKor}; recently Albin and Melrose have taken it a step further in
tracking the effects of the desingularization on equivariant cohomology and
equivariant K-theory (\cite{AlMel}).

Another new idea in this paper is the decomposition of equivariant vector
bundles over $G$-manifolds with one orbit type. A crucial step in the proof
required the construction of a subbundle of an equivariant bundle over a $G$%
-invariant part of a stratum that is the minimal $G$-bundle decomposition
that consists of direct sums of isotypical components of the bundle. We call
this decomposition the \emph{fine decomposition} and define it in Section %
\ref{FineDecompSection}. A more detailed account of this method will appear
in \cite{KRi}.

There are certain assumptions on the operator $D$ that are required to
produce the formula in the main theorem, mainly that it has product-like
structure near the singular strata after a $G$-equivariant homotopy. We note
that most of the major examples of transversally elliptic differential
operators have the required properties. In fact, in \cite{PrRi}, a large
variety of examples of naturally defined transversal operators similar to
Dirac operators are explored and shown under most conditions to provide all
possible index classes of equivariant transversally elliptic operators.
These operators almost always satisfy the required product condition at the
singular strata (see Section \ref{productAssumptionSection}). Further, we
note that when the basic index problem for transverse Dirac-type operators
on Riemannian foliations is converted to an invariant index problem, the
resulting operator on the $G$-manifold again satisfies the required
assumptions, under mild topological conditions on the bundle and principal
transverse symbol of the operator.

The outline of the paper is as follows. We first review the stratification
of $G$-manifolds in Section \ref{stratification} and establish the
elementary properties of transversally elliptic operators in Section \ref%
{TransEllipOpsSection}. In Section \ref{OneIsotropyTypeMainSection}, we
discuss equivariant analysis on manifolds with one orbit type and decompose
equivariant bundles over them in various ways, including the fine
decomposition mentioned above. The relevant properties of the supertrace of
the equivariant heat kernel are discussed in Section \ref%
{heatKernelEquivIndexSection}. In Section \ref{inducedOperatorSphereSection}%
, we compute the local contribution of the supertrace of a general constant
coefficient equivariant heat operator in the neighborhood of a singular
point of an orthogonal group action on a sphere. It is here that the
equivariant index is related to a boundary value problem, which explains the
presence of eta invariants in the main theorem. In Section \ref%
{BlowupDoubleMainSection}, we describe the iterative process of
desingularizing the $G$-manifold near a minimal stratum and producing a
double of the manifold with one less stratum. We apply the heat kernel
analysis, representation theory, and fine decomposition to produce a heat
kernel splitting formula in Section \ref{newMultiplicitiesSection}. This
process leads to a reduction theorem for the equivariant heat supertrace
(Theorem \ref{supertraceReduction}). The main theorem follows and is stated
in Section \ref{MainTheoremSubsection}. Several examples are discussed in
Section \ref{ExamplesApplicationsSection}; these examples show that all of
the terms in the formula above are nontrivial. The basic index theorem for
Riemannian foliations is stated and discussed in Section \ref%
{BasicIndexSubsection}.

We note that other researchers have investigated transverse index theory of
group actions and foliations in noncommutative geometry and topology; these
papers answer different questions about the analysis and topology of certain
groupoids (see \cite{ConnesMosc}). A survey of index theory, primarily in
the noncommutative geometry setting, can be found in \cite{Kord}.

We thank James Glazebrook, Efton Park and Igor Prokhorenkov for helpful
discussions. The authors would like to thank variously the Mathematisches
Forschungsinstitut Oberwolfach, the Erwin Schr\"{o}dinger International
Institute for Mathematical Physics (ESI), Vienna, the Department for
Mathematical Sciences (IMF) at Aarhus University, the Centre de Recerca Matem%
\`{a}tica (CRM), Barcelona, and the Department of Mathematics at TCU for
hospitality and support during the preparation of this work.

\section{Stratifications of $G$-manifolds}

\label{stratification}In the following, we will describe some standard
results from the theory of Lie group actions (see \cite{Bre}, \cite{Kaw}).
Such $G$-manifolds are stratified spaces, and the stratification can be
described explicitly. In the following discussion, $G$ is a compact Lie
group acting on a smooth, connected, closed manifold $M$. We assume that the
action is effective, meaning that no $g\in G$ fixes all of $M$. (Otherwise,
replace $G$ with $G\diagup \left\{ g\in G:gx=x\text{ for all }x\in M\right\} 
$.) Choose a Riemannian metric for which $G$ acts by isometries; average the
pullbacks of any fixed Riemannian metric over the group of diffeomorphisms
to obtain such a metric.

Given such an action and $x\in M$, the isotropy or stabilizer subgroup $%
G_{x}<G$ is defined to be $\left\{ g\in G:gx=x\right\} $. The orbit $%
\mathcal{O}_{x}$ of a point $x$ is defined to be $\left\{ gx:g\in G\right\} $%
. Since $G_{xg}=gG_{x}g^{-1}$, the conjugacy class of the isotropy subgroup
of a point is fixed along an orbit.

On any such $G$-manifold, the conjugacy class of the isotropy subgroups
along an orbit is called the \textbf{orbit type}. On any such $G$-manifold,
there are a finite number of orbit types, and there is a partial order on
the set of orbit types. Given subgroups $H$ and $K$ of $G$, we say that $%
\left[ H\right] \leq $ $\left[ K\right] $ if $H$ is conjugate to a subgroup
of $K$, and we say $\left[ H\right] <$ $\left[ K\right] $ if $\left[ H\right]
\leq $ $\left[ K\right] $ and $\left[ H\right] \neq $ $\left[ K\right] $. We
may enumerate the conjugacy classes of isotropy subgroups as $\left[ G_{0}%
\right] ,...,\left[ G_{r}\right] $ such that $\left[ G_{i}\right] \leq \left[
G_{j}\right] $ implies that $i\leq j$. It is well-known that the union of
the principal orbits (those with type $\left[ G_{0}\right] $) form an open
dense subset $M_{0}$ of the manifold $M$, and the other orbits are called 
\textbf{singular}. As a consequence, every isotropy subgroup $H$ satisfies $%
\left[ G_{0}\right] \leq \left[ H\right] $. Let $M_{j}$ denote the set of
points of $M$ of orbit type $\left[ G_{j}\right] $ for each $j$; the set $%
M_{j}$ is called the \textbf{stratum} corresponding to $\left[ G_{j}\right] $%
. If $\left[ G_{j}\right] \leq \left[ G_{k}\right] $, it follows that the
closure of $M_{j}$ contains the closure of $M_{k}$. A stratum $M_{j}$ is
called a \textbf{minimal stratum} if there does not exist a stratum $M_{k}$
such that $\left[ G_{j}\right] <\left[ G_{k}\right] $ (equivalently, such
that $\overline{M_{k}}\subsetneq \overline{M_{j}}$). It is known that each
stratum is a $G$-invariant submanifold of $M$, and in fact a minimal stratum
is a closed (but not necessarily connected) submanifold. Also, for each $j$,
the submanifold $M_{\geq j}:=\bigcup\limits_{\left[ G_{k}\right] \geq \left[
G_{j}\right] }M_{k}$ is a closed, $G$-invariant submanifold.

Now, given a proper, $G$-invariant submanifold $S$ of $M$ and $\varepsilon
>0 $, let $T_{\varepsilon }(S)$ denote the union of the images of the
exponential map at $s$ for $s\in S$ restricted to the open ball of radius $%
\varepsilon $ in the normal bundle at $S$. It follows that $T_{\varepsilon
}(S)$ is also $G$ -invariant. If $M_{j}$ is a stratum and $\varepsilon $ is
sufficiently small, then all orbits in $T_{\varepsilon }\left( M_{j}\right)
\setminus M_{j}$ are of type $\left[ G_{k}\right] $, where $\left[ G_{k}%
\right] <\left[ G_{j}\right] $. This implies that if $j<k$, $\overline{M_{j}}%
\cap \overline{M_{k}}\neq \varnothing $, and $M_{k}\subsetneq \overline{M_{j}%
}$, then $\overline{M_{j}}$ and $\overline{M_{k}}$ intersect at right
angles, and their intersection consists of more singular strata (with
isotropy groups containing conjugates of both $G_{k}$ and $G_{j}$).

Fix $\varepsilon >0$. We now decompose $M$ as a disjoint union of sets $%
M_{0}^{\varepsilon },\dots ,M_{r}^{\varepsilon }$. If there is only one
isotropy type on $M$, then $r=0$, and we let $M_{0}^{\varepsilon }=\Sigma
_{0}^{\varepsilon }=M_{0}=M$. Otherwise, for $j=r,r-1,...,0$, let $%
\varepsilon _{j}=2^{j}\varepsilon $, and let%
\begin{eqnarray}
{\ }\Sigma _{j}^{\varepsilon } &=&M_{j}\setminus \overline{%
\bigcup_{k>j}M_{k}^{\varepsilon }}  \label{Sigmajepsilon} \\
M_{j}^{\varepsilon } &=&T_{\varepsilon _{j}}\left( M_{j}\right) \setminus 
\overline{\bigcup_{k>j}M_{k}^{\varepsilon }},  \label{Mjepsilon}
\end{eqnarray}%
Thus, 
\begin{equation*}
{\ }T_{\varepsilon }\left( \Sigma _{j}^{\varepsilon }\right) \subset
M_{j}^{\varepsilon },~\Sigma _{j}^{\varepsilon }\subset M_{j}.
\end{equation*}

The following facts about this decomposition are contained in \cite[pp.~203ff%
]{Kaw}:

\begin{lemma}
\label{decomposition} For sufficiently small $\varepsilon >0$, we have, for
every $i\in \{0,\ldots ,r\}$:

\begin{enumerate}
\item $\displaystyle M=\coprod_{i=0}^{r}M_{i}^{\varepsilon }$ (disjoint
union).

\item $M_{i}^{\varepsilon }$ is a union of $G$-orbits; $\Sigma
_{i}^{\varepsilon }$ is a union of $G$-orbits.

\item The manifold $M_{i}^{\varepsilon }$ is diffeomorphic to the interior
of a compact $G$ -manifold with corners; the orbit space $M_{i}^{\varepsilon
}\diagup G$ is a smooth manifold that is isometric to the interior of a
triangulable, compact manifold with corners. The same is true for each $%
\Sigma _{i}^{\varepsilon }$.

\item If $\left[ G_{j}\right] $ is the isotropy type of an orbit in $%
M_{i}^{\varepsilon }$, then $j\leq i$ and $\left[ G_{j}\right] \leq \left[
G_{i}\right] $.

\item The distance between the submanifold $M_{j}$ and $M_{i}^{\varepsilon }$
for $j>i$ is at least $\varepsilon $.
\end{enumerate}
\end{lemma}

\begin{remark}
The lemma above remains true if at each stage $T_{\varepsilon }\left(
M_{j}\right) $ is replaced by any sufficiently small open neighborhood of $%
M_{j}$ that contains $T_{\varepsilon }\left( M_{j}\right) $, that is a union
of $G$-orbits, and whose closure is a manifold with corners.
\end{remark}

Let $\gamma $ be a geodesic orthogonal to $M_{j}$ through $w\in M_{j}$. This
situation occurs exactly when this geodesic is orthogonal both to the fixed
point set $M^{G_{j}}$ of $G_{j}$ and to the orbit $\mathcal{O}_{w}$ of $G$
containing $w$. For any $h\in G_{j}$, right multiplication by $h$ maps
geodesics orthogonal to $M^{G_{j}}$ through $w$ to themselves and likewise
maps geodesics orthogonal to $\mathcal{O}_{w}$ through $w$ to themselves.
Thus, the group $G_{j}$ acts orthogonally on the normal space to $w\in M_{j}$
by the differential of the left action. \ Observe in addition that there are
no fixed points for this action; that is, there is no element of the normal
space that is fixed by every $h\in G_{j}$. \ Since $G_{j}$ acts without
fixed points, the codimension of $M_{j}$ is at least two if $G_{j}$
preserves orientation.

\section{Properties of transversally elliptic operators\label%
{TransEllipOpsSection}}

Let $G$ be a compact Lie group, $M$ be a compact, Riemannian $G$-manifold,
and $E\rightarrow M$ a $G$-equivariant Hermitian vector bundle. For the
following constructions, we refer to \cite{BrH1}. The unitary $G$%
-representation on $L^{2}\left( M,E\right) $ admits an orthogonal Hilbert
sum decomposition with respect to equivalence classes $\left[ \rho \right] $
of irreducible representations $\rho :G\rightarrow U\left( V_{\rho }\right) $%
:%
\begin{equation}
L^{2}\left( M,E\right) \cong \bigoplus_{\left[ \rho \right] \in \widehat{G}%
}L^{2}\left( M,E\right) _{\rho }\otimes _{\mathbb{C}}V_{\rho }~.
\label{decompositionIrreducibles}
\end{equation}%
The component $L^{2}\left( M,E\right) _{\rho }$ of type $\left[ \rho \right] 
$ is given by 
\begin{equation*}
L^{2}\left( M,E\right) _{\rho }=\mathrm{Hom}_{G}\left( V_{\rho },L^{2}\left(
M,E\right) \right) \cong L^{2}\left( M,E\otimes _{\mathbb{C}}V_{\rho }^{\ast
}\right) ^{G}
\end{equation*}%
and (\ref{decompositionIrreducibles}) is determined up to unitary
equivalence by the inclusions (evaluation maps)%
\begin{equation}
i_{\rho }:L^{2}\left( M,E\right) _{\rho }\otimes _{\mathbb{C}}V_{\rho
}\rightarrow L^{2}\left( M,E\right) .  \label{evaluationMap}
\end{equation}%
We define $L^{2}\left( M,E\right) ^{\rho }\subset L^{2}\left( M,E\right) $
to be the image $i_{\rho }\left( L^{2}\left( M,E\right) _{\rho }\otimes _{%
\mathbb{C}}V_{\rho }\right) $ and $\Gamma \left( M,E\right) ^{\rho }=\Gamma
\left( M,E\right) \cap L^{2}\left( M,E\right) ^{\rho }$.

The orthogonal projections 
\begin{equation*}
P_{\rho }:L^{2}\left( M,E\right) \rightarrow L^{2}\left( M,E\right) ^{\rho }
\end{equation*}%
are given by integration over $G$:%
\begin{gather*}
L^{2}\left( M,E\right) \overset{\otimes I_{\rho }}{\rightarrow }L^{2}\left(
M,E\right) \otimes _{\mathbb{C}}\mathrm{End}_{\mathbb{C}}\left( V_{\rho
}\right) \cong L^{2}\left( M,E\otimes _{\mathbb{C}}V_{\rho }^{\ast }\right)
\otimes _{\mathbb{C}}V_{\rho } \\
\overset{\int_{G}\otimes I_{\rho }}{\rightarrow }L^{2}\left( M,E\otimes _{%
\mathbb{C}}V_{\rho }^{\ast }\right) ^{G}\otimes _{\mathbb{C}}V_{\rho }%
\overset{i_{\rho }}{\rightarrow }L^{2}\left( M,E\right) ^{\rho }
\end{gather*}

\begin{lemma}
\label{ProjectionLemma}Let $\rho $ be an irreducible unitary representation
of a compact Lie group $G$ on a complex vector space $V_{\rho }$. Let $\chi
_{\rho }\left( \cdot \right) =\mathrm{tr~}\rho \left( \cdot \right) $ be the
character of $\rho $, and let $\alpha :G\rightarrow U\left( Q\right) $ be
any other representation on a Hermitian vector space $Q$. Let $dh$ denote
the volume form induced by the biinvariant, normalized metric on $G$. Let $%
P_{\rho }^{\alpha }:Q\rightarrow Q^{\rho }$ denote the projection onto the
the subspace of $Q$ on which the restriction of $\alpha $ is of type $\rho $%
. Then the following equation of endomorphisms of $Q$ holds: 
\begin{equation*}
\int_{G}\overline{\chi _{\rho }}\left( h\right) \alpha \left( h\right) ~dh=%
\frac{1}{\dim V_{\rho }}P_{\rho }^{\alpha },
\end{equation*}
\end{lemma}

\begin{proof}
If $F$ is a class function on $G$, then for any element $g\in G$, 
\begin{eqnarray*}
\alpha \left( g\right) \int_{G}F\left( h\right) \alpha \left( h\right) ~dh
&=&\int_{G}F\left( h\right) \alpha \left( g\right) \alpha \left( h\right) ~dh
\\
&=&\int_{G}F\left( h\right) \alpha \left( ghg^{-1}\right) ~dh~\alpha \left(
g\right) ,\text{ and letting }\widetilde{h}=ghg^{-1} \\
&=&\int_{G}F\left( g^{-1}\widetilde{h}g\right) \alpha \left( \widetilde{h}%
\right) ~d\widetilde{h}~\alpha \left( g\right) \text{ since }dh\text{ is
biinvariant} \\
&=&\int_{G}F\left( \widetilde{h}\right) \alpha \left( \widetilde{h}\right) ~d%
\widetilde{h}~\alpha \left( g\right) =\int_{G}F\left( h\right) \alpha \left(
h\right) ~dh~\alpha \left( g\right) ,
\end{eqnarray*}%
so $\int_{G}F\left( h\right) \alpha \left( h\right) ~dh$ commutes with $%
\alpha \left( g\right) $ for all $g\in H$. One can see easily that this
remains true if we restrict the endomorphisms to an irreducible component $%
Q^{\beta }$ of type $\beta $. Schur's Lemma implies that on this component $%
\int_{G}F\left( h\right) \alpha \left( h\right) ~dh$ is a constant multiple
of the identity $I_{\beta }$. We evaluate the constant by taking traces. In
particular, on this component, 
\begin{eqnarray*}
\int_{G}\overline{\chi _{\rho }}\left( h\right) \mathrm{~}\alpha \left(
h\right) ~dh &=&\int_{G}\overline{\chi _{\rho }}\left( h\right) \mathrm{~}%
\beta \left( h\right) ~dh=cI_{\beta }\text{ implies} \\
\int_{G}\overline{\chi _{\rho }}\left( h\right) \chi _{\beta }\left(
h\right) ~dh &=&c~\dim V_{\rho }~,\text{ so} \\
c &=&\left\{ 
\begin{array}{ll}
~\frac{1}{\dim V_{\rho }} & \text{if }\beta =\rho \\ 
0 & \text{otherwise}%
\end{array}%
\right. .
\end{eqnarray*}
\end{proof}

Given a group representation $\theta :G\rightarrow U\left( V_{\theta
}\right) $ and an irreducible representation $\rho :G\rightarrow U\left(
V_{\tau }\right) $ of the compact Lie group $G$, the \emph{multiplicity space%
} of $\rho $ in $\theta $ is 
\begin{equation*}
\mathrm{Hom}_{G}\left( V_{\rho },V_{\theta }\right) ,
\end{equation*}%
and the \emph{multiplicity} of $\rho $ in $\theta $ is%
\begin{equation*}
\dim \mathrm{Hom}_{G}\left( V_{\rho },V_{\theta }\right) .
\end{equation*}%
Given a subgroup $H$ of $G$, and a representation $\sigma :H\rightarrow
U\left( W_{\sigma }\right) $, we may form the homogeneous vector bundle $%
G\times _{H}W_{\sigma }\rightarrow G\diagup H$. The space 
\begin{equation*}
\mathrm{Ind}\left( \sigma \right) =L^{2}\left( G\diagup H,G\times
_{H}W_{\sigma }\right)
\end{equation*}%
of sections is a representation space for $G$, called the \emph{%
representation of }$G$\emph{\ induced by }$\sigma $, or simply the \emph{%
induced representation}. By the Frobenius reciprocity theorem, $\mathrm{Ind}$
is the right adjoint functor to $\mathrm{Res}$, the restriction functor.
That is, given a representation $\tau :G\rightarrow U\left( V_{\tau }\right) 
$, $V_{\mathrm{Res}\left( \tau \right) }$ is the representation space $%
V_{\tau }$ with $H$-representation $\mathrm{Res}\left( \tau \right) =\left.
\tau \right\vert _{H}$, and the multiplicity space of $\tau $ in $\mathrm{Ind%
}\left( \sigma \right) $ satisfies 
\begin{eqnarray*}
\mathrm{Hom}_{G}\left( V_{\tau },\mathrm{Ind}\left( \sigma \right) \right)
&\cong &L^{2}\left( G\diagup H,\left( G\times _{H}W_{\sigma }\right) \otimes
_{\mathbb{C}}V_{\tau }^{\ast }\right) ^{G} \\
&\cong &\mathrm{Hom}_{H}\left( V_{\mathrm{Res}\left( \tau \right)
},W_{\sigma }\right)
\end{eqnarray*}%
for every representation $\sigma :H\rightarrow U\left( W_{\sigma }\right) $.
This follows from the fact that a section $s$ of the homogeneous vector
bundle $G\times _{H}W_{\sigma }$ is given by an $H$-equivariant function $%
f_{s}:G\rightarrow W_{\sigma }$, and $f_{s}$ is a $G$-invariant section if
and only if $f_{s}$ is a constant in $W_{\sigma }^{H}$.

Let $D$ be a $G$-equivariant, symmetric differential operator of order $k$
on $\Gamma \left( M,E\right) $, which is transversally elliptic with respect
to the $G$-action (see \cite{A}). This means that the principal symbol $%
\sigma _{k}\left( D\right) \left( \xi \right) $ of $D$ is invertible on all
nonzero 
\begin{eqnarray*}
\xi &\in &T_{G}^{\ast }\left( M\right) _{x}=\{\xi \in T_{x}^{\ast }\left(
M\right) ~|~\xi \left( X\right) =0\  \\
&&\text{for every}~X\ \text{tangent to the orbit at }x.\}
\end{eqnarray*}%
Then $D^{2}$ is a $G$-equivariant, symmetric differential operator of order $%
2k$ on $\Gamma \left( M,E\right) $, which is transversally strongly elliptic.

The following proposition is contained to some extent in \cite{A} and \cite%
{BrH1}.

\begin{proposition}
Let $D$ be a $G$-equivariant, symmetric, transversally elliptic differential
operator of order $k$ on $\Gamma \left( M,E\right) $ .

\begin{enumerate}
\item The operators%
\begin{equation*}
D_{\rho }=\left( D\otimes I_{\rho ^{\ast }}\right) ^{G}=\left. D\otimes
I_{\rho ^{\ast }}\right\vert _{\Gamma \left( M,E\right) _{\rho }}
\end{equation*}%
(and their powers) are essentially self-adjoint on $L^{2}\left( M,E\right)
_{\rho }$ and generate the self-adjoint operators $D_{\rho }^{\ast }=%
\mathcal{S}_{\rho }$ .

\item The operator $D$ (and its powers) is essentially self-adjoint on $%
L^{2}\left( M,E\right) $ and generates a $G$-equivariant, self-adjoint
operator $\mathcal{R}$. With respect to the Hilbert sum decomposition (\ref%
{decompositionIrreducibles}), the operator $\mathcal{R}$ decomposes as a sum%
\begin{equation*}
\mathcal{R}=\bigoplus_{\left[ \rho \right] \in \widehat{G}}\mathcal{S}_{\rho
}\otimes I_{\rho }~.
\end{equation*}

\item Each component $\mathcal{S}_{\rho }$ has discrete spectrum without
finite accumulation points and admits a complete system of smooth
eigensections, and we have%
\begin{equation*}
E_{\lambda }\left( \mathcal{R}\right) _{\rho }\cong E_{\lambda }\left( 
\mathcal{S}_{\rho }\right) =E_{\lambda }\left( D_{\rho }\right) .
\end{equation*}

\item For each $\left[ \rho \right] \in \widehat{G}$, the eigenspaces of $%
\mathcal{S}_{\rho }\otimes I_{\rho }$ are finite $G$-modules of the form%
\begin{equation*}
E_{\lambda }\left( \mathcal{S}_{\rho }\otimes I_{\rho }\right) \cong
E_{\lambda }\left( \mathcal{S}_{\rho }\right) \otimes _{\mathbb{C}}V_{\rho
}~.
\end{equation*}
\end{enumerate}
\end{proposition}

\begin{proof}
$D_{\rho }^{2}$ is completed to a strongly elliptic operator $\widetilde{D}%
_{\rho }$ by adding the $k^{\mathrm{th}}$ power of an appropriate Casimir
operator, restricting to $D_{\rho }^{2}$ on $L^{2}\left( M,E\right) _{\rho
}=L^{2}\left( M,E\otimes _{\mathbb{C}}V_{\rho }^{\ast }\right) ^{G}$. It
follows that $D_{\rho }^{2}$ is essentially self-adjoint. Furthermore, $%
\left( \widetilde{D}_{\rho }+I\right) ^{\ast }$ is an isomorphism from the
Sobolev space $H^{2k}\left( M,E\right) $ to $L^{2}\left( M,E\right) $. To
see that the operators $D_{\rho }$ are essentially self-adjoint and generate
the self-adjoint operators $\mathcal{S}_{\rho }$ on $L^{2}\left( M,E\right)
_{\rho }$, suppose that $D_{\rho }^{\ast }u=\pm \iota ~u$, with $u$ nonzero
in the domain $\mathrm{Dom}\left( D_{\rho }^{\ast }\right) $. Then we have%
\begin{equation*}
\,\left\langle D_{\rho }s,u\right\rangle =\left\langle s,D_{\rho }^{\ast
}u\right\rangle =\mp \iota \left\langle s,u\right\rangle
\end{equation*}%
for all $s\in \Gamma \left( M,E\right) _{\rho }$. Therefore,%
\begin{equation*}
\left\langle D_{\rho }^{2}s,u\right\rangle =\left\langle D_{\rho }s,D_{\rho
}^{\ast }u\right\rangle =\mp \iota \left\langle D_{\rho }s,u\right\rangle
=\mp \iota \left\langle s,D_{\rho }^{\ast }u\right\rangle =-\left\langle
s,u\right\rangle .
\end{equation*}%
That is, 
\begin{equation*}
\left\langle s,\left( \widetilde{D}_{\rho }+I\right) ^{\ast }u\right\rangle
=\left\langle s,\left( D_{\rho }^{2}+I\right) ^{\ast }u\right\rangle
=\left\langle \left( D_{\rho }^{2}+I\right) s,u\right\rangle =0,
\end{equation*}%
a contradiction.

Next, the family $\left\{ \mathcal{S}_{\rho }\right\} _{\left[ \rho \right]
\in \widehat{G}}$ of self-adjoint operators defines a $G$-equivariant,
self-adjoint operator 
\begin{equation*}
\mathcal{R}=\bigoplus_{\left[ \rho \right] \in \widehat{G}}\mathcal{S}_{\rho
}\otimes I_{\rho }
\end{equation*}%
on $L^{2}\left( M,E\right) $ whose domain is given by%
\begin{equation*}
\left\{ \left. u=\left( u_{\rho }\right) _{\left[ \rho \right] \in \widehat{G%
}}\right\vert 
\begin{array}{c}
u_{\rho }\in \mathrm{Dom}\left( \mathcal{S}_{\rho }\right) \otimes _{\mathbb{%
C}}I_{\rho }, \\ 
\sum_{\left[ \rho \right] }\left( \left\Vert u_{\rho }\right\Vert
^{2}+\left\Vert \left( \mathcal{S}_{\rho }\otimes I_{\rho }\right) u_{\rho
}\right\Vert ^{2}\right) <\infty%
\end{array}%
\right\}
\end{equation*}%
and satisfies $P_{\rho }\mathrm{Dom}\left( \mathcal{R}\right) =\mathrm{Dom}%
\left( \mathcal{S}_{\rho }\right) \otimes I_{\rho }\subset \mathrm{Dom}%
\left( \mathcal{R}\right) $. Thus%
\begin{equation*}
\mathcal{R}u=\sum_{\left[ \rho \right] }\left( \mathcal{S}_{\rho }\otimes
I_{\rho }\right) u_{\rho }\text{ },~u\in \mathrm{Dom}\left( \mathcal{R}%
\right) ,~u_{\rho }=P_{\rho }u.
\end{equation*}%
For $s\in \Gamma \left( M,E\right) $, $u\in \mathrm{Dom}\left( \mathcal{R}%
\right) $, we have%
\begin{eqnarray*}
\left\langle s,\mathcal{R}u\right\rangle &=&\sum_{\left[ \rho \right]
}\left\langle P_{\rho }s,\left( \mathcal{S}_{\rho }\otimes I_{\rho }\right)
u_{\rho }\right\rangle \\
&=&\sum_{\left[ \rho \right] }\left\langle \left( D_{\rho }\otimes I_{\rho
}\right) P_{\rho }s,u_{\rho }\right\rangle \\
&=&\sum_{\left[ \rho \right] }\left\langle P_{\rho }Ds,u_{\rho
}\right\rangle =\sum_{\left[ \rho \right] }\left\langle Ds,u_{\rho
}\right\rangle =\left\langle Ds,u\right\rangle ,
\end{eqnarray*}%
and thus $\mathrm{Dom}\left( \mathcal{R}\right) \subseteq \mathrm{Dom}\left(
D^{\ast }\right) $. Since the operators $\mathcal{S}_{\rho }=D_{\rho }^{\ast
\ast }=D_{\rho }^{\ast }$ are the closures of the components $D_{\rho }$ of $%
D$ and $\mathcal{S}_{\rho }\otimes I_{\rho }=D_{\rho }^{\ast \ast }\otimes
I_{\rho }=\left( D_{\rho }\otimes I_{\rho }\right) ^{\ast \ast }$, it
follows that%
\begin{equation*}
D^{\ast \ast }=\bigoplus_{\left[ \rho \right] }\left( D_{\rho }\otimes
I_{\rho }\right) ^{\ast \ast }=\bigoplus_{\left[ \rho \right] }\mathcal{S}%
_{\rho }\otimes I_{\rho }=\mathcal{R}.
\end{equation*}%
Since $\mathcal{R}$ is self-adjoint, it follows that $D$ is essentially
self-adjoint with closure $\mathcal{R}$.
\end{proof}

\begin{corollary}
\label{FredholmSobolev}The operators $D_{\rho }$ induce operators $\overline{%
D}_{\rho ,s}$ 
\begin{equation*}
\overline{D}_{\rho ,s}:H^{s}\left( \Gamma \left( M,E^{+}\right) ^{\rho
}\right) \rightarrow H^{s-1}\left( \Gamma \left( M,E^{-}\right) ^{\rho
}\right)
\end{equation*}%
that are Fredholm and independent of $s$.
\end{corollary}

\begin{corollary}
$D$ admits a complete system of smooth eigensections. Since the spectrum is
given by $Spec\left( D\right) =\bigcup_{\left[ \rho \right] }Spec\left(
D_{\rho }\right) $, it need not be discrete, and the eigenvalues may have
infinite multiplicities.
\end{corollary}

Note that the eigenvalue counting function of $D_{\rho }^{2}$ (of order $2k$%
) satisfies the asymptotic formula%
\begin{align}
N\left( \lambda \right) :& =\sum_{\lambda ^{\rho }\leq \lambda }\dim
E_{\lambda ^{\rho }}\left( D_{\rho }\right)  \notag \\
& \sim c\lambda ^{m/2k},  \label{eigenvalueAsymptotics}
\end{align}%
where $m$ is the dimension of $M\diagup G$ (see \cite{BrH1}, applied to the
strongly elliptic operator $D^{2}+C-\lambda _{\rho }$).

\section{The case of one isotropy type\label{OneIsotropyTypeMainSection}}

In this section we consider the equivariant index problem and representation
theory in the case where the group action has one isotropy type. The
interested reader may consult the paper \cite{KRi} to obtain more detailed
exposition and more explicit and extensive results on $G$-bundles over
manifolds with one isotropy type.

\subsection{Induced bundles over the orbit space\label%
{OneIsotropyTypeSection}}

The statements in this subsection are for the most part known (see, for
example, \cite{BrH1}), but we will use them later in the paper. Let $G$ be a
compact Lie group, and let $X$ be a (not necessarily closed) Riemannian $G$%
-manifold, and suppose that there is only one isotropy type. For all $x\in X$%
, let $\mathcal{O}_{x}$ denote the orbit $Gx$. Let $\overline{X}$ denote the
manifold $G\setminus X$, and let $\pi :X\rightarrow \overline{X}$ denote the
projection. The metric $g$ on $X$ induces a metric $\overline{g}$ on $%
\overline{X}$ defined uniquely by 
\begin{equation*}
\overline{g}_{\pi \left( x\right) }\left( \pi _{\ast }v,\pi _{\ast }w\right)
=\left( \mathrm{vol}\left( \mathcal{O}_{x}\right) \right) ^{2/\dim \left( 
\overline{X}\right) }g\left( v,w\right)
\end{equation*}%
for any $x\in X$ and $v,w\in N_{x}\left( \mathcal{O}_{x}\right) \subset
T_{x}X$. The scale factor is chosen so that $\mathrm{vol~}X=\mathrm{vol~}%
\overline{X}$. Let $E$ be a Hermitian, $G$-equivariant vector bundle over $X$%
, and let $\rho :G\rightarrow U\left( V_{\rho }\right) $ be an irreducible
unitary representation. Denote by $G_{x}$ the isotropy group at $x\in X$.
Define the vector spaces $\mathcal{E}_{\rho }\left( \mathcal{O}_{x}\right) $
and $\mathcal{E}^{\rho }\left( \mathcal{O}_{x}\right) $ by 
\begin{eqnarray*}
\mathcal{E}_{\rho }\left( \mathcal{O}_{x}\right) &:&=L^{2}\left( \mathcal{O}%
_{x},E\right) _{\rho }=\mathrm{Hom}_{G}\left( V_{\rho },L^{2}\left( \mathcal{%
O}_{x},E\right) \right) \\
&\cong &\mathrm{Hom}_{G}\left( V_{\rho },L^{2}\left( \mathcal{O}_{x},G\times
_{G_{x}}E_{x}\right) \right) \\
&\cong &\mathrm{Hom}_{G_{x}}\left( V_{\rho },E_{x}\right) \cong \left(
E_{x}\otimes V_{\rho }^{\ast }\right) ^{G_{x}}, \\
\mathcal{E}^{\rho }\left( \mathcal{O}_{x}\right) &:&=L^{2}\left( \mathcal{O}%
_{x},E\right) ^{\rho }=i_{\rho }\left( \mathcal{E}_{\rho }\left( \pi \left(
x\right) \right) \otimes _{\mathbb{C}}V_{\rho }\right) .
\end{eqnarray*}%
Note that $i_{\rho }$ above is not necessarily one to one. Both $\mathcal{E}%
_{\rho }\left( \mathcal{O}_{x}\right) $ and $\mathcal{E}^{\rho }\left( 
\mathcal{O}_{x}\right) $ are well-defined vector spaces whose dimensions are
constant, and moreover $\mathcal{E}_{\rho }\left( \cdot \right) $ and $%
\mathcal{E}^{\rho }\left( \cdot \right) $ form Hermitian vector bundles over 
$\overline{X}$ (see \cite[Lemma 1.2 ff]{BrH1}). In fact, the vector spaces $%
E_{\rho ,x}$ defined for each $x\in X$ by%
\begin{equation*}
E_{\rho ,x}=\mathrm{Hom}_{G_{x}}\left( V_{\rho },E_{x}\right)
\end{equation*}%
form a $G$-invariant subbundle $E_{\rho }$ of $\mathrm{Hom}\left( V_{\rho
},E\right) $ over $X$. The Hermitian structure on $\mathcal{E}^{\rho }$ is
defined uniquely by%
\begin{equation*}
\left\langle s_{1},s_{2}\right\rangle _{\pi \left( x\right) }=\left\langle
s_{1}\left( x\right) ,s_{2}\left( x\right) \right\rangle _{x}
\end{equation*}%
for every $x\in X$. With these choices, we have canonical isomorphisms 
\begin{eqnarray}
\Gamma \left( \overline{X},\mathcal{E}_{\rho }\right) &\cong &\Gamma \left(
X,E_{\rho }\right) ^{G}\cong \Gamma \left( X,E\right) _{\rho },~\Gamma
\left( \overline{X},\mathcal{E}^{\rho }\right) \cong \Gamma \left(
X,E\right) ^{\rho }  \notag \\
L^{2}\left( \overline{X},\mathcal{E}_{\rho }\right) &\cong &L^{2}\left(
X,E_{\rho }\right) ^{G}\cong L^{2}\left( X,E\right) _{\rho },~L^{2}\left( 
\overline{X},\mathcal{E}^{\rho }\right) \cong L^{2}\left( X,E\right) ^{\rho
}.  \label{sectionsSameOnQuotient}
\end{eqnarray}%
If $E$ has a grading, there are gradings induced on $\mathcal{E}_{\rho }$
and $\mathcal{E}^{\rho }$ in the obvious way.

\subsection{Equivariant operators on $G$-manifolds with one orbit type}

Given a transversally elliptic, $G$-equivariant operator $D$ on sections of $%
E$ over $X$, the operator 
\begin{equation*}
D^{\rho }=\left. D\right\vert _{\Gamma \left( X,E\right) ^{\rho }}
\end{equation*}%
induces an operator, called $\mathcal{D}^{\rho }$, on sections of $\mathcal{E%
}^{\rho }$ over $\overline{X}$. Similarly, $D$ induces an operator $\mathcal{%
D}_{\rho }$ on sections of $\mathcal{E}_{\rho }$ over $\overline{X}$.
Furthermore, $\mathcal{D}^{\rho }$ and $\mathcal{D}_{\rho }$ are elliptic
operators. If $D$ is essentially self-adjoint on $X$, then $\mathcal{D}%
^{\rho }$ and $\mathcal{D}_{\rho }$ are also essentially self-adjoint on $%
L^{2}\left( \overline{X},\mathcal{E}^{\rho }\right) $ and $L^{2}\left( 
\overline{X},\mathcal{E}_{\rho }\right) $. The supertrace of the equivariant
heat kernel corresponding to $D$ restricted to $\Gamma \left( X,E\right)
^{\rho }$ can be identified with the supertrace of the (ordinary) heat
kernel associated to the elliptic operator $\mathcal{D}^{\rho }$ on $\Gamma
\left( \overline{X},\mathcal{E}^{\rho }\right) $ (similarly $\mathcal{D}%
_{\rho }$ on $\Gamma \left( \overline{X},\mathcal{E}_{\rho }\right) $.)

Thus,%
\begin{eqnarray*}
\mathrm{ind}^{\rho }\left( D\right) &=&\frac{1}{\dim V_{\rho }}\mathrm{ind}%
\left( \mathcal{D}^{\rho }\right) \\
&=&\mathrm{ind}\left( \mathcal{D}_{\rho }\right) ,
\end{eqnarray*}%
which is the Atiyah-Singer index of the elliptic operator $\mathcal{D}_{\rho
}$ on the base manifold $\overline{X}$. This is also a special case of our
main theorem (Theorem \ref{MainTheorem}).

\subsection{The normalized isotypical decomposition}

We retain the notation of the previous subsections. For further detail
regarding topics in this section, we refer the reader to \cite{KRi}. Again,
we emphasize that $X$ has one isotropy type $\left[ H\right] $, so that the
orbits form a Riemannian foliation of $X$. Let $N=N\left( H\right) $; note
that $N$-bundles over the fixed point set $X^{H}$ induce $G$-bundles over $X$
and vice versa.

Let $E\rightarrow X$ be a given $G$-equivariant vector bundle over $X$. For
each $x\in X^{H}$, let $\Sigma _{E_{x}}=\left\{ \left[ \sigma :H\rightarrow
U\left( W_{\sigma }\right) \right] \text{ irreducible}:\mathrm{Hom}%
_{H}\left( W_{\sigma },E_{x}\right) \neq 0\right\} $ be the set of
equivalence classes of irreducible representations of $H$ present in $E_{x}$%
; by rigidity, $\Sigma _{E_{x}}$ is locally constant in $x$. Let $\Sigma
_{E}=\bigcup\limits_{x\in X^{H}}\Sigma _{E_{x}}$. The isotypical (or
primary) decomposition of the $H$-module $E_{x}$ is of the form 
\begin{equation*}
E_{x}=\bigoplus_{\left[ \sigma \right] \in \Sigma _{E_{x}}}E_{x}^{\left[
\sigma \right] }
\end{equation*}%
where%
\begin{equation*}
E_{x}^{\left[ \sigma \right] }=i_{\sigma }\left( \mathrm{Hom}_{H}\left(
W_{\sigma },E_{x}\right) \otimes W_{\sigma }\right) .
\end{equation*}%
Note that the representations $E_{x}^{\left[ \sigma \right] }$ are not
necessarily irreducible. Further, these subspaces do not necessarily form $N$%
-equivariant bundles as $x$ varies over $X^{H}\subset X$. In particular, if
the normalizer $N$ is disconnected it may be the case that an element $n\in
N $ maps a primary component $E_{x}^{\left[ \sigma \right] }$ to an
inequivalent primary component of $E_{nx}$, and $nx\in X^{H}$. The
representation of $H$ on $nE_{x}^{\left[ \sigma \right] }\subset E_{nx}$ is
given by%
\begin{equation*}
hv=n\left( n^{-1}hn\right) \left( n^{-1}v\right) ,\text{ }h\in H,v\in E_{nx}
\end{equation*}%
so that the representation on $nE_{x}^{\left[ \sigma \right] }$ is
equivalent to a direct sum of representations of type $\left[ \sigma
_{x}^{n}:H\rightarrow U\left( W_{\sigma _{x}}\right) \right] $ defined by%
\begin{equation}
\sigma _{x}^{n}\left( h\right) =\sigma _{x}\left( n^{-1}hn\right) .
\label{sigma-n-representation}
\end{equation}%
For this reason, we will instead use a more coarse decomposition of $E_{x}$.
For each $n\in N$, we see that $E_{x}^{\left[ \sigma ^{n}\right]
}=n^{-1}\left( E_{nx}^{\left[ \sigma \right] }\right) $, an $H$-invariant
subspace of $E_{x}$ that is the primary $\left[ \sigma ^{n}\right] $-part of
the representation of $H$ on $E_{x}$. The representation of $H$ on $E_{x}^{%
\left[ \sigma ^{n}\right] }$ is equivalent to the representation of $H$ on $%
nE_{x}^{\left[ \sigma \right] }\subset E_{nx}$. We say that the two
representations $\sigma _{x}$ and $\sigma _{x}^{n}$ are \textbf{%
normalizer-conjugate}. It is often the case that $\sigma _{x}^{n}$ is
equivalent to $\sigma _{x}$; as explained in \cite{KRi}, the rigidity of
irreducible representations implies that the set 
\begin{equation*}
\widetilde{N}_{\left[ \sigma \right] }=\left\{ n\in N:\left[ \sigma ^{n}%
\right] =\left[ \sigma \right] \right\} ,
\end{equation*}%
is a subgroup of $N$ such that $N\diagup \widetilde{N}_{\left[ \sigma \right]
}$ is finite. Let $n_{1},...,n_{k}\in N$ be elements such that $N\diagup 
\widetilde{N}_{\left[ \sigma \right] }=\left\{ n_{1}\widetilde{N}_{\left[
\sigma \right] }~,...,n_{k}\widetilde{N}_{\left[ \sigma \right] }\right\} $.
Let 
\begin{equation*}
\left( E_{x}\right) ^{\overline{\sigma }}=\bigoplus_{j=1}^{k}E_{x}^{\left[
\sigma ^{n_{j}}\right] }\subset E_{x}.
\end{equation*}%
Here, $\overline{\sigma }$ is the set of equivalence classes of irreducible
representations inside the set $\left\{ \sigma ^{n}:n\in N\right\} $, so
that we see that the set of all $\left( E_{x}\right) ^{\overline{\sigma }}$
in $E_{x}$ is in fact indexed by the finite set of such $\overline{\sigma }$%
. We decompose for each $x\in X^{H}\subset X$,%
\begin{equation*}
E_{x}=\bigoplus_{\overline{\sigma }}\left( E_{x}\right) ^{\overline{\sigma }%
}.
\end{equation*}%
These subspaces $\left( E_{x}\right) ^{\overline{\sigma }}$ are in fact
invariant under the action of $n\in N$ and patch together to form
well-defined $N$-equivariant subbundles of $\left. E\right\vert _{X^{H}}$ .
Thus, $G$ acts on these subspaces to produce well-defined subbundles of $E$
over each orbit, and these subbundles patch together to form well-defined $G$%
-subbundles $E^{1},...,E^{r}$ over $X$. In summary, we may decompose the
vector bundle $E$ as%
\begin{equation}
E=\bigoplus_{j}E^{j}~~,  \label{NormalizedIsotypicalDecomposition}
\end{equation}%
where each $G$-bundle $E^{j}$ is the direct sum of primary components $%
E^{\sigma }$ corresponding to irreducible representations $\sigma $ of the
isotropy subgroup that are all normalizer-conjugate. Further, the
representations present in $E^{j}$ are not normalizer-conjugate to those
present in $E^{k}$ for $j\neq k$. We call this the \textbf{normalized
isotypical decomposition}; see \cite{KRi} for details. Let $\widetilde{N}$
be the connected component of $N$ relative to $H$, so that $\pi _{0}\left(
N\diagup H\right) \cong N\diagup \widetilde{N}$. If we denote by $\widetilde{%
N}_{\left[ \sigma \right] }$ the subgroup (of finite index) of $N$ fixing $%
\left[ \sigma \right] \in \Sigma _{E}$, so that $\widetilde{N}\subseteq 
\widetilde{N}_{\left[ \sigma \right] }\subseteq N$ and $E^{\left[ \sigma %
\right] }$ is $\widetilde{N}_{\left[ \sigma \right] }$-equivariant, we may
rephrase the preceding construction by saying that the normalized isotypical
components are obtained by `inducing up' the isotypical components $E^{\left[
\sigma \right] }$ from $\widetilde{N}_{\left[ \sigma \right] }$ to $N$. That
is, the normalized isotypical component over $X^{H}$ induced from $\left[
\sigma \right] \in \Sigma _{E}$ is the set 
\begin{equation*}
N\times _{\widetilde{N}_{\left[ \sigma \right] }}E^{\left[ \sigma \right] }
\end{equation*}%
This gives an $N$-subbundle of $E$ over $X^{H}$ containing $E^{\left[ \sigma %
\right] }$ that induces the normalized isotypical component over all of $X$.

\subsection{The refined isotypical decomposition\label{FineDecompSection}}

With notation as above, let $X^{H}$ be the fixed point set of $H$, and for $%
\alpha \in \pi _{0}\left( X^{H}\right) $, let $X_{\alpha }^{H}$ denote the
corresponding connected component of $X^{H}$.

\begin{definition}
\label{componentRelGDefn}We denote $X_{\alpha }=GX_{\alpha }^{H}$, and $%
X_{\alpha }$ is called a \textbf{component of} $X$ \textbf{relative to} $G$.
\end{definition}

\begin{remark}
The space $X_{\alpha }$ is not necessarily connected, but it is the inverse
image of a connected component of $G\diagdown X=N\diagdown X^{H}$ under the
projection $X\rightarrow G\diagdown X$. Also, note that $X_{\alpha
}=X_{\beta }$ if there exists $n\in N$ such that $nX_{\alpha }^{H}=X_{\beta
}^{H}$. If $X$ is a closed manifold, then there are a finite number of
components of $X$ relative to $G$.
\end{remark}

We now introduce a decomposition of a $G$-bundle $E\rightarrow X$ over a $G$%
-space with single orbit type $\left[ H\right] $ that is a priori finer than
the normalized isotypical decomposition. Let $E_{\alpha }$ be the
restriction $\left. E\right\vert _{X_{\alpha }^{H}}$. As in the previous
section, let $\widetilde{N}_{\left[ \sigma \right] }=\left\{ n\in N:\left[
\sigma ^{n}\right] \text{~is~equivalent~to}~\left[ \sigma \right] ~\right\} $
. If the isotypical component $E_{\alpha }^{\left[ \sigma \right] }$ is
nontrivial, then it is invariant under the subgroup $\widetilde{N}_{\alpha ,%
\left[ \sigma \right] }\subseteq \widetilde{N}_{\left[ \sigma \right] }$
that leaves in addition the connected component $X_{\alpha }^{H}$ invariant;
again, this subgroup has finite index in $N$. The isotypical components
transform under $n\in N$ as%
\begin{equation*}
n:E_{\alpha }^{\left[ \sigma \right] }\overset{\cong }{\longrightarrow }E_{%
\overline{n}\left( \alpha \right) }^{\left[ \sigma ^{n}\right] }~,
\end{equation*}%
where $\overline{n}$ denotes the residue class class of $n\in N$ in $%
N\diagup \widetilde{N}_{\alpha ,\left[ \sigma \right] }~$. Then a
decomposition of $E$ is obtained by `inducing up' the isotypical components $%
E_{\alpha }^{\left[ \sigma \right] }$ from $\widetilde{N}_{\alpha ,\left[
\sigma \right] }$ to $N$. That is, 
\begin{equation*}
E_{\alpha ,\left[ \sigma \right] }^{N}=N\times _{\widetilde{N}_{\alpha ,%
\left[ \sigma \right] }}E_{\alpha }^{\left[ \sigma \right] }
\end{equation*}%
is a bundle containing $\left. E_{\alpha }^{\left[ \sigma \right]
}\right\vert _{X_{\alpha }^{H}}$ . This is an $N$-bundle over $NX_{\alpha
}^{H}\subseteq X^{H}$, and a similar bundle may be formed over each distinct 
$NX_{\beta }^{H}$, with $\beta \in \pi _{0}\left( X^{H}\right) $. Further,
observe that since each bundle $E_{\alpha ,\left[ \sigma \right] }^{N}$ is
an $N$-bundle over $NX_{\alpha }^{H}$, it defines a unique $G$ bundle $%
E_{\alpha ,\left[ \sigma \right] }^{G}$.

\begin{definition}
\label{fineComponentDefinition}The $G$-bundle $E_{\alpha ,\left[ \sigma %
\right] }^{G}$ over the submanifold $X_{\alpha }$ is called a \textbf{fine
component} or the \textbf{fine component of }$E\rightarrow X$ \textbf{%
associated to }$\left( \alpha ,\left[ \sigma \right] \right) $.
\end{definition}

If $G\diagdown X$ is not connected, one must construct the fine components
separately over each $X_{\alpha }$. If $E$ has finite rank, then $E$ may be
decomposed as a direct sum of distinct fine components over each $X_{\alpha
} $. In any case, $E_{\alpha ,\left[ \sigma \right] }^{N}$ is a finite
direct sum of isotypical components over each $X_{\alpha }^{H}$.

\begin{definition}
\label{FineDecompositionDefinition}The direct sum decomposition of $\left.
E\right\vert _{X_{\alpha }}$ into subbundles $E^{b}$ that are fine
components $E_{\alpha ,\left[ \sigma \right] }^{G}$ for some $\left[ \sigma %
\right] $, written 
\begin{equation*}
\left. E\right\vert _{X_{\alpha }}=\bigoplus\limits_{b}E^{b}~,
\end{equation*}%
is called the \textbf{refined isotypical decomposition} (or \textbf{fine
decomposition}) of $\left. E\right\vert _{X_{\alpha }}$.
\end{definition}

In the case where $G\diagdown X$ is connected, the group $\pi _{0}\left(
N\diagup H\right) $ acts transitively on the connected components $\pi
_{0}\left( X^{H}\right) $, and thus $X_{\alpha }=X$. We comment that if $%
\left[ \sigma ,W_{\sigma }\right] $ is an irreducible $H$-representation
present in $E_{x}$ with $x\in X_{\alpha }^{H}$, then $E_{x}^{\left[ \sigma %
\right] }$ is a subspace of a distinct $E_{x}^{b}$ for some $b$. The
subspace $E_{x}^{b}$ also contains $E_{x}^{\left[ \sigma ^{n}\right] }$ for
every $n$ such that $nX_{\alpha }^{H}=X_{\alpha }^{H}$~.

\begin{remark}
\label{constantMultiplicityRemark}Observe that by construction, for $x\in
X_{\alpha }^{H}$ the multiplicity and dimension of each $\left[ \sigma %
\right] $ present in a specific $E_{x}^{b}$ is independent of $\left[ \sigma %
\right] $. Thus, $E_{x}^{\left[ \sigma ^{n}\right] }$ and $E_{x}^{\left[
\sigma \right] }$ have the same multiplicity and dimension if $nX_{\alpha
}^{H}=X_{\alpha }^{H}$~.
\end{remark}

\begin{remark}
The advantage of this decomposition over the isotypical decomposition is
that each $E^{b}$ is a $G$-bundle defined over all of $X_{\alpha }$, and the
isotypical decomposition may only be defined over $X_{\alpha }^{H}$.
\end{remark}

This new decomposition is a priori a finer decomposition than the normal
isotypical decomposition into $G$-subbundles.

\begin{definition}
\label{adaptedDefn}Now, let $E$ be a $G$-equivariant vector bundle over $X$,
and let $E^{b}~$be a fine component as in Definition \ref%
{fineComponentDefinition} corresponding to a specific component $X_{\alpha
}=GX_{\alpha }^{H}$ of $X$ relative to $G$. Suppose that another $G$-bundle $%
W$ over $X_{\alpha }$ has finite rank and has the property that the
equivalence classes of $G_{y}$-representations present in $E_{y}^{b},y\in
X_{\alpha }$ exactly coincide with the equivalence classes of $G_{y}$%
-representations present in $W_{y}$, and that $W$ has a single component in
the fine decomposition. Then we say that $W$ is \textbf{adapted} to $E^{b}$.
\end{definition}

\begin{lemma}
\label{AdaptedToAnyBundleLemma}In the definition above, if another $G$%
-bundle $W$ over $X_{\alpha }$ has finite rank and has the property that the
equivalence classes of $G_{y}$-representations present in $E_{y}^{b},y\in
X_{\alpha }$ exactly coincide with the equivalence classes of $G_{y}$%
-representations present in $W_{y}$, then it follows that $W$ has a single
component in the fine decomposition and hence is adapted to $E^{b}$. Thus,
the last phrase in the corresponding sentence in the above definition is
superfluous.
\end{lemma}

\begin{proof}
Suppose that we choose an equivalence class $\left[ \sigma \right] $ of $H$%
-representations present in $W_{x}$, $x\in X_{\alpha }^{H}$. Let $\left[
\sigma ^{\prime }\right] $ be any other equivalence class; then, by
hypothesis, there exists $n\in N$ such that $nX_{\alpha }^{H}=X_{\alpha
}^{H} $ and $\left[ \sigma ^{\prime }\right] =\left[ \sigma ^{n}\right] $.
Then, observe that $nW_{x}^{\left[ \sigma \right] }=W_{nx}^{\left[ \sigma
^{n}\right] }=W_{x}^{\left[ \sigma ^{n}\right] }$, with the last equality
coming from the rigidity of irreducible $H$-representations. Thus, $W$ is
contained in a single fine component, and so it must have a single component
in the fine decomposition.
\end{proof}

\subsection{Canonical isotropy $G$-bundles}

In what follows, we show that there are naturally defined finite-dimensional
vector bundles that are adapted to any fine components. Once and for all, we
enumerate the irreducible representations $\left\{ \left[ \rho _{j},V_{\rho
_{j}}\right] \right\} _{j=1,2,...}$ of $G$. Let $\left[ \sigma ,W_{\sigma }%
\right] $ be any irreducible $H$-representation. Let $G\times _{H}W_{\sigma
} $ be the corresponding homogeneous vector bundle over the homogeneous
space $G\diagup H$. Then the $L^{2}$-sections of this vector bundle
decompose into irreducible $G$-representations. In particular, let $\left[
\rho _{j_{0}},V_{\rho _{j_{0}}}\right] $ be the equivalence class of
irreducible representations that is present in $L^{2}\left( G\diagup
H,G\times _{H}W_{\sigma }\right) $ and that has the lowest index $j_{0}$.
Then Frobenius reciprocity implies%
\begin{equation*}
0\neq \mathrm{Hom}_{G}\left( V_{\rho _{j_{0}}},L^{2}\left( G\diagup
H,G\times _{H}W_{\sigma }\right) \right) \cong \mathrm{Hom}_{H}\left( V_{%
\mathrm{\mathrm{Res}}\left( \rho _{j_{0}}\right) },W_{\sigma }\right) ,
\end{equation*}%
so that the restriction of $\rho _{j_{0}}$ to $H$ contains the $H$%
-representation $\left[ \sigma \right] $. Now, for a component $X_{\alpha
}^{H}$ of $X^{H}$, with $X_{\alpha }=GX_{\alpha }^{H}$ its component in $X$
relative to $G$, the trivial bundle%
\begin{equation*}
X_{\alpha }\times V_{\rho _{j_{0}}}
\end{equation*}%
is a $G$-bundle (with diagonal action) that contains a nontrivial fine
component $W_{\alpha ,\left[ \sigma \right] }$ containing $X_{\alpha
}^{H}\times \left( V_{\rho _{j_{0}}}\right) ^{\left[ \sigma \right] }$.

\begin{definition}
\label{canonicalIsotropyBundleDefinition}We call $W_{\alpha ,\left[ \sigma %
\right] }\rightarrow X_{\alpha }$ the \textbf{canonical isotropy }$G$\textbf{%
-bundle associated to }$\left( \alpha ,\left[ \sigma \right] \right) \in \pi
_{0}\left( X^{H}\right) \times \widehat{H}$. Observe that $W_{\alpha ,\left[
\sigma \right] }$ depends only on the enumeration of irreducible
representations of $G$, the irreducible $H$-representation $\left[ \sigma %
\right] $ and the component $X_{\alpha }^{H}$. We also denote the following
positive integers associated to $W_{\alpha ,\left[ \sigma \right] }$:

\begin{itemize}
\item $m_{\alpha ,\left[ \sigma \right] }=\dim \mathrm{Hom}_{H}\left(
W_{\sigma },W_{\alpha ,\left[ \sigma \right] ,x}\right) =\dim \mathrm{Hom}%
_{H}\left( W_{\sigma },V_{\rho _{j_{0}}}\right) $ (the \textbf{associated
multiplicity}), independent of the choice of $\left[ \sigma ,W_{\sigma }%
\right] $ present in $W_{\alpha ,\left[ \sigma \right] ,x}$ , $x\in
X_{\alpha }^{H}$ (see Remark \ref{constantMultiplicityRemark}).

\item $d_{\alpha ,\left[ \sigma \right] }=\dim W_{\sigma }$(the \textbf{%
associated representation dimension}), independent of the choice of $\left[
\sigma ,W_{\sigma }\right] $ present in $W_{\alpha ,\left[ \sigma \right]
,x} $ , $x\in X_{\alpha }^{H}$.

\item $n_{\alpha ,\left[ \sigma \right] }=\frac{\mathrm{rank}\left(
W_{\alpha ,\left[ \sigma \right] }\right) }{m_{\alpha ,\left[ \sigma \right]
}d_{\alpha ,\left[ \sigma \right] }}$ (the \textbf{inequivalence number}),
the number of inequivalent representations present in $W_{\alpha ,\left[
\sigma \right] ,x}$ , $x\in X_{\alpha }^{H}$.
\end{itemize}
\end{definition}

\begin{remark}
Observe that $W_{\alpha ,\left[ \sigma \right] }=W_{\alpha ^{\prime },\left[
\sigma ^{\prime }\right] }$ if $\left[ \sigma ^{\prime }\right] =\left[
\sigma ^{n}\right] $ for some $n\in N$ such that $nX_{\alpha }^{H}=X_{\alpha
^{\prime }}^{H}~$.
\end{remark}

The lemma below follows immediately from Lemma \ref{AdaptedToAnyBundleLemma}.

\begin{lemma}
\label{canIsotropyGbundleAdaptedExists}Given any $G$-bundle $E\rightarrow X$
and any fine component $E^{b}$ of $E$ over some $X_{\alpha }=GX_{\alpha
}^{H} $, there exists a canonical isotropy $G$-bundle $W_{\alpha ,\left[
\sigma \right] }$ adapted to $E^{b}\rightarrow X_{\alpha }$.
\end{lemma}

\subsection{Decomposing sections in tensor products of equivariant bundles}

We now consider a method of decomposing sections of a tensor product of
equivariant bundles. Suppose that $E_{1}\rightarrow X$ is a (possibly
infinite-dimensional) $G$-equivariant vector bundle over a manifold with one
orbit type $\left[ H\right] $. Let $X_{\alpha }=GX_{\alpha }^{H}\subseteq X$%
, and let $E_{1,X_{\alpha }}:=\left. E_{1}\right\vert _{X_{\alpha
}}=\allowbreak \bigoplus\limits_{b}E_{1}^{b}$ be the fine decomposition
(Definition \ref{FineDecompositionDefinition}). Let $W^{b}=W_{\alpha ,\left[
\sigma \right] }\rightarrow X_{\alpha }$ be the canonical isotropy $G$%
-bundle adapted to $E_{1}^{b}$, as in Definition \ref%
{canonicalIsotropyBundleDefinition} and Lemma \ref%
{canIsotropyGbundleAdaptedExists}. Every $\left[ \sigma \right] \in \widehat{%
H}$ that is present in $W_{x}^{b}$ has associated multiplicity $m_{b}$ ,
which is independent of $x\in X_{\alpha }^{H}$. Let $W_{x}^{b}=\bigoplus%
\limits_{j}W_{j,x}^{b}$ be the isotypical decomposition, so that each $j$
refers to a distinct irreducible $H$-representation type $\left[ \sigma
_{b,j}\right] $, each with dimension $m_{b}\cdot d_{b}$, where $d_{b}$ is
the dimension of that irreducible representation. Let $E_{2}\rightarrow X$
be another $G$-equivariant vector bundle over $X$. Restricted to a single
orbit $\mathcal{O}_{x}=Gx$ in $X_{\alpha }$, let $E_{2,\mathcal{O}%
_{x}}=\left. E_{2}\right\vert _{\mathcal{O}_{x}}$. Then Frobenius
reciprocity implies%
\begin{eqnarray*}
&&\mathrm{Hom}_{G}\left( V_{\rho },\Gamma \left( \mathcal{O}_{x},E_{1,%
\mathcal{O}_{x}}^{b}\otimes E_{2,\mathcal{O}_{x}}\right) \right) \otimes
\bigoplus\limits_{j}\mathrm{Hom}_{H}\left( W_{j,x}^{b},W_{j,x}^{b}\right) \\
&\cong &\mathrm{Hom}_{H}\left( V_{\mathrm{Res}\left( \rho \right)
},E_{1,x}^{b}\otimes E_{2,x}\right) \otimes \bigoplus\limits_{j}\mathrm{Hom}%
_{H}\left( W_{j,x}^{b},W_{j,x}^{b}\right) \\
&\cong &\bigoplus\limits_{j}\mathrm{Hom}_{H}\left( V_{\mathrm{Res}\left(
\rho \right) },\mathrm{Hom}_{H}\left( W_{j,x}^{b},E_{1,x}^{b}\right) \otimes
W_{j,x}^{b}\otimes E_{2,x}\right) ,
\end{eqnarray*}%
Since $\mathrm{Hom}_{H}\left( W_{j,x}^{b},E_{1,x}^{b}\right) $ is a trivial $%
H$-space, 
\begin{eqnarray*}
&&\mathrm{Hom}_{G}\left( V_{\rho },\Gamma \left( \mathcal{O}_{x},E_{1,%
\mathcal{O}_{x}}^{b}\otimes E_{2,\mathcal{O}_{x}}\right) \right) \otimes
\bigoplus\limits_{j}\mathrm{Hom}_{H}\left( W_{j,x}^{b},W_{j,x}^{b}\right) \\
&\cong &\bigoplus\limits_{j}\mathrm{Hom}_{H}\left(
W_{j,x}^{b},E_{1,x}^{b}\right) \otimes \mathrm{Hom}_{H}\left( V_{\mathrm{Res}%
\left( \rho \right) },W_{j,x}^{b}\otimes E_{2,x}\right)
\end{eqnarray*}%
Then%
\begin{eqnarray*}
&&\Gamma \left( \mathcal{O}_{x},E_{1,\mathcal{O}_{x}}^{b}\otimes E_{2,%
\mathcal{O}_{x}}\right) ^{\rho }\otimes \bigoplus\limits_{j}\mathrm{Hom}%
_{H}\left( W_{j,x}^{b},W_{j,x}^{b}\right) \\
&\cong &\mathrm{Hom}_{G}\left( V_{\rho },\Gamma \left( \mathcal{O}_{x},E_{1,%
\mathcal{O}_{x}}^{b}\otimes E_{2,\mathcal{O}_{x}}\right) \right) \otimes
V_{\rho }\otimes \bigoplus\limits_{j}\mathrm{Hom}_{H}\left(
W_{j,x}^{b},W_{j,x}^{b}\right) \\
&\cong &\bigoplus\limits_{j}\mathrm{Hom}_{H}\left(
W_{j,x}^{b},E_{1,x}^{b}\right) \otimes \mathrm{Hom}_{H}\left( V_{\mathrm{Res}%
\left( \rho \right) },W_{j,x}^{b}\otimes E_{2,x}\right) \otimes V_{\rho } \\
&\cong &\bigoplus\limits_{j}\mathrm{Hom}_{H}\left(
W_{j,x}^{b},E_{1,x}^{b}\right) \otimes \Gamma \left( \mathcal{O}_{x},%
\widetilde{W}_{j}^{b}\otimes E_{2}\right) ^{\rho },
\end{eqnarray*}%
where $\widetilde{W}_{j}^{b}\rightarrow \mathcal{O}_{x}$ is the bundle $%
G\times _{H}W_{j,x}^{b}\subseteq \left. W^{b}\right\vert _{\mathcal{O}_{x}}$
with the identification $x=eH$. Observe that the isomorphism above is an
isomorphism of $G$-modules, where $G$ acts trivially on $\mathrm{Hom}%
_{H}\left( W_{j,x}^{b},W_{j,x}^{b}\right) $ and on $\mathrm{Hom}_{H}\left(
W_{j,x}^{b},E_{1,x}^{b}\right) $. Let $L:E_{1}\rightarrow E_{1}$ be a $G$%
-equivariant bundle map, and let $B:\Gamma \left( X,E_{2}\right) \rightarrow
\Gamma \left( X,E_{2}\right) $ be a $G$-equivariant operator. Let $%
L^{b}:E_{1}^{b}\rightarrow E_{1}^{b}$ be the restriction. We note that $%
L^{b} $ acts on $E_{1,x}^{b}\cong \bigoplus\limits_{j}\mathrm{Hom}_{H}\left(
W_{j,x}^{b},E_{1,x}^{b}\right) \otimes W_{j,x}^{b}$ by $L^{\prime ~b}\otimes 
\mathbf{1}$, where $L^{\prime ~b}:\mathrm{Hom}_{H}\left(
W_{j,x}^{b},E_{1,x}^{b}\right) \rightarrow \mathrm{Hom}_{H}\left(
W_{j,x}^{b},E_{1,x}^{b}\right) $ acts by post-composing with $L^{b}$.

The following Lemma is a consequence of the derivation above and the slice
theorem.

\begin{lemma}
\label{SectionSplittingLemma}Let $T_{\varepsilon }\rightarrow \mathcal{O}%
_{x} $ be a tubular neighborhood of $\mathcal{O}_{x}\subset X$ such that $%
T_{\varepsilon }\cong G\times _{H}D_{\varepsilon }$, where $x\in
D_{\varepsilon }\subset X^{H}$ is a ball transverse to $\mathcal{O}_{x}$.
Then the operator $L\otimes B\otimes \mathbf{1}$ corresponds to $L^{\prime
~b}\otimes \left( \mathbf{1}\otimes B\right) ^{\rho }$ through the
isomorphism 
\begin{eqnarray*}
&&\Gamma \left( T_{\varepsilon },E_{1}^{b}\otimes E_{2}\right) ^{\rho
}\otimes \bigoplus\limits_{j}\mathrm{Hom}_{H}\left( \left.
W_{j}^{b}\right\vert _{D_{\varepsilon }},\left. W_{j}^{b}\right\vert
_{D_{\varepsilon }}\right) \\
&\cong &\bigoplus\limits_{j}\mathrm{Hom}_{H}\left( \left.
W_{j}^{b}\right\vert _{D_{\varepsilon }},\left. E_{1}^{b}\right\vert
_{D_{\varepsilon }}\right) \otimes \Gamma \left( T_{\varepsilon },\widetilde{%
W}_{j}^{b}\otimes E_{2}\right) ^{\rho },
\end{eqnarray*}%
with the isomorphism given by evaluation. If $\widehat{\otimes }$ denotes
the graded tensor product, then%
\begin{eqnarray*}
&&\Gamma \left( T_{\varepsilon },E_{1}^{b}\widehat{\otimes }E_{2}\right)
^{\rho }\otimes \bigoplus\limits_{j}\mathrm{Hom}_{H}\left( \left.
W_{j}^{b}\right\vert _{D_{\varepsilon }},\left. W_{j}^{b}\right\vert
_{D_{\varepsilon }}\right) \\
&\cong &\bigoplus\limits_{j}\mathrm{Hom}_{H}\left( \left.
W_{j}^{b}\right\vert _{D_{\varepsilon }},\left. E_{1}^{b}\right\vert
_{D_{\varepsilon }}\right) \widehat{\otimes }\Gamma \left( T_{\varepsilon },%
\widetilde{W}_{j}^{b}\otimes E_{2}\right) ^{\rho }.
\end{eqnarray*}%
In the above, $W^{b}=\bigoplus\limits_{j}W_{j}^{b}$ is the isotypical
decomposition.
\end{lemma}

\section{Properties of the Equivariant Heat Kernel and Equivariant Index 
\label{heatKernelEquivIndexSection}}

\subsection{The equivariant heat kernel and index\label{equivariant index
section}}

\medskip We now review some properties of the equivariant index and
equivariant heat kernel that are known to experts in the field but are not
written in the literature in the form and generality required in this paper
(see \cite{A}, \cite{BrH1}, \cite{BrH2}, \cite{Be-G-V}). With notation as in
the introduction, let $E=E^{+}\oplus E^{-}$ be a graded, $G$-equivariant
vector bundle over $M$. We consider a first order $G$-equivariant
differential operator $D^{+}:$ $\Gamma \left( M,E^{+}\right) \rightarrow
\Gamma \left( M,E^{-}\right) $ which is transversally elliptic, and let $%
D^{-}$ be the formal adjoint of $D^{+}$. The restriction $D^{\pm ,\rho
}=\left. D^{\pm }\right\vert _{\Gamma \left( M,E\right) ^{\rho }}$ behaves
in a similar way to an elliptic operator. Let $C:\Gamma \left( M,E\right)
\rightarrow \Gamma \left( M,E\right) $ be the Casimir operator described in
the introduction, and let $\lambda _{\rho }$ be the eigenvalue of $C$
associated to the representation type $\left[ \rho \right] $. The following
argument can be seen in some form in \cite{A}. Given a section $\alpha \in
\Gamma \left( M,E^{+}\right) ^{\rho }$, we have 
\begin{equation*}
D^{-}D^{+}\alpha =\left( D^{-}D^{+}+C-\lambda _{\rho }\right) \alpha .
\end{equation*}%
Then $D^{-}D^{+}+C-\lambda _{\rho }$ is self-adjoint and elliptic and has
finite dimensional eigenspaces consisting of smooth sections. Thus, the
eigenspaces of $D^{-}D^{+}$ restricted to $\Gamma \left( M,E^{+}\right)
^{\rho }$ are finite dimensional and consist of smooth sections, and the
index $\mathrm{ind}^{\rho }\left( D\right) $ is well-defined. Further, the $%
\left[ \rho \right] $-part $K^{\left[ \rho \right] }$ of the heat kernel of $%
e^{-tD^{-}D^{+}}$ is the same as the $\left[ \rho \right] $-part of the heat
kernel $K\left( t,\cdot ,\cdot \right) $ of $e^{-t\left(
D^{-}D^{+}+C-\lambda _{\rho }\right) }$. Let 
\begin{equation*}
\beta _{1}\leq \beta _{2}\leq ...
\end{equation*}%
be the eigenvalues of $D^{-}D^{+}+C-\lambda _{\rho }$ repeated according to
multiplicities, which correspond to the $L^{2}$ orthonormal set of
eigensections $\left\{ \alpha _{1},\alpha _{2},...\right\} $. We may choose
that basis so that each $\alpha _{j}$ belongs to a specific irreducible
representation space of $G$. The kernel $K^{+}$ of $e^{-t\left(
D^{-}D^{+}+C-\lambda _{\rho }\right) }$ satisfies 
\begin{equation*}
K^{+}\left( t,x,y\right) =\sum_{k}e^{-t\beta _{k}}\alpha _{k}\left( x\right)
\otimes \left( \alpha _{k}\left( y\right) \right) ^{\ast }\in \mathrm{Hom}%
\left( E_{y}^{+},E_{x}^{+}\right) .
\end{equation*}%
Observe that for all $g\in G$, 
\begin{eqnarray*}
g\cdot K^{+}\left( t,g^{-1}x,x\right) &=&\sum_{k}e^{-t\beta _{k}}g\cdot
\alpha _{k}\left( g^{-1}x\right) \otimes \left( \alpha _{k}\left( x\right)
\right) ^{\ast } \\
&=&\sum_{k}e^{-t\beta _{k}}\sum_{\left[ \eta \right] }\eta \left( g\right)
\left( P_{\eta }\alpha _{k}\right) \left( x\right) \otimes \left( \alpha
_{k}\left( x\right) \right) ^{\ast },
\end{eqnarray*}%
where the interior sum is over all equivalence classes of irreducible
unitary representations $\eta $ of $G$, and $P_{\eta }$ is the projection
onto the $\left[ \eta \right] $-component. By the way we have chosen the
basis $\left\{ \alpha _{1},\alpha _{2},...\right\} $, each $P_{\eta }\alpha
_{k}$ is either $\alpha _{k}$ or $0$, so in fact $\left( P_{\eta }\alpha
_{k}\right) \left( x\right) \otimes \left( \alpha _{k}\left( x\right)
\right) ^{\ast }=\left( P_{\eta }\alpha _{k}\right) \left( x\right) \otimes
\left( P_{\eta }\alpha _{k}\left( x\right) \right) ^{\ast }$. Next, let $%
\rho $ be a particular irreducible representation, and let $\chi _{\rho }$
be its character. By Lemma \ref{ProjectionLemma}, if $dg$ is the
biinvariant, normalized Haar volume form,

\begin{equation*}
\int_{g\in G}\eta \left( g\right) P_{\eta }\overline{\chi _{\rho }\left(
g\right) }~dg=\frac{1}{\dim V_{\rho }}P_{\rho }^{\eta }P_{\eta },
\end{equation*}%
and%
\begin{eqnarray*}
&&\int_{x\in M}\int_{g\in G}\mathrm{tr}\left( \eta \left( g\right) \left(
P_{\eta }\alpha _{k}\right) \left( x\right) \otimes \left( \alpha _{k}\left(
x\right) \right) ^{\ast }\right) ~\overline{\chi _{\rho }\left( g\right) }%
~dg~\left\vert dx\right\vert \\
&=&\int_{x\in M}\int_{g\in G}\mathrm{tr}\left( \eta \left( g\right) \left(
P_{\eta }\alpha _{k}\right) \left( x\right) \otimes \left( P_{\eta }\alpha
_{k}\left( x\right) \right) ^{\ast }\right) ~\overline{\chi _{\rho }\left(
g\right) }~dg~\left\vert dx\right\vert \\
&=&\frac{1}{\dim V_{\rho }}\varepsilon _{k}^{\eta }\delta _{\eta \rho
}\int_{x\in M}\mathrm{tr}\left( \alpha _{k}\left( x\right) \otimes \alpha
_{k}\left( x\right) ^{\ast }\right) ~\left\vert dx\right\vert =\frac{1}{\dim
V_{\rho }}\varepsilon _{k}^{\eta }\delta _{\eta \rho },
\end{eqnarray*}%
where $\varepsilon _{k}^{\eta }$ is $1$ if $\alpha _{k}$ is a section of
type $\eta $ and $0$ otherwise, and where $\delta _{\eta \rho }=1$ if $\eta $
is equivalent to $\rho $ and is zero otherwise. Using this information, we
have that 
\begin{eqnarray*}
\int_{x\in M}\int_{g\in G}\mathrm{tr}\left( g\cdot K^{+}\left(
t,g^{-1}x,x\right) \right) ~\overline{\chi _{\rho }\left( g\right) }%
~dg~\left\vert dx\right\vert &=&\frac{1}{\dim V_{\rho }}\sum_{k}e^{-t\beta
_{k}}\varepsilon _{k}^{\rho } \\
&=&\frac{1}{\dim V_{\rho }}\sum_{k^{\left[ \rho \right] }}e^{-t\beta _{k^{%
\left[ \rho \right] }}},
\end{eqnarray*}%
where $k^{\left[ \rho \right] }$ are the specific positive integers
corresponding to eigensections $\alpha _{k^{\left[ \rho \right] }}$ of type $%
\left[ \rho \right] $. In summary, 
\begin{eqnarray*}
\mathrm{tr}\left( \left. e^{-tD^{+}D^{-}}\right\vert _{\Gamma \left(
M,E^{+}\right) ^{\rho }}\right) &=&\sum_{k^{\left[ \rho \right] }}e^{-t\beta
_{k^{\left[ \rho \right] }}} \\
&=&\dim V_{\rho }\int_{x\in M}\int_{g\in G}\mathrm{tr}\left( g\cdot
K^{+}\left( t,g^{-1}x,x\right) \right) ~\overline{\chi _{\rho }\left(
g\right) }~dg~\left\vert dx\right\vert
\end{eqnarray*}%
Similar arguments apply to the operator $D^{+}D^{-}$. The usual
McKean-Singer argument implies that the supertrace of the heat kernel $K$
corresponding to $D=\left( 
\begin{array}{cc}
0 & D^{-} \\ 
D^{+} & 0%
\end{array}%
\right) $ restricted to $\Gamma \left( M,E^{+}\oplus E^{-}\right) ^{\rho }$
is the same as the index $\mathrm{ind}^{\rho }\left( D\right) $. By the
argument above, 
\begin{equation}
\mathrm{ind}^{\rho }\left( D\right) =\int_{x\in M}\int_{g\in G}\,\mathrm{str~%
}g\cdot K\left( t,g^{-1}x,x\right) ~\overline{\chi _{\rho }\left( g\right) }%
~dg~\left\vert dx\right\vert ,  \label{heatKernelExpressionForIndex}
\end{equation}%
where $K$ is the heat kernel of $e^{-t\left( D^{2}+C-\lambda _{\rho }\right)
}$. Since the heat kernel $K$ changes smoothly with respect to $G$%
-equivariant deformations of the metric and of the operator $D$ and the
right hand side is an integer, we see that $\mathrm{ind}^{\rho }\left(
D\right) $ is stable under such homotopies of the operator $D^{+}$ through $%
G $-equivariant transversally elliptic operators. This implies that the
indices $\mathrm{ind}^{G}\left( D\right) $ and $\mathrm{ind}_{g}\left(
D\right) $ mentioned in the introduction depend only on the $G$-equivariant
homotopy class of the principal transverse symbol of $D^{+}$. (We knew this
already because of the Fredholm properties discussed in Section \ref%
{TransEllipOpsSection}.)

Note also that $e^{-t\left( D^{2}+C-\lambda _{\rho }\right) }$ differs from $%
e^{-t\left( D^{2}+C\right) }$ by a factor of $e^{-\lambda _{\rho }t}$, so we
also have that 
\begin{equation*}
\mathrm{ind}^{\rho }\left( D\right) =\lim_{t\rightarrow 0}\int_{x\in
M}\int_{g\in G}\,\mathrm{str~}g\cdot K^{0}\left( t,g^{-1}x,x\right) ~%
\overline{\chi _{\rho }\left( g\right) }~dg~\left\vert dx\right\vert \,,
\end{equation*}%
where $K^{0}$ is the heat kernel of $e^{-t\left( D^{2}+C\right) }$.

\subsection{The asymptotic expansion of the supertrace of the equivariant
heat kernel}

We now use Lemma \ref{decomposition} to decompose $M=%
\bigcup_{i=0}^{r}M_{i}^{\varepsilon }$ into a disjoint union of pieces so
that the integral (\ref{heatKernelExpressionForIndex}) can be simplified.
Also, observe that the kernel $K$ and the metric on $M$ are smooth functions
of the operator $D^{+}$ and the metric on $M$. As we have mentioned, the
index $\mathrm{ind}^{\rho }\left( D\right) $ is invariant under smooth, $G$%
-equivariant perturbations of $D^{+}$ and the metric. In particular, on each 
$M_{i}^{\varepsilon }$ for $i>0$ (ie not including $M_{0}$, the union of the
principal orbits), in normal directions to the singular stratum $M_{i}$, we
may flatten the metric, trivialize the bundles $E^{\pm }$, and make the
operator $D$ have constant coefficients with respect to the trivialized
normal coordinates. Hence, 
\begin{gather}
\mathrm{ind}^{\rho }\left( D\right) =\sum_{i}\int_{w\in M_{i}^{\varepsilon
}}\int_{g\in G}\mathrm{\mathrm{str}\,}\left( g\cdot K\left(
t,g^{-1}w,w\right) \right) ~\overline{\chi _{\rho }\left( g\right) }%
~dg~\left\vert dw\right\vert  \notag \\
=\sum_{i}\int_{\left( x,z_{x}\right) \in M_{i}^{\varepsilon }}\int_{g\in G}%
\mathrm{\mathrm{str}\,}\left( g\cdot K\left( t,g^{-1}\left( x,z_{x}\right)
,\left( x,z_{x}\right) \right) \right) \,\overline{\chi _{\rho }\left(
g\right) }~dg\,\,\left\vert dz_{x}\right\vert ~\left\vert dx\right\vert _{i},
\label{heatKernelDecompStrata}
\end{gather}%
where we choose coordinates $w=\left( x,z_{x}\right) $ for the point $\exp
_{x}^{\bot }\left( z_{x}\right) $, where $x\in M_{i}$, $z_{x}\in
B_{\varepsilon _{i}}=B_{2^{i}\varepsilon }\left( N_{x}\left( M_{i}\right)
\right) $, and $\exp _{x}^{\bot }$ is the normal exponential map. In the
expression above, $\left\vert dz_{x}\right\vert $ is the Euclidean density
on $N_{x}\left( M_{i}\right) $, and $\left\vert dx\right\vert _{i}$ is the
Riemannian density on $M_{i}$. We note that the outer integral is not a true
product integral if $M_{i}$ is not the lowest stratum, because for the most
part $x\in \Sigma _{i}^{\varepsilon }$ and $z_{x}\in B_{\varepsilon
_{i}}=B_{2^{i}\varepsilon }\left( N_{x}\left( \Sigma _{i}^{\varepsilon
}\right) \right) $, but to include all of $M_{i}^{\varepsilon }$, some $x\in
M_{i}\setminus \Sigma _{i}^{\varepsilon }$ and $z_{x}$ away from the zero
section may be needed; metrically, this is the integral over a cylinder of
small radius whose ends are concave spherical (of larger radius) instead of
flat. Recall that $K\left( t,g^{-1}\left( x,z_{x}\right) ,\left(
x,z_{x}\right) \right) $ is a map from $E_{\left( x,z_{x}\right) }=E_{x}$ to 
$E_{g^{-1}\left( x,z_{x}\right) }=E_{g^{-1}x}$.

One important idea is that the asymptotics of $K\left( t,g^{-1}w,w\right) $
as $t\rightarrow 0$ are completely determined by the operator's local
expression along the minimal geodesic connecting $g^{-1}w$ and $w$, if $%
g^{-1}w$ and $w$ are sufficiently close together. If the distance between $%
g^{-1}w$ and $w$ is bounded away from zero, there is a constant $c>0$ such
that $K\left( t,g^{-1}w,w\right) =\mathcal{O}\left( e^{-c/t}\right) $ as $%
t\rightarrow 0$. For these reasons, it is clear that the asymptotics of the
supertrace $\mathrm{ind}^{\rho }\left( D\right) $ are locally determined
over spaces of orbits in $M$, meaning that the contribution to the index is
a sum of $t^{0}$-asymptotics of integrals of $\int_{U}\int_{G}\mathrm{%
\mathrm{str}}\left( g\cdot K\left( t,g^{-1}w,w\right) \right) \,\overline{%
\chi _{\rho }\left( g\right) }~dg$ $\left\vert dw\right\vert $ over a finite
collection of saturated sets $U\subset M$ that are unions of orbits that
intersect a neighborhood of a point of $M$. In particular, the index $%
\mathrm{ind}^{\rho }\left( D\right) $ is the sum of the $t^{0}$ asymptotics
of the integrals over $M_{i}^{\varepsilon }$. However, the integral of the $%
t^{0}$ asymptotic coefficient of $\int_{G}\mathrm{\mathrm{str}}\left( g\cdot
K\left( t,g^{-1}w,w\right) \right) \overline{\chi _{\rho }\left( g\right) }%
~dg$ over all of $M$ is not the index, and the integral of the $t^{0}$
asymptotic coefficient of $\int_{M}\mathrm{\mathrm{str}\,}\left( g\cdot
K\left( t,g^{-1}w,w\right) \right) ~\mathrm{dvol}_{M}$ over all of $G$ is
not the index. Thus, the integrals over $M$ and $G$ may not be separated
when computing the local index contributions. In particular, the singular
strata may not be ignored.

\subsection{Heat kernel equivariance property}

With notation as in the previous sections, the kernel $K^{+}$ of $%
e^{-t\left( D^{-}D^{+}+C-\lambda _{\rho }\right) }$ satisfies 
\begin{equation*}
K^{+}\left( t,x,y\right) =\sum_{k}e^{-t\beta _{k}}\alpha _{k}\left( x\right)
\otimes \left( \alpha _{k}\left( y\right) \right) ^{\ast }\in \mathrm{Hom}%
\left( E_{y}^{+},E_{x}^{+}\right) ,
\end{equation*}%
with $\left\{ \alpha _{j}\in \Gamma \left( E^{+}\right) \right\} $ is an $%
L^{2}$-orthogonal basis of eigensections of $\Delta =D^{-}D^{+}+C-\lambda
_{\rho }$ corresponding to eigenvalues $\beta _{j}\geq 0$ (counted with
multiplicity).

If $g\cdot $ denotes the action of $g\in G$ on sections of $E$, note that $%
\left( g\cdot u\right) \left( x\right) =g\cdot u\left( g^{-1}x\right) $. The 
$G$-equivariance implies%
\begin{equation*}
g\cdot \left( \Delta \alpha _{j}\right) =\left( \Delta \left( g\cdot \alpha
_{j}\right) \right) ,
\end{equation*}%
or%
\begin{equation*}
\beta _{j}\left( g\cdot \alpha _{j}\right) =\Delta \left( g\cdot \alpha
_{j}\right) ,
\end{equation*}%
so that $\left\{ g\cdot \alpha _{j}\right\} $ is another $L^{2}$-orthonormal
basis of eigensections, and 
\begin{eqnarray}
K^{+}\left( t,x,y\right) &=&\sum e^{-t\beta _{k}}g\cdot \alpha _{k}\left(
g^{-1}x\right) \otimes \left[ g\cdot \alpha _{k}\left( g^{-1}y\right) \right]
^{\ast }  \notag \\
&=&\sum e^{-t\beta _{k}}g\cdot \alpha _{k}\left( g^{-1}x\right) \otimes 
\left[ \alpha _{k}\left( g^{-1}y\right) \right] ^{\ast }\circ g^{-1}\cdot 
\notag \\
&=&g\cdot K^{+}\left( t,g^{-1}x,g^{-1}y\right) \circ g^{-1}\cdot ~\in 
\mathrm{Hom}\left( E_{y},E_{x}\right) .  \label{heatKernelInvariance}
\end{eqnarray}%
Note that $\mathrm{Hom}\left( E,E\right) $ is a $G$-bundle, and it carries
the left $G$-action defined on $L_{x}\in \mathrm{Hom}\left(
E_{x},E_{x}\right) =E_{x}\otimes E_{x}^{\ast }$ by%
\begin{equation*}
g\bullet L_{x}=g\cdot L_{x}\circ g^{-1}\cdot ~\in \mathrm{Hom}\left(
E_{gx},E_{gx}\right) ,
\end{equation*}%
or%
\begin{equation*}
g\bullet \left( v\otimes w^{\ast }\right) =\left( g\cdot v\right) \otimes
\left( w^{\ast }\circ g^{-1}\cdot \right) .
\end{equation*}%
With this action,%
\begin{equation*}
g\bullet \left( K^{+}\left( t,g^{-1}x,g^{-1}x\right) \right) =K^{+}\left(
t,x,x\right) .
\end{equation*}%
Note $g\cdot K^{+}\left( t,g^{-1}x,x\right) \in \mathrm{Hom}\left(
E_{x},E_{x}\right) $. Observe that if we let $L_{g,x}=g\cdot K^{+}\left(
t,g^{-1}x,x\right) \in \mathrm{Hom}\left( E_{x},E_{x}\right) $, then 
\begin{eqnarray*}
g^{\prime }\bullet L_{g,g^{\prime -1}x} &=&g^{\prime }\cdot \left( g\cdot
K^{+}\left( t,g^{-1}\left( g^{\prime -1}x\right) ,g^{\prime -1}x\right)
\right) \circ g^{\prime -1} \\
&=&g^{\prime }\cdot g\cdot g^{\prime -1}K^{+}\left( t,g^{\prime
}g^{-1}\left( g^{\prime -1}x\right) ,g^{\prime }g^{\prime -1}x\right) \circ
g^{\prime }\circ g^{\prime -1}
\end{eqnarray*}%
by (\ref{heatKernelInvariance}), or%
\begin{equation*}
g^{\prime }\bullet L_{g,g^{\prime -1}x}=\left( g^{\prime }gg^{\prime
-1}\right) \cdot K^{+}\left( t,\left( g^{\prime }gg^{\prime -1}\right)
^{-1}x,x\right) =L_{g^{\prime }gg^{\prime -1},x}~.
\end{equation*}%
This implies that%
\begin{eqnarray*}
K^{+}\left( t,x,x\right) ^{\rho } &=&\dim V^{\rho }\int_{G}g\cdot
K^{+}\left( t,g^{-1}x,x\right) ~\overline{\chi _{\rho }\left( g\right) }~dg
\\
&=&\dim V^{\rho }\int_{G}\left( g^{\prime }gg^{\prime -1}\right) \cdot
K^{+}\left( t,\left( g^{\prime }gg^{\prime -1}\right) ^{-1}x,x\right) ~%
\overline{\chi _{\rho }\left( g\right) }~dg,
\end{eqnarray*}%
by changing variables from $g$ to $g^{\prime }gg^{\prime -1}$ for some fixed 
$g^{\prime }\in G$. By the formula above, we obtain%
\begin{eqnarray*}
K^{+}\left( t,x,x\right) ^{\rho } &=&\left( \dim V^{\rho }\right) g^{\prime
}\cdot \int_{G}g\cdot K^{+}\left( t,g^{-1}\left( g^{\prime -1}x\right)
,g^{\prime -1}x\right) ~\overline{\chi _{\rho }\left( g\right) }~dg\circ
g^{\prime -1} \\
&=&g^{\prime }\cdot K^{+}\left( t,g^{\prime -1}x,g^{\prime -1}x\right)
^{\rho }\circ g^{\prime -1}.
\end{eqnarray*}%
Equivalently, as sections of $\mathrm{Hom}\left( E,E\right) $, for all $g\in
G$,%
\begin{equation*}
K^{+}\left( t,gx,gx\right) ^{\rho }=g\bullet K^{+}\left( t,x,x\right) ^{\rho
}=g\cdot K^{+}\left( t,x,x\right) ^{\rho }\circ g^{-1}.
\end{equation*}%
Furthermore, since the calculations above certainly apply to $K^{-}$,%
\begin{eqnarray}
\mathrm{tr}K^{\pm }\left( t,gx,gx\right) ^{\rho } &=&\mathrm{tr}K^{\pm
}\left( t,x,x\right) ^{\rho },  \notag \\
\mathrm{str}K\left( t,gx,gx\right) ^{\rho } &=&\mathrm{str}K\left(
t,x,x\right) ^{\rho }.  \label{traceHeatKernelInvariance}
\end{eqnarray}%
Thus, these traces descend to functions on the orbit space. In order to
evaluate the integral of $K^{\pm }\left( t,x,x\right) ^{\rho }$ over a
neighborhood of a singular stratum $\Sigma $, it is sufficient to evaluate
it over the normal bundle to $\Sigma $ restricted to local sections of $%
\Sigma \rightarrow G\diagdown \Sigma $.

Note that the calculations above may be applied to nonelliptic kernels. In
particular, we may let $D$ be a fiberwise equivariant elliptic operator;
then the calculations above apply to it.

\section{The Equivariant Index of an operator on a sphere\label%
{EquivarIndexOpSphereSection}}

The purpose of this section is to do an explicit calculation of local heat
trace asymptotics of a very specific type of operator; this analysis will
apply to a much more general situation later in the paper. We will relate
the local equivariant supertrace of a heat operator on Euclidean space to
that of the natural pullback of the operator to the radial blowup of the
origin. In what follows, we will make very explicit choices in order to do
the calculation, but the final result will be independent of those choices.

\subsection{Polar coordinate form of constant coefficient operator on $%
\mathbb{R}^{k}$\label{Euclidean operator}}

In this section, we consider the following situation. Let $Q:C^{\infty
}\left( \mathbb{R}^{k},\mathbb{C}^{d}\right) \rightarrow C^{\infty }\left( 
\mathbb{R}^{k},\mathbb{C}^{d}\right) $ be an elliptic, first order, constant
coefficient differential operator, where $k\geq 2$. Let 
\begin{equation*}
Q_{1}=\sum_{j=1}^{k}A_{j}\partial _{j},
\end{equation*}%
where by assumption $A_{j}\in GL\left( d,\mathbb{C}\right) $, and for any $%
c=\left( c_{1},...,c_{k}\right) \in \mathbb{R}^{k}$, $\det \left( \sum
c_{j}A_{j}\right) =0$ implies $c=0$. Note that 
\begin{equation*}
Q_{1}^{\ast }=-\sum_{j=1}^{k}A_{j}^{\ast }\partial _{j}
\end{equation*}%
We form the symmetric operator $Q:C^{\infty }\left( \mathbb{R}^{k},\mathbb{C}%
^{d}\oplus \mathbb{C}^{d}\right) \rightarrow C^{\infty }\left( \mathbb{R}%
^{k},\mathbb{C}^{d}\oplus \mathbb{C}^{d}\right) $ 
\begin{equation*}
Q=\left( 
\begin{array}{cc}
0 & Q_{1}^{\ast } \\ 
Q_{1} & 0%
\end{array}%
\right) =\sum_{j=1}^{k}\left( 
\begin{array}{cc}
0 & -A_{j}^{\ast } \\ 
A_{j} & 0%
\end{array}%
\right) \partial _{j}.
\end{equation*}

We write $Q$ in polar coordinates $\left( r,\theta \right) \in \left(
0,\infty \right) \times S^{k-1}$, with $x=r\theta =r\left( \theta
_{1},...,\theta _{k}\right) $, as in Section \ref{SphereAppendix} 
\begin{eqnarray*}
Q_{1} &=&Z^{+}\left( \partial _{r}+\frac{1}{r}Q^{S+}\right) ,\text{ so that}
\\
Q &=&\left( 
\begin{array}{cc}
0 & \left( \partial _{r}+\frac{1}{r}Q^{S+}\right) ^{\ast }Z^{+\ast } \\ 
Z^{+}\left( \partial _{r}+\frac{1}{r}Q^{S+}\right) & 0%
\end{array}%
\right) .
\end{eqnarray*}%
The purpose of this section is to prove the formula%
\begin{equation*}
\int_{B_{\varepsilon }\left( 0\right) }\alpha ^{\beta }=\int_{B_{\varepsilon
}\left( \partial \frac{1}{2}DS^{k}\right) }\widetilde{\alpha }^{\beta }+%
\frac{-\eta \left( Q^{S+,\beta }\right) +h\left( Q^{S+,\beta }\right) }{2},
\end{equation*}%
where $\alpha ^{\beta }=\mathrm{str}K_{Q}\left( t,x,x\right) ^{\beta }$ is
the supertrace of the equivariant heat kernel of $Q^{2}$ restricted to
sections of type $\beta $ ($\beta $ being an irreducible representation of a
group that commutes with $Q$), and where $B_{\varepsilon }\left( 0\right) $
is the ball of radius $\varepsilon $ around the origin of $\mathbb{R}^{k}$,
and where $\widetilde{\alpha }^{\beta }$ and $B_{\varepsilon }\left(
\partial \frac{1}{2}DS^{k}\right) $ correspond to the same neighborhood
after a radial blowup of the origin confined to a smaller ball.

In what follows, we will require another result. Assume $Z=-Z^{\ast }$, and $%
Q^{S}=Q^{S\ast }$ on the sphere, so that $Q^{S}Z+ZQ^{S}=\left( k-1\right) Z$%
. (We note that the second assumption is satisfied automatically; see
Section \ref{eigenvaluesDSsection}.) Suppose that $q$ is a smooth, positive,
real-valued function on a subinterval $I$ of $\left( 0,\infty \right) $, and
let $\widetilde{Q}:C^{\infty }\left( I\times S^{k-1},\mathbb{C}^{d}\right)
\rightarrow C^{\infty }\left( I\times S^{k-1},\mathbb{C}^{d}\right) $ be
defined by%
\begin{equation*}
\widetilde{Q}=Z\left( \partial _{r}+\frac{1}{q\left( r\right) }Q^{S}\right) .
\end{equation*}%
Let%
\begin{equation*}
ds^{2}=\left[ g\left( r\right) \right] ^{2}d\theta ^{2}+dr^{2}
\end{equation*}%
be the metric on $I\times S^{k-1}$, where $g$ is a smooth positive function.
The formal adjoint of $\partial _{r}$ with respect to this metric is then%
\begin{equation*}
\partial _{r}^{\ast }=-\partial _{r}-\frac{\left( k-1\right) g^{\prime
}\left( r\right) }{g\left( r\right) },
\end{equation*}%
so that%
\begin{eqnarray}
\widetilde{Q}^{\ast } &=&\left( -\partial _{r}-\frac{\left( k-1\right)
g^{\prime }\left( r\right) }{g\left( r\right) }+\frac{1}{q\left( r\right) }%
Q^{S}\right) Z^{\ast }  \notag \\
&=&Z\left( \partial _{r}+\frac{\left( k-1\right) g^{\prime }\left( r\right) 
}{g\left( r\right) }+\frac{1}{q\left( r\right) }Q^{S}-\frac{k-1}{q\left(
r\right) }\right)  \notag \\
&=&\widetilde{Q}+\left[ \frac{\left( k-1\right) g^{\prime }\left( r\right) }{%
g\left( r\right) }-\frac{k-1}{q\left( r\right) }\right] Z.  \label{adjoint2}
\end{eqnarray}

\subsection{An induced operator on a sphere\label%
{inducedOperatorSphereSection}}

We now consider a new operator induced from the operator in Section \ref%
{Euclidean operator} as a symmetric operator on $\mathbb{C}^{d}\oplus 
\mathbb{C}^{d}$-valued functions on $S^{k}\subset \mathbb{R}^{k+1}$; in
particular, the metric and operator will be chosen to agree exactly with the
original operator near the south pole. Suppose that $H<O\left( k\right) $
acts on $\mathbb{R}^{k+1}$ in the first $k$ coordinates and isometrically on
the fibers $\mathbb{C}^{d}$and commutes with $Q$ defined in (\ref{D operator
polar coords}). Thus $H$ represents on each copy of $\mathbb{C}^{d}$ as a
unitary representation. Further, every element of $H$ commutes with $Z$ and
with $Q^{s}$. Using these operators, we will define a new operator on a rank 
$2d$ bundle $E=E^{+}\oplus E^{-}$ over $S^{k}\subset \mathbb{R}^{k+1}$ to
obtain an $H$-equivariant elliptic operator on $\Gamma \left(
S^{k},E^{+}\oplus E^{-}\right) $ . In spherical coordinates $\left( \phi
,\theta \right) \in \left[ 0,\pi \right] \times S^{k-1}$, with $\phi $
corresponding to the (spherical) distance from the south pole $x_{k+1}=-1$,
the operator we consider is defined on the lower hemisphere $S^{k,s}$ to be 
\begin{equation*}
\mathbf{Q}=\left( 
\begin{array}{cc}
0 & \left( \partial _{\phi }+\frac{1}{f\left( \phi \right) }Q^{S+}\right)
^{\ast }Z^{-} \\ 
Z^{+}\left( \partial _{\phi }+\frac{1}{f\left( \phi \right) }Q^{S+}\right) & 
0%
\end{array}%
\right)
\end{equation*}%
mapping $C^{\infty }\left( S^{k,s},\mathbb{C}^{d}\oplus \mathbb{C}%
^{d}\right) $ to itself, where $f\left( \phi \right) $ is a function that is
smooth and positive on $\left( 0,\pi \right) $ and such that $f\left( \phi
\right) =\phi $ for $\phi $ near zero and $f\left( \phi \right) =1$ near $%
\phi =\frac{\pi }{2}$, so that the operator exactly agrees with the
Euclidean operator $Q$ near the south pole and is a cylindrical operator
near the equator. We choose the manifold metric $ds^{2}=\left[ g\left( \phi
\right) \right] ^{2}d\theta ^{2}+d\phi ^{2}$ near the south pole and equator
to agree with the standard Euclidean metric and cylindrical metric, so that $%
g\left( \phi \right) =\phi $ for $\phi $ near zero and $g\left( \phi \right)
=1$ for $\phi $ near $\frac{\pi }{2}$. By Equation (\ref{adjoint2}), 
\begin{equation*}
\left( \partial _{\phi }+\frac{1}{f\left( \phi \right) }Q^{S+}\right) ^{\ast
}=-\partial _{\phi }-\frac{\left( k-1\right) g^{\prime }\left( \phi \right) 
}{g\left( \phi \right) }+\frac{1}{f\left( \phi \right) }Q^{S+},
\end{equation*}

Again, we must choose the bundle metric so that $Z$ is a unitary
transformation, so in fact $Z^{+}$ and $Z^{-}=\left( Z^{+}\right) ^{\ast }$
are unitary as maps from $\mathbb{C}^{d}$ to itself. We now use the approach
used in the index theory for manifolds with boundary to double the operator
(see, for example, \cite{Bo-Wo}). On the upper hemisphere, we use the same
coordinates as on the reflection in the lower hemisphere so that the
orientation is reversed, and we identify the opposite parts of the bundle $%
\mathbb{C}^{d}\oplus \mathbb{C}^{d}$ via the map $Z$ over the cylinder. The
operator \textbf{$Q$} has the same description in the upper hemisphere as in
the lower hemisphere. That is, let $\left( \theta _{s},\phi _{s}\right) $ be
the coordinates in the lower hemisphere, and let $\left( \theta _{n},\phi
_{n}\right) $ be coordinates in the upper hemisphere. Near the equator, the
transition function is 
\begin{equation*}
\theta _{n}=\theta _{s};~\phi _{n}=\pi -\phi _{s}.
\end{equation*}%
The bundles $\mathbb{C}^{d}\oplus \mathbb{C}^{d}$ over the upper and lower
hemispheres are identified as follows. A section $\left( L_{+}\left( \theta
_{s},\phi _{s}\right) ,L_{-}\left( \theta _{s},\phi _{s}\right) \right) $ in
the lower hemisphere is identified with the section \linebreak $\left(
U_{+}\left( \theta _{n},\phi _{n}\right) ,U_{-}\left( \theta _{n},\phi
_{n}\right) \right) $ in the upper hemisphere near the equator if 
\begin{eqnarray*}
Z^{+}\left( \theta _{s}\right) ^{-1}L_{-}\left( \theta _{s},\phi _{s}\right)
&=&U_{+}\left( \theta _{s},\pi -\phi _{s}\right) \text{ and} \\
Z^{-}\left( \theta _{s}\right) ^{-1}L_{+}\left( \theta _{s},\phi _{s}\right)
&=&U_{-}\left( \theta _{s},\pi -\phi _{s}\right) ,\text{or} \\
\left( 
\begin{array}{cc}
0 & Z^{+}\left( \theta _{s}\right) ^{-1} \\ 
Z^{-}\left( \theta _{s}\right) ^{-1} & 0%
\end{array}%
\right) \left( 
\begin{array}{c}
L_{+}\left( \theta _{s},\phi _{s}\right) \\ 
L_{-}\left( \theta _{s},\phi _{s}\right)%
\end{array}%
\right) &=&\left( 
\begin{array}{c}
U_{+}\left( \theta _{s},\pi -\phi _{s}\right) \\ 
U_{-}\left( \theta _{s},\pi -\phi _{s}\right)%
\end{array}%
\right)
\end{eqnarray*}%
The bundle $E^{+}$ is defined to be $\mathbb{C}^{d}\oplus \left\{ 0\right\} $
over the lower hemisphere and $\left\{ 0\right\} \oplus \mathbb{C}^{d}$
(with identifications as above), and the bundle $E^{-}$ is defined
similarly. The operator \textbf{$Q$} can be extended to be well-defined on
entire sphere, as follows. Define it as before on the two hemispheres, and
so that near the equator 
\begin{eqnarray*}
\mathbf{Q}\left( 
\begin{array}{c}
L_{+}\left( \theta _{s},\phi _{s}\right) \\ 
L_{-}\left( \theta _{s},\phi _{s}\right)%
\end{array}%
\right) &=&\left( 
\begin{array}{cc}
0 & \left( -\partial _{\phi _{s}}+Q^{S+}\right) Z^{-} \\ 
Z^{+}\left( \partial _{\phi _{s}}+Q^{S+}\right) & 0%
\end{array}%
\right) \left( 
\begin{array}{c}
L_{+}\left( \theta _{s},\phi _{s}\right) \\ 
L_{-}\left( \theta _{s},\phi _{s}\right)%
\end{array}%
\right) \\
&=&\left( 
\begin{array}{cc}
0 & \left( -\partial _{\phi _{n}}+Q^{S+}\right) Z^{-} \\ 
Z^{+}\left( \partial _{\phi _{n}}+Q^{S+}\right) & 0%
\end{array}%
\right) \left( 
\begin{array}{c}
U_{+}\left( \theta _{s},\phi _{n}\right) \\ 
U_{-}\left( \theta _{s},\phi _{n}\right)%
\end{array}%
\right)
\end{eqnarray*}%
To check compatibility, we see that 
\begin{eqnarray*}
&&\left( 
\begin{array}{cc}
0 & \left( -\partial _{\phi _{n}}+Q^{S+}\right) Z^{-} \\ 
Z^{+}\left( \partial _{\phi _{n}}+Q^{S+}\right) & 0%
\end{array}%
\right) \left( 
\begin{array}{c}
U_{+}\left( \theta _{s},\phi _{n}\right) \\ 
U_{-}\left( \theta _{s},\phi _{n}\right)%
\end{array}%
\right) \\
&=&\left( 
\begin{array}{cc}
0 & \left( \partial _{\phi _{s}}+Q^{S+}\right) Z^{-} \\ 
Z^{+}\left( -\partial _{\phi _{s}}+Q^{S+}\right) & 0%
\end{array}%
\right) \left( 
\begin{array}{cc}
0 & \left( Z^{+}\right) ^{-1} \\ 
\left( Z^{-}\right) ^{-1} & 0%
\end{array}%
\right) \left( 
\begin{array}{c}
L_{+}\left( \theta _{s},\phi _{s}\right) \\ 
L_{-}\left( \theta _{s},\phi _{s}\right)%
\end{array}%
\right) \\
&=&\left( 
\begin{array}{c}
\left( -\partial _{\phi _{s}}+Q^{S+}\right) L_{+}\left( \theta _{s},\phi
_{s}\right) \\ 
Z^{+}\left( -\partial _{\phi _{s}}+Q^{S+}\right) \left( Z^{+}\right)
^{-1}L_{-}\left( \theta _{s},\phi _{s}\right)%
\end{array}%
\right)
\end{eqnarray*}%
We continue to simplify to obtain 
\begin{eqnarray*}
&&\left( 
\begin{array}{cc}
0 & \left( -\partial _{\phi _{n}}+Q^{S+}\right) Z^{-} \\ 
Z^{+}\left( \partial _{\phi _{n}}+Q^{S+}\right) & 0%
\end{array}%
\right) \left( 
\begin{array}{c}
U_{+}\left( \theta _{s},\phi _{n}\right) \\ 
U_{-}\left( \theta _{s},\phi _{n}\right)%
\end{array}%
\right) \\
&=&\left( 
\begin{array}{cc}
0 & \left( Z^{+}\right) ^{-1} \\ 
\left( Z^{-}\right) ^{-1} & 0%
\end{array}%
\right) \left( 
\begin{array}{cc}
0 & \left( -\partial _{\phi _{s}}+Q^{S+}\right) Z^{-} \\ 
Z^{+}\left( \partial _{\phi _{s}}+Q^{S+}\right) & 0%
\end{array}%
\right) \left( 
\begin{array}{c}
L_{+}\left( \theta _{s},\phi _{s}\right) \\ 
L_{-}\left( \theta _{s},\phi _{s}\right)%
\end{array}%
\right) .
\end{eqnarray*}

\medskip Thus, the newly defined elliptic differential operator \textbf{$Q$}
on $\Gamma \left( S^{k},E^{+}\oplus E^{-}\right) $ is well-defined. This
operator is self adjoint, and remains $H$-equivariant by construction. In
fact \textbf{$Q$} has trivial kernel and thus zero index, and all possible
equivariant indices are likewise zero. Note that the bundle $E$ is no longer
trivial.

\subsection{Local contribution of the heat supertrace near the south pole 
\label{HeatKernelSphereCalculation}}

\subsubsection{Local elements of the kernel}

We now wish to compute the local supertrace of the equivariant heat kernel
associated to \textbf{$Q$}, in particular the part near the south pole. Let $%
\beta :H\rightarrow GL\left( V_{\beta }\right) $ denote an irreducible
representation of $H$. We know that the supertrace of $e^{-t\mathbf{Q}^{2}}$
restricted to sections of type $\beta $ is the index of \textbf{$Q$}$^{\beta
}=\left. \mathbf{Q}\right\vert _{\Gamma \left( S^{k},E^{+}\oplus
E^{-}\right) ^{\beta }}$, which is zero; however, the local supertrace need
not be trivial. Let $u\in \Gamma \left( S^{k},E^{+}\right) ^{\beta }$. Since 
$H$ commutes with $Q^{S}$, $\Gamma \left( S^{k},E^{+}\right) ^{\beta }$ is a
direct sum of eigenspaces of $Q^{S}$. Near the south pole we write 
\begin{equation*}
u\left( \theta _{s},\phi _{s}\right) =\sum_{\lambda }u_{\lambda }\left( \phi
_{s}\right) f_{\lambda }\left( \theta _{s}\right) ,
\end{equation*}%
where $Q^{S}f_{\lambda }=Q^{S+}f_{\lambda }=\lambda f_{\lambda }$.

If $u\in \Gamma _{\mathrm{loc}}\left( S^{k},E^{+}\right) ^{\beta }$ is a
local smooth solution to the equation \textbf{$Q$}$^{\beta }u=0$ near the
south pole, then 
\begin{equation*}
\partial _{\phi _{s}}u_{\lambda }+\frac{\lambda }{\phi _{s}}u_{\lambda }=0,
\end{equation*}%
or 
\begin{equation*}
u_{\lambda }\left( \phi _{s}\right) =c_{\lambda }\phi _{s}^{-\lambda }
\end{equation*}%
for some $c_{\lambda }\in \mathbb{C}$. Note that for $k>2$, $\lambda <\frac{k%
}{2}-1$ if and only if $\phi _{s}^{-\lambda }f_{\lambda }\left( \theta
_{s}\right) $ provides an element of the Sobolev space $H^{1}$ that locally
solves the differential equation (for $k=2$, the condition is $\lambda \leq
0 $). Because $\ker $\textbf{$Q$} is (globally) trivial, it is not possible
to continue this solution to a global $H^{1}$ solution of \textbf{$Q$}$u=0$
on the whole sphere.

On the other hand, if $v\in \Gamma _{\mathrm{loc}}\left( S^{k},E^{-}\right)
^{\beta }$ is a local smooth solution to \textbf{$Q$}$^{\beta }u=0$, then%
\begin{eqnarray*}
\mathbf{Q}^{\beta }v &=&\left( \partial _{\phi _{s}}+\frac{1}{\phi _{s}}%
Q^{S+}\right) ^{\ast }Z^{-}v \\
&=&\left( -\partial _{\phi _{s}}-\frac{k-1}{\phi _{s}}+\frac{1}{\phi _{s}}%
Q^{S+}\right) Z^{-}v \\
&=&Z^{-}\left( \partial _{\phi _{s}}+\frac{1}{\phi _{s}}Q^{S-}\right) v=0
\end{eqnarray*}%
so that $w=Z^{-}v\in \Gamma _{\mathrm{loc}}\left( S^{k},E^{+}\right) ^{\beta
}$ satisfies%
\begin{equation*}
\left( -\partial _{\phi _{s}}-\frac{k-1}{\phi _{s}}+\frac{1}{\phi _{s}}%
Q^{S+}\right) w=0.
\end{equation*}%
We write%
\begin{equation*}
w\left( \theta _{s},\phi _{s}\right) =\sum_{\lambda }w_{\lambda }\left( \phi
_{s}\right) f_{\lambda }\left( \theta _{s}\right) ,
\end{equation*}%
where as before $Q^{S}f_{\lambda }=Q^{S+}f_{\lambda }=\lambda f_{\lambda }$.

Then \textbf{$Q$}$^{\beta }\left( Z^{-}\right) ^{-1}w=0$ near the south pole
implies 
\begin{equation*}
\partial _{\phi _{s}}w_{\lambda }+\frac{k-1-\lambda }{\phi _{s}}w_{\lambda
}=0,
\end{equation*}%
or 
\begin{equation*}
w_{\lambda }\left( \phi _{s}\right) =e_{\lambda }\phi _{s}^{-k+1+\lambda }
\end{equation*}%
for some $e_{\lambda }\in \mathbb{C}$. Observe that for $k>2$, $\lambda >%
\frac{k}{2}$ if and only if $\phi _{s}^{-k+1+\lambda }\left( Z^{-}\right)
^{-1}f_{\lambda }\left( \theta _{s}\right) $ provides an element of $H^{1}$
that locally solves \textbf{$Q$}$v=0$ (the condition is $\lambda \geq 1$ for 
$k=2$). Because $\ker $\textbf{$Q$} is (globally) trivial, it is not
possible to continue this solution (or a constant solution) to a global $%
H^{1}$ solution of \textbf{$Q$}$\left( Z^{-}\right) ^{-1}w=0$ on the whole
sphere.

We remark further that if $k>2$, the sections of $E^{+}$ that satisfy 
\textbf{$Q$}$^{\beta }u=0$ on the sphere minus a small neighborhood of the
south pole that do extend to the north pole are exactly those of the form $%
\sum_{\lambda >\frac{k}{2}}c_{\lambda }\phi _{s}^{-\lambda }f_{\lambda
}\left( \theta _{s}\right) $ near the south pole (for $k=2$ the sum is over
all $\lambda \geq 1$). If $k>2$, the sections of $E^{-}$ that satisfy 
\textbf{$Q$}$^{\beta }u=0$ on the sphere minus a small neighborhood of the
south pole that do extend to the north pole are exactly those of the form $%
\sum_{\lambda <\frac{k}{2}-1}c_{\lambda }\phi _{s}^{-\lambda }f_{\lambda
}\left( \theta _{s}\right) $ near the south pole (for $k=2$ the sum is over
all $\lambda \leq 0$).

\subsubsection{A\ related boundary value problem\label{relatedBVP}}

In this section, we will replace the index problem for \textbf{$Q$}$^{\beta
} $ on the sphere with a related Atiyah-Patodi-Singer boundary problem, and
we will explicitly determine the difference between these indices in terms
of local heat kernel supertraces. Choose $\varepsilon >0$ such that $f\left(
y\right) =y$ on the $5\varepsilon $-ball centered at the south pole.
Consider all the data similar as before, but we modify the operator near the
south pole $x_{k+1}=-1$. We replace $\phi _{s}$ by $r\left( \phi _{s}\right)
=\exp \left( C_{\varepsilon }+\int_{0}^{\phi _{s}}\frac{1}{\psi \left(
y\right) }dy\right) $ for small $\phi _{s}$, where $\psi :\mathbb{%
R\rightarrow R}$ is a smooth function such that $\psi \left( y\right) =y$ in
a neighborhood of $y=3\varepsilon $ and $\psi \left( y\right) =\varepsilon $
for $y\leq \varepsilon $, and where $C_{\varepsilon }$ is a constant chosen
so that $C_{\varepsilon }+\int_{0}^{3\varepsilon }\frac{1}{\psi \left(
y\right) }dy=\log \left( 3\varepsilon \right) $. Suppose that $u\left( \phi
_{s},\theta _{s}\right) =\sum_{\lambda <\frac{k}{2}-1}c_{\lambda }\phi
_{s}^{-\lambda }f_{\lambda }\left( \theta _{s}\right) \in \Gamma \left(
S^{k},E^{+}\right) ^{\beta }$ is a local $H^{1}$ solution to \textbf{$Q$}$%
^{\beta }u=0$ near the south pole. Then 
\begin{eqnarray*}
\widetilde{u}\left( \phi _{s},\theta \right) &:&=u\left( r\left( \phi
_{s}\right) ,\theta \right) =\sum_{\lambda <\frac{k}{2}-1}c_{\lambda
}r^{-\lambda }f_{\lambda }\left( \theta \right) \\
&=&\sum_{\lambda <\frac{k}{2}-1}c_{\lambda }\exp \left( -\lambda \left(
C_{\varepsilon }+\int_{0}^{\phi _{s}}\frac{1}{\psi \left( y\right) }%
dy\right) \right) f_{\lambda }\left( \theta \right)
\end{eqnarray*}%
are local solutions to the equation 
\begin{equation*}
\widetilde{Q}\widetilde{u}=Z^{+}\left( \partial _{\phi _{s}}+\frac{1}{\psi
\left( \phi _{s}\right) }Q^{S+}\right) \widetilde{u}=0.
\end{equation*}%
(In the above sums, for $k=2$ the condition $\lambda <\frac{k}{2}-1$ is
replaced by $\lambda \leq 0$.) For $\phi _{s}<\varepsilon $, the equation
reads 
\begin{equation*}
\left( \partial _{\phi _{s}}+\frac{1}{\varepsilon }Q^{S+}\right) \widetilde{u%
}=0,
\end{equation*}%
and for $\phi _{s}>3\varepsilon $, the equation reads 
\begin{equation*}
\left( \partial _{\phi _{s}}+\frac{1}{f\left( \phi _{s}\right) }%
Q^{S+}\right) u=0.
\end{equation*}%
Note that $\widetilde{u}$ is $H$-equivariant, because the transformation
above affects the $\phi _{s}$ variable alone. Further, observe that the map $%
u\rightarrow \widetilde{u}$ is continuous with respect to the $C^{\infty }$
norm and is injective, and the $H^{1}$ condition on $u$ translates to the
boundary condition $\left( \phi _{s}=0\right) $%
\begin{gather*}
P_{\geq \frac{k}{2}-1}\left( \left. \widetilde{u}\right\vert _{\partial
}\right) =0 \\
\left( \text{or }P_{>0}\left( \left. \widetilde{u}\right\vert _{\partial
}\right) =0\text{ if }k=2\right) ,
\end{gather*}%
where $P_{\geq \frac{k}{2}-1}$ is the spectral projection onto the span of
the eigenspaces of $\left( Q^{S+}\right) ^{\beta }$ corresponding to
eigenvalues at least $\frac{k}{2}-1$. Note that no such $\widetilde{u}$ may
be continued to a global element of the kernel of the modified \textbf{$Q$}$%
^{\beta }$, by the reasoning used in the previous section. Thus the boundary
value problem for $\widetilde{u}$ with spectral boundary condition as above
yields no kernel, as does the original problem for \textbf{$Q$}$^{\beta }$
on the sphere.

Next, we modify the metric in a very specific way. We will choose the metric
so that%
\begin{equation}
ds^{2}=\left[ \tilde{g}\left( \phi _{s}\right) \right] ^{2}d\theta
^{2}+d\phi _{s}^{2},  \label{modifiedMetric}
\end{equation}%
where $\tilde{g}\left( \phi _{s}\right) =\phi _{s}$ in a neighborhood of $%
\phi _{s}=3\varepsilon $ (so that it agrees with the Euclidean metric), and
such that $\tilde{g}\left( \phi _{s}\right) $ is constant if $\phi
_{s}<\varepsilon $ (the cylindrical metric). Then, if $v\in \Gamma \left(
S^{k},E^{-}\right) ^{\beta }$ is a local solution to the equation $\left( 
\mathbf{Q}^{\beta }\right) ^{\ast }v=0$ , then $w=\left( Z^{-}\right) ^{-1}v$
satisfies 
\begin{eqnarray*}
\widetilde{w}\left( \phi _{s},\theta \right) &:&=w\left( r\left( \phi
_{s}\right) ,\theta \right) =\sum_{\lambda >\frac{k}{2}}c_{\lambda
}r^{-\lambda }f_{\lambda }\left( \theta \right) \\
&=&\sum_{\lambda >\frac{k}{2}}c_{\lambda }g\left( \phi _{s}\right) ^{-\left(
k-1\right) }\exp \left( \lambda \left( C_{\varepsilon }+\int_{0}^{\phi _{s}}%
\frac{1}{\psi \left( y\right) }dy\right) \right) f_{\lambda }\left( \theta
\right)
\end{eqnarray*}%
are local solutions to the equation 
\begin{equation*}
\left( -\partial _{\phi _{s}}-\frac{\left( k-1\right) \tilde{g}^{\prime
}\left( \phi _{s}\right) }{\tilde{g}\left( \phi _{s}\right) }+\frac{1}{\psi
\left( \phi _{s}\right) }Q^{S+}\right) \widetilde{w}=0.
\end{equation*}%
(In the sums above, if $k=2$ then the condition $\lambda >\frac{k}{2}$ is
replaced by $\lambda \geq 1$.) For $\phi _{s}<\varepsilon $, the equation
reads 
\begin{equation*}
\left( -\partial _{\phi _{s}}+\frac{1}{\varepsilon }Q^{S+}\right) \widetilde{%
w}=0,
\end{equation*}%
and for $\phi _{s}>3\varepsilon $, the equation reads 
\begin{equation*}
\left( -\partial _{\phi _{s}}-\frac{\left( k-1\right) \tilde{g}^{\prime
}\left( \phi _{s}\right) }{\tilde{g}\left( \phi _{s}\right) }+\frac{1}{%
f\left( \phi _{s}\right) }Q^{S+}\right) \widetilde{w}=0.
\end{equation*}%
This time $\widetilde{w}$ satisfies the boundary condition $\left( \phi
_{s}=0\right) $%
\begin{gather*}
P_{\leq \frac{k}{2}}\left( \left. \widetilde{w}\right\vert _{\partial
}\right) =0 \\
\left( \text{or }P_{<1}\left( \left. \widetilde{w}\right\vert _{\partial
}\right) =0\text{ if }k=2\right) ,
\end{gather*}%
where $P_{\leq \frac{k}{2}}$ is the spectral projection onto the span of the
eigenspaces of $\left( Q^{S+}\right) ^{\beta }$ corresponding to eigenvalues
at most $\frac{k}{2}$. Note that no such $\widetilde{w}$ may be continued to
a global element of the kernel of the modified \textbf{$Q$}$^{\beta }$, by
the reasoning used in the previous section. Thus the index problem for the
original \textbf{$Q$}$^{\beta }$ (for $u$ and $w$) is the same as the
boundary value \textquotedblleft index problem\textquotedblright\ for $%
\widetilde{u}$ and $\widetilde{w}$ described as above.

Observe that the operator and metric now fit the conditions of the
Atiyah-Patodi-Singer theorem (\cite{APS1}) for manifolds with boundary, but
the boundary conditions are not the same. More precisely, the related
situation is that of the equivariant version of the Atiyah-Patodi-Singer for
elliptic operators in \cite[Theorem 1.2]{Don}. The equivariant theorem
actually gives a formula for the character $\mathrm{ind}_{g}\left( \mathbf{Q}%
\right) $, which by ellipticity is a smooth function on $H$. We are
interested in $\mathrm{ind}^{\beta }\left( \mathbf{Q}\right) $, which is
easily computable from $\mathrm{ind}_{g}\left( \mathbf{Q}\right) $ by
determining the $\beta $-component of the character. All of the parts of the
equivariant index formula (the integral of the heat supertrace, the eta
invariant, and the dimension of the kernel) are smooth functions on $H$ and
can be computed similarly. The resulting equivariant eta invariant of any $H$%
-equivariant operator $L$ is defined to be 
\begin{equation}
\eta \left( L^{\beta }\right) =\left. \sum_{\lambda ^{\beta }\neq 0}\mathrm{%
sign}\left( \lambda ^{\beta }\right) \left\vert \lambda ^{\beta }\right\vert
^{-s}\right\vert _{s=0},  \label{equivariant eta definition}
\end{equation}%
where $\beta :H\rightarrow U\left( V_{\beta }\right) $ is an irreducible
representation and where the sum is over all nonzero eigenvalues (including
multiplicities) of $L$ restricted to the space of sections of type $\beta $.
(Recall that the eigenvalues grow as in Formula (\ref{eigenvalueAsymptotics}%
).)

The Atiyah-Patodi-Singer boundary conditions on the punctured sphere would be%
\begin{equation*}
P_{\geq 0}\widetilde{u}=0;~P_{<0}\widetilde{w}=0.
\end{equation*}%
If these conditions are used instead of the conditions induced from the
closed manifold problem, note that again no such $\widetilde{u}$ could be
continued to the north pole, because this would require that $\lambda <0$
and $\lambda >\frac{k}{2}$ (or $\lambda \geq 1$ for $k=2$) simultaneously.
However, the possibility exists that some choices of $\widetilde{w}$ could
be continued to sections of $\ker $ \textbf{$Q$}$^{\beta -}$ over the north
pole, requiring only that $0\leq \lambda <\frac{k}{2}-1$ (or $0\leq \lambda
\leq 0$ for $k=2$). By Proposition \ref{DSspectrumCorollaryHversion}, the
only part of the spectrum of $Q^{S+}$ in this range is the zero eigenvalue,
and thus the equivariant Atiyah-Patodi-Singer index of the modified \textbf{$%
Q$}$^{\beta }$ is 
\begin{equation}
-h\left( Q^{S+,\beta }\right) :=-\dim \ker Q^{S+,\beta }=-\dim \mathbb{C}%
^{d,\beta }=-\dim E^{+,\beta }.  \label{equivariant h definition}
\end{equation}%
By the equivariant version of the Atiyah-Patodi-Singer Theorem, we have that%
\begin{eqnarray*}
-h\left( Q^{S+,\beta }\right) &=&\int_{\frac{1}{2}DS^{k}}\widetilde{\alpha }%
^{\beta }-\frac{1}{2}\left( \eta \left( Q^{S+,\beta }\right) +h\left(
Q^{S+,\beta }\right) \right) \\
&=&\int_{S^{k}-B_{\varepsilon }\left( SP\right) }^{\beta }\widetilde{\alpha }%
^{\beta }+\int_{B_{\varepsilon }\left( \partial \frac{1}{2}DS^{k}\right) }%
\widetilde{\alpha }^{\beta }-\frac{1}{2}\left( \eta \left( Q^{S+,\beta
}\right) +h\left( Q^{S+,\beta }\right) \right) ,\text{ or} \\
0 &=&\int_{S^{k}-B_{\varepsilon }\left( SP\right) }\widetilde{\alpha }%
^{\beta }+\int_{B_{\varepsilon }\left( \partial \frac{1}{2}DS^{k}\right) }%
\widetilde{\alpha }^{\beta }+\frac{1}{2}\left( -\eta \left( Q^{S+,\beta
}\right) +h\left( Q^{S+,\beta }\right) \right) ,
\end{eqnarray*}%
where $\widetilde{\alpha }^{\beta }$ is the the supertrace of the heat
kernel of \textbf{$Q$}$^{\beta }$ over the double $DS^{k}$ of $S^{k}$ (blown
up at the south pole), where $\frac{1}{2}DS^{k}$ means the half of the
double, where $B_{\varepsilon }\left( SP\right) $ means an $\varepsilon $%
-neighborhood of the south pole, and where $\eta \left( Q^{S+,\beta }\right) 
$ is the eta invariant of $Q^{S+,\beta }$. The set $B_{\varepsilon }\left(
\partial \frac{1}{2}DS^{k}\right) $ means the $\varepsilon $-collar around
the boundary $\partial \frac{1}{2}DS^{k}$, so that $\frac{1}{2}DS^{k}=\left(
S^{k}-B_{\varepsilon }\left( SP\right) \right) \cup B_{\varepsilon }\left(
\partial \frac{1}{2}DS^{k}\right) $. Thus, if $\alpha _{S}^{\beta }$ denotes
the local heat supertrace of the operator on the sphere, 
\begin{eqnarray*}
0 &=&\mathrm{ind}\left( \mathbf{Q}^{\beta }\right) =\int_{S^{k}}\alpha
_{S}^{\beta } \\
&=&\int_{S^{k}-B_{\varepsilon }\left( SP\right) }\alpha _{S}^{\beta
}+\int_{B_{\varepsilon }\left( SP\right) }\alpha _{S}^{\beta } \\
&=&\int_{S^{k}-B_{\varepsilon }\left( SP\right) }\widetilde{\alpha }^{\beta
}+\int_{B_{\varepsilon }\left( \partial \frac{1}{2}DS^{k}\right) }\widetilde{%
\alpha }^{\beta }+\frac{1}{2}\left( -\eta \left( Q^{S+,\beta }\right)
+h\left( Q^{S+,\beta }\right) \right) .
\end{eqnarray*}%
If $\alpha ^{\beta }$ denotes the Euclidean heat supertrace corresponding to
the original operator $Q^{\beta }$ on $\mathbb{R}^{k}$, 
\begin{eqnarray*}
\int_{S^{k}-B_{\varepsilon }\left( SP\right) }\alpha _{S}^{\beta }
&=&\int_{S^{k}-B_{\varepsilon }\left( SP\right) }\widetilde{\alpha }^{\beta
}+\mathcal{O}\left( e^{-c/t}\right) ,\text{ and} \\
\int_{B_{\varepsilon }\left( 0\right) }\alpha ^{\beta }
&=&\int_{B_{\varepsilon }\left( SP\right) }\alpha _{S}^{\beta }+\mathcal{O}%
\left( e^{-c/t}\right) ,
\end{eqnarray*}
for some $c>0$, since the local data is the same. Then 
\begin{equation}
\int_{B_{\varepsilon }\left( 0\right) }\alpha ^{\beta }=\int_{B_{\varepsilon
}\left( \partial \frac{1}{2}DS^{k}\right) }\widetilde{\alpha }^{\beta }+%
\frac{1}{2}\left( -\eta \left( Q^{S+,\beta }\right) +h\left( Q^{S+,\beta
}\right) \right) +\mathcal{O}\left( e^{-c/t}\right) ,
\label{localSupertraceEta}
\end{equation}%
the contribution of the equivariant heat supertrace of the constant
coefficient operator $Q^{\beta }$ in an $\varepsilon $-ball around the
origin of $\mathbb{R}^{k}$. As explained in Remark \ref{stabilityRemarks}, $%
\eta \left( Q^{S+,\beta }\right) $ and $h\left( Q^{S+,\beta }\right) $ are
invariant under stable $H$-equivariant homotopies of the original operator $%
Q $. We note that the formula above is scale invariant and independent of
all the various choices we made through the calculation. It expresses the
index contribution of the Euclidean neighborhood in terms of the index
contribution from the same neighborhood of the blown-up operator, where a
small ball around the center of the Euclidean neighborhood is replaced by a
collar. Because of the homotopy and scale invariance, this formula remains
valid if the operator is blown up in any way so that it is constant on the
small collared neighborhood and so that the spherical part $Q^{S+,\beta }$
comes from the polar coordinate expression. Similarly, the metric may be
chosen in any way as long as it is cylindrical on the collar.

\subsection{The case $k=1$}

We repeat the analysis for the case $k=1$. We consider a operator $%
Q:C^{\infty }\left( \mathbb{R},\mathbb{C}^{d}\oplus \mathbb{C}^{d}\right)
\rightarrow C^{\infty }\left( \mathbb{R},\mathbb{C}^{d}\oplus \mathbb{C}%
^{d}\right) $ of the form 
\begin{eqnarray*}
Q &=&A\partial _{x} \\
&=&\left( 
\begin{array}{cc}
0 & Q_{1}^{\ast } \\ 
Q_{1} & 0%
\end{array}%
\right) \\
&=&\left( 
\begin{array}{cc}
0 & -A_{1}^{\ast } \\ 
A_{1} & 0%
\end{array}%
\right) \partial _{x}=\left\{ 
\begin{array}{ll}
A\partial _{r} & \text{if }x>0 \\ 
-A\partial _{r}~ & \text{if }x<0%
\end{array}%
\right. ,
\end{eqnarray*}%
where by assumption $A_{1}\in GL\left( d,\mathbb{C}\right) $ and $%
r=\left\vert x\right\vert $. Suppose that $H$ acts nontrivially on $\mathbb{R%
}$ by multiplication by $\pm 1$, that $H$ acts unitarily on each copy of $%
\mathbb{C}^{d}$, and that $Q$ is $H$-equivariant. Let $H_{0}$ be the
(normal) subgroup of $H$ that fixes $\mathbb{R}$; then $H_{0}$ contains the
identity component of $H$, and $H\diagup H_{0}\cong \mathbb{Z}_{2}$. Let $%
\gamma \in H$ be chosen so that $\gamma H_{0}$ generates $H\diagup H_{0}$.
Then we have that 
\begin{eqnarray*}
hA &=&Ah\text{ for all }h\in H_{0},\text{ and} \\
\gamma A &=&-A\gamma ,
\end{eqnarray*}%
or if we write

\begin{equation*}
h=\left( 
\begin{array}{cc}
h_{+} & 0 \\ 
0 & h_{-}%
\end{array}%
\right) ,~\gamma =\left( 
\begin{array}{cc}
\gamma _{+} & 0 \\ 
0 & \gamma _{-}%
\end{array}%
\right) ,
\end{equation*}%
we have 
\begin{eqnarray}
h_{-} &=&A_{1}h_{+}A_{1}^{-1}  \notag \\
\gamma _{-} &=&-A_{1}\gamma _{+}A_{1}^{-1}.  \label{commutation relations}
\end{eqnarray}%
Further, if $i\lambda \in i\mathbb{R\diagdown }\left\{ 0\right\} $ is an
eigenvalue of $A$ corresponding to eigenspace 
\begin{equation*}
E_{i\lambda }=\left\{ \left( 
\begin{array}{c}
v_{+} \\ 
v_{-}%
\end{array}%
\right) :A\left( 
\begin{array}{c}
v_{+} \\ 
v_{-}%
\end{array}%
\right) =i\lambda \left( 
\begin{array}{c}
v_{+} \\ 
v_{-}%
\end{array}%
\right) \right\} ,
\end{equation*}%
then 
\begin{equation*}
v_{-}=\frac{-i}{\lambda }A_{1}v_{+},
\end{equation*}%
and%
\begin{equation*}
E_{-i\lambda }=\left\{ \left( 
\begin{array}{c}
v_{+} \\ 
-v_{-}%
\end{array}%
\right) :\left( 
\begin{array}{c}
v_{+} \\ 
v_{-}%
\end{array}%
\right) \in E_{i\lambda }\right\} .
\end{equation*}%
By (\ref{commutation relations}), if $W_{+,\left\vert \lambda \right\vert
}^{\beta }$ is the subspace of $E_{i\lambda }\cap \mathbb{C}%
^{d+}=E_{-i\lambda }\cap \mathbb{C}^{d+}$ consisting of the vectors of $H$%
-representation type $\left[ \beta ,V_{\beta }\right] $, then 
\begin{equation*}
W_{-,\left\vert \lambda \right\vert }^{\overline{\beta }}:=\frac{-i}{\lambda 
}A_{1}W_{+,\left\vert \lambda \right\vert }^{\beta }\cong W_{+,\left\vert
\lambda \right\vert }^{\beta }
\end{equation*}%
is the subspace of $E_{i\lambda }\cap \mathbb{C}^{d-}=E_{-i\lambda }\cap 
\mathbb{C}^{d-}$ consisting of vectors of $H$-representation type $\overline{%
\beta }$, where $\overline{\beta }$ is the \textquotedblleft
conjugate\textquotedblright\ of $\beta $, defined up to equivalence as 
\begin{eqnarray*}
\overline{\beta }\left( h\right) &=&\beta \left( h\right) \text{ for all }%
h\in H_{0},~\text{and} \\
\overline{\beta }\left( \gamma \right) &=&-\beta \left( \gamma \right) .
\end{eqnarray*}%
Observe that $W_{+,\left\vert \lambda \right\vert }^{\beta }$ is the $\beta $
part of the $\lambda ^{2}$ eigenspace of $A_{1}^{\ast }A_{1}$ consisting of
vectors of type $\beta $, and $W_{-,\left\vert \lambda \right\vert }^{\beta
} $ is the $\beta $ part of the $\lambda ^{2}$ eigenspace of $%
A_{1}A_{1}^{\ast }$ consisting of vectors of type $\beta $. Let $P_{\pm
,\left\vert \lambda \right\vert }^{\beta }:\mathbb{C}^{d\pm }\rightarrow
W_{\pm ,\left\vert \lambda \right\vert }^{\beta }$ denote the orthogonal
projections.

We now calculate the local contribution $\int_{x=-\varepsilon }^{\varepsilon
}\alpha ^{\beta }$ of the equivariant heat supertrace corresponding to 
\begin{equation*}
Q^{2}=\left( 
\begin{array}{cc}
-A_{1}^{\ast }A_{1} & 0 \\ 
0 & -A_{1}A_{1}^{\ast }%
\end{array}%
\right) \partial _{x}^{2}.
\end{equation*}%
We have that%
\begin{equation*}
Q^{2}P_{\pm ,\left\vert \lambda \right\vert }^{\beta }=-\lambda ^{2}\partial
_{x}^{2}P_{\pm ,\left\vert \lambda \right\vert }^{\beta }=-\lambda
^{2}P_{\pm ,\left\vert \lambda \right\vert }^{\beta }\partial _{x}^{2}
\end{equation*}%
Since the fundamental solution of $\frac{\partial }{\partial t}-\lambda ^{2}%
\frac{\partial ^{2}}{\partial x^{2}}$ on $\mathbb{R}$ is $\frac{1}{%
\left\vert \lambda \right\vert \sqrt{4\pi t}}\exp \left( -\left( x-y\right)
^{2}/\left( 4\lambda ^{2}t\right) \right) $, if $K\left( t,x,y\right) $ is
the kernel of the operator $e^{-tQ^{2}}$ on $\mathbb{R}$, for $0<\delta
<\varepsilon $, we have 
\begin{eqnarray*}
\frac{1}{\dim V_{\beta }}\int_{x=-\delta }^{\delta }\alpha ^{\beta }
&=&\int_{x=-\delta }^{\delta }\int_{h\in H_{0}}\left( \mathrm{tr~}h\cdot
K^{+}\left( t,x,x\right) -\mathrm{tr~}h\cdot K^{-}\left( t,x,x\right)
\right) ~\overline{\chi _{\beta }\left( h\right) }~dh \\
&&+\int_{h\in H_{0}}\left( \mathrm{tr~}\gamma h\cdot K^{+}\left( t,\gamma
^{-1}x,x\right) -\mathrm{tr~}\gamma h\cdot K^{-}\left( t,\gamma
^{-1}x,x\right) \right) ~\overline{\chi _{\beta }\left( h\right) }~dh~dx \\
&=&\int_{h\in H_{0}}\sum_{\left\vert \lambda \right\vert ,\tau }\frac{%
2\delta }{\left\vert \lambda \right\vert \sqrt{4\pi t}}\left( \mathrm{tr~}%
\left( h\cdot P_{+,\left\vert \lambda \right\vert }^{\tau }\right) -\mathrm{%
tr~}\left( h\cdot P_{-,\left\vert \lambda \right\vert }^{\tau }\right)
\right) ~\overline{\chi _{\beta }\left( h\right) }~dh \\
&&+\int_{x=-\delta }^{\delta }\int_{h\in H_{0}}\sum_{\left\vert \lambda
\right\vert ,\tau }\frac{1}{\left\vert \lambda \right\vert \sqrt{4\pi t}}%
\exp \left( \frac{-\left\vert 2x\right\vert ^{2}}{4\lambda ^{2}t}\right)
\left( \mathrm{tr~}\left( \gamma h\cdot \sum_{\left\vert \lambda \right\vert
,\tau }P_{+,\left\vert \lambda \right\vert }^{\tau }\right) \right. \\
&&\left. -\mathrm{tr~}\left( \gamma h\cdot \sum_{\left\vert \lambda
\right\vert ,\tau }P_{-,\left\vert \lambda \right\vert }^{\tau }\right)
\right) ~\overline{\chi _{\beta }\left( h\right) }~dh~dx.
\end{eqnarray*}%
Next, since $\int_{-\infty }^{\infty }\frac{1}{\left\vert \lambda
\right\vert \sqrt{4\pi t}}\exp \left( \frac{-\left\vert 2x\right\vert ^{2}}{%
4\lambda ^{2}t}\right) dx=\allowbreak \frac{1}{2}$, 
\begin{gather*}
\frac{1}{\dim V_{\beta }}\int_{x=-\delta }^{\delta }\alpha ^{\beta
}=\sum_{\left\vert \lambda \right\vert ,\tau }\frac{2\delta }{\left\vert
\lambda \right\vert \sqrt{4\pi t}}\int_{h\in H_{0}}\left( \mathrm{tr~}\left(
h\cdot P_{+,\left\vert \lambda \right\vert }^{\tau }\right) -\mathrm{tr~}%
\left( h\cdot P_{-,\left\vert \lambda \right\vert }^{\tau }\right) \right) ~%
\overline{\chi _{\beta }\left( h\right) }~dh \\
+\frac{1}{2}\sum_{\left\vert \lambda \right\vert ,\tau }\int_{h\in
H_{0}}\left( \mathrm{tr~}\left( \gamma h\cdot P_{+,\left\vert \lambda
\right\vert }^{\tau }\right) -\mathrm{tr~}\left( \gamma h\cdot
P_{-,\left\vert \lambda \right\vert }^{\tau }\right) \right) ~\overline{\chi
_{\beta }\left( h\right) }~dh+O\left( e^{-c/t}\right) ~.
\end{gather*}%
By (\ref{ProjectionLemma}) we conclude that the contribution of the
equivariant heat supertrace on $\left( -\varepsilon ,\varepsilon \right)
\subset \mathbb{R}^{1}$ is, for any $\delta $ such that $0<\delta
<\varepsilon $, 
\begin{equation*}
\int_{-\varepsilon }^{\varepsilon }\alpha ^{\beta }=\int_{\left[
-\varepsilon ,\varepsilon \right] \diagdown \left[ -\delta ,\delta \right]
}\alpha ^{\beta }+\frac{1}{2}\left( \mathrm{\dim }\left( \mathbb{C}%
^{d+}\right) ^{\beta }-\mathrm{\dim }\left( \mathbb{C}^{d+}\right) ^{%
\overline{\beta }}\right) +O\left( e^{-c/t}\right) .
\end{equation*}%
In order to make our formulas consistent throughout the paper, we define for
the (equivariant) zero operator $0:\mathbb{C}^{d+}\rightarrow \mathbb{C}%
^{d-} $%
\begin{equation}
h\left( 0^{\beta }\right) :=\dim \left( \mathbb{C}^{d+}\right) ^{\beta
};~~\eta \left( 0^{\beta }\right) :=\dim \left( \mathbb{C}^{d+}\right) ^{%
\overline{\beta }},  \label{equivariant eta definition dim zero}
\end{equation}%
so that%
\begin{equation}
\int_{-\varepsilon }^{\varepsilon }\alpha ^{\beta }=\int_{\left[
-\varepsilon ,\varepsilon \right] \diagdown \left[ -\delta ,\delta \right]
}\alpha ^{\beta }+\frac{1}{2}\left( h\left( 0^{\beta }\right) -\eta \left(
0^{\beta }\right) \right) +\mathcal{O}\left( e^{-c/t}\right) .
\label{localSupertraceEtaCodim1}
\end{equation}

\section{Desingularizing along a singular stratum\label%
{BlowupDoubleMainSection}}

\subsection{Topological Desingularization\label{BlowUpSection}}

Assume that $G$ is a compact Lie group that acts on a Riemannian manifold $M$
by isometries. We will construct a new $G$-manifold $N$ that has a single
stratum (of type $\left[ G_{0}\right] $) and that is a branched cover of $M$%
, branched over the singular strata. A distinguished fundamental domain of $%
M_{0}$ in $N$ is called the desingularization of $M$ and is denoted $%
\widetilde{M}$. We also refer to \cite{AlMel} for their recent related
explanation of this desingularization (which they call \emph{resolution}).

A sequence of modifications is used to construct $N$ and $\widetilde{M}%
\subset N$. Let $M_{j}$ be a minimal stratum. Let $T_{\varepsilon }\left(
M_{j}\right) $ denote a tubular neighborhood of radius $\varepsilon $ around 
$M_{j}$, with $\varepsilon $ chosen sufficiently small so that all orbits in 
$T_{\varepsilon }\left( M_{j}\right) \setminus M_{j}$ are of type $\left[
G_{k}\right] $, where $\left[ G_{k}\right] <\left[ G_{j}\right] $. Let 
\begin{equation*}
N^{1}=\left( M\setminus T_{\varepsilon }\left( M_{j}\right) \right) \cup
_{\partial T_{\varepsilon }\left( M_{j}\right) }\left( M\setminus
T_{\varepsilon }\left( M_{j}\right) \right)
\end{equation*}%
be the manifold constructed by gluing two copies of $\left( M\setminus
T_{\varepsilon }\left( M_{j}\right) \right) $ smoothly along the boundary
(see Section \ref{codimOneSection} for the codimension one case). Since the $%
T_{\varepsilon }\left( M_{j}\right) $ is saturated (a union of $G$-orbits),
the $G$-action lifts to $N^{1}$. Note that the strata of the $G$-action on $%
N^{1}$ correspond to strata in $M\setminus T_{\varepsilon }\left(
M_{j}\right) $. If $M_{k}\cap \left( M\setminus T_{\varepsilon }\left(
M_{j}\right) \right) $ is nontrivial, then the stratum corresponding to
isotropy type $\left[ G_{k}\right] $ on $N^{1}$ is 
\begin{equation*}
N_{k}^{1}=\left( M_{k}\cap \left( M\setminus T_{\varepsilon }\left(
M_{j}\right) \right) \right) \cup _{\left( M_{k}\cap \partial T_{\varepsilon
}\left( M_{j}\right) \right) }\left( M_{k}\cap \left( M\setminus
T_{\varepsilon }\left( M_{j}\right) \right) \right) .
\end{equation*}%
Thus, $N^{1}$ is a $G$-manifold with one fewer stratum than $M$, and $%
M\setminus M_{j}$ is diffeomorphic to one copy of $\left( M\setminus
T_{\varepsilon }\left( M_{j}\right) \right) $, denoted $\widetilde{M}^{1}$
in $N^{1}$. In fact, $N^{1}$ is a branched double cover of $M$, branched
over $M_{j}$. If $N^{1}$ has one orbit type, then we set $N=N^{1}$ and $%
\widetilde{M}=\widetilde{M}^{1}$. If $N^{1}$ has more than one orbit type,
we repeat the process with the $G$-manifold $N^{1}$ to produce a new $G$%
-manifold $N^{2}$ with two fewer orbit types than $M$ and that is a $4$-fold
branched cover of $M$. Again, $\widetilde{M}^{2}$ is a fundamental domain of 
$\widetilde{M}^{1}\setminus \left\{ \text{a minimal stratum}\right\} $,
which is a fundamental domain of $M$ with two strata removed. We continue
until $N=N^{r}$ is a $G$-manifold with all orbits of type $\left[ G_{0}%
\right] $ and is a $2^{r}$-fold branched cover of $M$, branched over $%
M\setminus M_{0}$. We set $\widetilde{M}=\widetilde{M}^{r}$, which is a
fundamental domain of $M_{0}$ in $N$.

Further, one may independently desingularize $M_{\geq j}$, since this
submanifold is itself a closed $G$-manifold. If $M_{\geq j}$ has more than
one connected component, we may desingularize all components simultaneously.
The isotropy type of all points of $\widetilde{M_{\geq j}}$ is $\left[ G_{j}%
\right] $, and $\widetilde{M_{\geq j}}\diagup G$ is a smooth (open) manifold.

\subsection{Modification of the metric and differential operator}

We now more precisely describe the desingularization. If $M$ is equipped
with a $G$-equivariant, transversally elliptic differential operator on
sections of an equivariant vector bundle over $M$, then this data may be
pulled back to the desingularization $\widetilde{M}$. Given the bundle and
operator over $N^{j}$, simply form the invertible double of the operator on $%
N^{j+1}$, which is the double of the manifold with boundary $N^{j}\setminus
T_{\varepsilon }\left( \Sigma \right) $, where $\Sigma $ is a minimal
stratum on $N^{j}$.

Specifically, we modify the metric equivariantly so that there exists $%
\varepsilon >0$ such that the tubular neighborhood $B_{2\varepsilon }\Sigma $
of $\Sigma $ in $N^{j}$ is isometric to a ball of radius $2\varepsilon $ in
the normal bundle $N\Sigma $. In polar coordinates, this metric is $%
ds^{2}=dr^{2}+d\sigma ^{2}+r^{2}d\theta _{\sigma }^{2}$, with $r\in \left(
0,2\varepsilon \right) $, $d\sigma ^{2}$ is the metric on $\Sigma $, and $%
d\theta _{\sigma }^{2}$ is the metric on $S\left( N_{\sigma }\Sigma \right) $%
, the unit sphere in $N_{\sigma }\Sigma $; note that $d\theta _{\sigma }^{2}$
is isometric to the Euclidean metric on the unit sphere. We simply choose
the horizontal metric on $B_{2\varepsilon }\Sigma $ to be the pullback of
the metric on the base $\Sigma $, the fiber metric to be Euclidean, and we
require that horizontal and vertical vectors be orthogonal. We do not assume
that the horizontal distribution is integrable.

Next, we replace $r^{2}$ with $f\left( r\right) =\left[ \tilde{g}\left(
r\right) \right] ^{2}$ in the expression for the metric, where $\tilde{g}$
is defined as in (\ref{modifiedMetric}), so that the metric is cylindrical
for small $r$.

In our description of the modification of the differential operator, we will
need the notation for the (external) product of differential operators.
Suppose that $F\hookrightarrow X\overset{\pi }{\rightarrow }B$ is a fiber
bundle that is locally a metric product. Given an operator $A_{1,x}:\Gamma
\left( \pi ^{-1}\left( x\right) ,E_{1}\right) \rightarrow \Gamma \left( \pi
^{-1}\left( x\right) ,F_{1}\right) $ that is locally given as a differential
operator $A_{1}:\Gamma \left( F,E_{1}\right) \rightarrow \Gamma \left(
F,F_{1}\right) $ and $A_{2}:\Gamma \left( B,E_{2}\right) \rightarrow \Gamma
\left( B,F_{2}\right) $ on Hermitian bundles, we define the product 
\begin{equation*}
A_{1,x}\ast A_{2}:\Gamma \left( X,\left( E_{1}\boxtimes E_{2}\right) \oplus
\left( F_{1}\boxtimes F_{2}\right) \right) \rightarrow \Gamma \left(
X,\left( F_{1}\boxtimes E_{2}\right) \oplus \left( E_{1}\boxtimes
F_{2}\right) \right)
\end{equation*}%
as the unique linear operator that satisfies locally 
\begin{equation*}
A_{1,x}\ast A_{2}=\left( 
\begin{array}{ll}
A_{1}\boxtimes \mathbf{1} & -\mathbf{1}\boxtimes A_{2}^{\ast } \\ 
\mathbf{1}\boxtimes A_{2} & A_{1}^{\ast }\boxtimes \mathbf{1}%
\end{array}%
\right)
\end{equation*}%
on sections of 
\begin{equation*}
\left( 
\begin{array}{l}
E_{1}\boxtimes E_{2} \\ 
F_{1}\boxtimes F_{2}%
\end{array}%
\right)
\end{equation*}%
of the form $\left( 
\begin{array}{l}
u_{1}\boxtimes u_{2} \\ 
v_{1}\boxtimes v_{2}%
\end{array}%
\right) $, where $u_{1}\in \Gamma \left( F,E_{1}\right) $, $u_{2}\in \Gamma
\left( B,E_{2}\right) $, $v_{1}\in \Gamma \left( F,F_{1}\right) $, $v_{2}\in
\Gamma \left( B,E_{2}\right) $. This coincides with the product in various
versions of K-theory (see, for example, \cite{A}, \cite[pp. 384ff]{LM}),
which is used to define the Thom Isomorphism in vector bundles.

Let $D=D^{+}:\Gamma \left( N^{j},E^{+}\right) \rightarrow \Gamma \left(
N^{j},E^{-}\right) $ be the given first order, transversally elliptic, $G$%
-equivariant differential operator. Let $\Sigma \ $be a minimal stratum of $%
N^{j}$. Here we assume that $\Sigma $ has codimension at least two. We
modify the metrics and bundles equivariantly so that there exists $%
\varepsilon >0$ such that the tubular neighborhood $B_{\varepsilon }\left(
\Sigma \right) $ of $\Sigma $ in $M$ is isometric to a ball of radius $%
\varepsilon $ in the normal bundle $N\Sigma $, and so that the $G$%
-equivariant bundle $E$ over $B_{\varepsilon }\left( \Sigma \right) $ is a
pullback of the bundle $\left. E\right\vert _{\Sigma }\rightarrow \Sigma $.
We assume that near $\Sigma $, after a $G$-equivariant homotopy $D^{+}$ can
be written on $B_{\varepsilon }\left( \Sigma \right) $ locally as the
product 
\begin{equation}
D^{+}=\left( D_{N}\ast D_{\Sigma }\right) ^{+},  \label{Dsplitting}
\end{equation}%
where $D_{\Sigma }\ $is a transversally elliptic, $G$-equivariant, first
order operator on the stratum $\Sigma $, and $D_{N}$ is a $G$-equivariant,
first order operator on $B_{\varepsilon }\left( \Sigma \right) $ that is
elliptic on the fibers. If $r$ is the distance from $\Sigma $, we write $%
D_{N}$ in polar coordinates as%
\begin{equation*}
D_{N}=Z\left( \nabla _{\partial _{r}}^{E}+\frac{1}{r}D^{S}\right)
\end{equation*}%
where $Z=-i\sigma \left( D_{N}\right) \left( \partial _{r}\right) $ is a
local bundle isomorphism and the map $D^{S}$ is a purely first order
operator that differentiates in the unit normal bundle directions tangent to 
$S_{x}\Sigma $.

We modify the operator $D_{N}$ on each Euclidean fiber of $N\Sigma \overset{%
\pi }{\rightarrow }\Sigma $ exactly as in Section \ref{relatedBVP}; the
result is a $G$-manifold $\widetilde{M}^{j}$ with boundary $\partial 
\widetilde{M}^{j}$, a $G$-vector bundle $\widetilde{E}^{j}$, and the induced
operator $\widetilde{D}^{j}$, all of which locally agree with the original
counterparts outside $B_{\varepsilon }\left( \Sigma \right) $. We may double 
$\widetilde{M}^{j}$ along the boundary $\partial \widetilde{M}^{j}$ and
reverse the chirality of $\widetilde{E}^{j}$ as described in Section \ref%
{inducedOperatorSphereSection} for the case of a hemisphere and in general
in \cite[Ch. 9]{Bo-Wo}. Doubling produces a closed $G$-manifold $N^{j}$, a $%
G $-vector bundle $E^{j}$, and a first-order transversally elliptic
differential operator $D^{j}$. This process may be iterated until all orbits
of the resulting $G$-manifold are principal. The case where some strata have
codimension $1$ is addressed in the following section.

\subsection{Codimension one case\label{codimOneSection}}

\medskip We now give the definitions of the previous section for the case
when there is a minimal stratum $\Sigma $ of codimension $1$. Only the
changes to the argument are noted. This means that the isotropy subgroup $H$
corresponding to $\Sigma $ contains a principal isotropy subgroup of index
two. If $r$ is the distance from $\Sigma $, then $D_{N}$ has the form%
\begin{equation*}
D_{N}=Z\left( \nabla _{\partial _{r}}^{E}+\frac{1}{r}D^{S}\right) =Z\nabla
_{\partial _{r}}^{E}
\end{equation*}%
where $Z=-i\sigma \left( D_{N}\right) \left( \partial _{r}\right) $ is a
local bundle isomorphism and the map $D^{S}=0$.

In this case, there is no reason to modify the metric inside $B_{\varepsilon
}\left( \Sigma \right) $. The \textquotedblleft
desingularization\textquotedblright\ of $M$ along $\Sigma $ is the manifold
with boundary $\widetilde{M}=M\diagdown B_{\delta }\left( \Sigma \right) $
for some $0<\delta <\varepsilon $; the singular stratum is replaced by the
boundary $\partial \widetilde{M}=S_{\delta }\left( \Sigma \right) $, which
is a two-fold cover of $\Sigma $ and whose normal bundle is necessarily
oriented (via $\partial _{r}$). The double $M^{\prime }$ is identical to the
double of $\widetilde{M}$ along its boundary, and as in the previous section 
$M^{\prime }$ contains one less stratum.

\subsection{Discussion of operator product assumption\label%
{productAssumptionSection}}

We now explain specific situations that guarantee that, after a $G$%
-equivariant homotopy, $D^{+}$ may be written locally as a product of
operators as in (\ref{Dsplitting}) over the tubular neighborhhood $%
B_{\varepsilon }\left( \Sigma \right) $ over a singular stratum $\Sigma $.
This demonstrates that this assumption is not overly restrictive. In \cite%
{PrRi}, a large variety of examples of naturally defined transversal
operators similar to Dirac operators are explored and shown under suitable
conditions to provide all possible index classes of equivariant
transversally elliptic operators. The principal symbols of these operators
always have the form described below, and they almost always satisfy one of
the spin$^{c}$ conditions. We also emphasize that one might think that this
assumption places conditions on the curvature of the normal bundle $N\Sigma $%
; however, this is not the case for the following reason. The condition is
on the $G$-homotopy class of the principal transverse symbol of $D$. The
curvature of the bundle only effects the zeroth order part of the symbol.
For example, if $Y\rightarrow X$ is any fiber bundle over a spin$^{c}$
manifold $X$ with fiber $F$, then a Dirac-type operator $D$ on $Y$ has the
form $D=\partial _{X}\ast D_{F}+Z$, where $D_{F}$ is a family of fiberwise
Dirac-type operators, $\partial _{X}$ is the spin$^{c}$ Dirac operator on $X$%
, and $Z$ is a bundle endomorphism.

First, we show that if $D^{+}$ is $G$-homotopic to a transversal Dirac
operator at points of $\Sigma $, and if either $\Sigma $ is spin$^{c}$ or
its normal bundle $N\Sigma \rightarrow \Sigma $ is (fiberwise) spin$^{c}$,
then it has the desired form. Moreover, we also remark that certain
operators, like those resembling transversal de Rham operators, always
satisfy this splitting condition with no assumptions on $\Sigma $.

Let $N\mathcal{F}$ be normal bundle of the foliation$\mathcal{F}$ by $G$%
-orbits in $\Sigma $, and let $N\Sigma $ be the normal bundle of $\Sigma $
in $M$. Suppose that (after a $G$-homotopy) the principal transverse symbol
of $D^{+}$ (evaluated at $\xi \in N_{x}^{\ast }\mathcal{F}\oplus N_{x}^{\ast
}\Sigma $) at points $x\in \Sigma $ takes the form of a constant multiple of
Clifford multiplication. That is, we assume there is an action $c$ of $%
\mathbb{C}\mathrm{l}\left( N\mathcal{F}\oplus N\Sigma \right) $ on $E$ and a
Clifford connection $\nabla $ on $E$ such that the local expression for $D$
is given by the composition%
\begin{equation*}
\Gamma \left( E\right) \overset{\nabla }{\rightarrow }\Gamma \left( E\otimes
T^{\ast }M\right) \overset{\mathrm{proj}}{\rightarrow }\Gamma \left(
E\otimes \left( N^{\ast }\mathcal{F}\oplus N^{\ast }\Sigma \right) \right) 
\overset{\cong }{\rightarrow }\Gamma \left( E\otimes \left( N\mathcal{F}%
\oplus N\Sigma \right) \right) \overset{c}{\rightarrow }\Gamma \left(
E\right) .
\end{equation*}

The principal transverse symbol $\sigma \left( D^{+}\right) $ at $\xi
_{x}\in T_{x}^{\ast }\Sigma $ is%
\begin{equation*}
\sigma \left( D^{+}\right) \left( \xi _{x}\right) =\sum_{j=1}^{q}ic\left(
\xi _{x}\right) :E_{x}^{+}\rightarrow E_{x}^{-}
\end{equation*}%
Suppose $N\Sigma $ is spin$^{c}$; then there exists a vector bundle $%
S=S^{+}\oplus S^{-}\rightarrow \Sigma $ that is an irreducible
representation of $\mathbb{C}\mathrm{l}\left( N\Sigma \right) $ over each
point of $\Sigma $, and we let $E^{\Sigma }=\mathrm{End}_{\mathbb{C}\mathrm{l%
}\left( N\Sigma \right) }\left( E\right) $ and have%
\begin{equation*}
E\cong S\widehat{\otimes }E^{\Sigma }
\end{equation*}%
as a graded tensor product, such that the action of $\mathbb{C}\mathrm{l}%
\left( N\mathcal{F}\oplus N\Sigma \right) \cong \mathbb{C}\mathrm{l}\left(
N\Sigma \right) \widehat{\otimes }\mathbb{C}\mathrm{l}\left( N\mathcal{F}%
\right) $ (as a graded tensor product) on $E^{+}$ decomposes as%
\begin{equation*}
\left( 
\begin{array}{cc}
c\left( x\right) \otimes \mathbf{1} & -\mathbf{1}\otimes c\left( y\right)
^{\ast } \\ 
\mathbf{1}\otimes c\left( y\right) & c\left( x\right) ^{\ast }\otimes 
\mathbf{1}%
\end{array}%
\right) :\left( 
\begin{array}{c}
S^{+}\otimes E^{\Sigma +} \\ 
S^{-}\otimes E^{\Sigma -}%
\end{array}%
\right) \rightarrow \left( 
\begin{array}{c}
S^{-}\otimes E^{\Sigma +} \\ 
S^{+}\otimes E^{\Sigma -}%
\end{array}%
\right)
\end{equation*}%
(see \cite{ASi}, \cite{LM}). If we let the operator $\partial ^{N}$ denote
the spin$^{c}$ transversal Dirac operator on sections of $\pi ^{\ast
}S\rightarrow N\Sigma $, and let $D_{\Sigma }$ be the transversal Dirac
operator defined by the action of $\mathbb{C}\mathrm{l}\left( N\mathcal{F}%
\right) $ on $E^{\Sigma }$, then we have%
\begin{equation*}
D^{+}=\left( \partial ^{N}\ast D_{\Sigma }\right) ^{+}
\end{equation*}%
up to zero$^{\mathrm{th}}$ order terms (coming from curvature of the fiber).

The same argument works if instead we have that the bundle $N\mathcal{F}%
\rightarrow \Sigma $ is spin$^{c}$ in addition to the assumption that $D$
restricts to a transversal Dirac operator on the stratum $\Sigma $. In this
case there exists a spin$^{c}$ Dirac operator $\partial ^{\Sigma }$ on
sections of a complex spinor bundle over $\Sigma $ is transversally elliptic
to the orbit foliation $\mathcal{F}$, and we have a formula of the form%
\begin{equation*}
D^{+}=\left( D_{N}\ast \partial ^{\Sigma }\right) ^{+},
\end{equation*}%
again up to zeroth order terms.

Even if $N\Sigma \rightarrow \Sigma $ and $N\mathcal{F}\rightarrow \Sigma $
are not spin$^{c}$, many other first order operators have splittings as in
Equation (\ref{Dsplitting}). For example, if $D^{+}$ is a transversal de
Rham operator from even to odd forms, then $D^{+}$ is the product of de Rham
operators in the $N\Sigma $ and $N\mathcal{F}$ directions.

\section{Heat Kernel Splitting Formula \label{newMultiplicitiesSection}}

\subsection{Tensor Products of Representations}

If $H$ is a compact Lie group and $V=V_{1}\otimes V_{2}$ is a tensor product
of $H$-modules, and if $\left\{ \left[ \sigma _{j},W_{\sigma _{j}}\right]
:j\geq 0\right\} $ are the irreducible unitary representations of $H$,
counted only up to equivalence, then we define the integers $\gamma
_{ab}^{j} $ (Clebsch-Gordan coefficients) as%
\begin{eqnarray*}
\gamma _{ab}^{j} &=&\text{multiplicity of }\sigma _{j}\text{ in }\sigma
_{a}\otimes \sigma _{b} \\
&=&\mathrm{rank~\mathrm{Hom}}_{H}\left( W_{\sigma _{j}},W_{\sigma
_{a}}\otimes W_{\sigma _{b}}\right) ,
\end{eqnarray*}%
and 
\begin{equation*}
W_{\sigma _{a}}\otimes W_{\sigma _{b}}\cong \bigoplus_{j}\gamma
_{ab}^{j}W_{\sigma _{j}}.
\end{equation*}%
Quite often these integers are $1$ or $0$ (see \cite{Parth}, \cite{Kum}).
For example, if $\sigma _{0}$ denotes the trivial representation, then%
\begin{equation*}
\gamma _{ab}^{0}=\left\{ 
\begin{array}{ll}
1~~ & \text{if }\sigma _{b}=\sigma _{a}^{\ast } \\ 
0 & \text{otherwise}%
\end{array}%
\right.
\end{equation*}%
It is important to realize that for fixed $j$, as above there may be an
infinite number of $\gamma _{ab}^{j}$ that are nonzero. On the other hand,
if two of the numbers $a,b,j$ are fixed, then only a finite number of the $%
\gamma _{ab}^{j}$ are nonzero as the third index varies over all
possibilities.

\subsection{Decomposition near an orbit}

With notation as before, let $\Sigma $ be a minimal stratum of the $G$%
-action on the vector bundle $E\rightarrow M$, corresponding to an isotropy
type $\left[ H\right] $. We modify the metrics and bundles equivariantly so
that there exists $\varepsilon >0$ such that the tubular neighborhood $%
B_{\varepsilon }\Sigma $ of $\Sigma $ in $M$ is isometric to a ball of
radius $\varepsilon $ in the normal bundle $N\Sigma $ (endowed with the
natural product-like metric), and so that the $G$-equivariant bundle $E$
over $B_{\varepsilon }\Sigma $ is a pullback of the bundle $\left.
E\right\vert _{\Sigma }\rightarrow \Sigma $.

Choose $\Sigma _{\alpha }=G\Sigma _{\alpha }^{H}\subseteq \Sigma $ to be a
component of $\Sigma $ relative to $G$ (Definition \ref{componentRelGDefn}).
Note that $\Sigma $ is partitioned into a finite number of such components,
which is the number of components of $G\diagdown \Sigma $. We assume that%
\begin{equation*}
\left. E\right\vert _{\Sigma _{\alpha }}=E_{N}\widehat{\otimes }E_{\Sigma },
\end{equation*}%
the graded tensor product of two $G$-vector bundles over $\Sigma _{\alpha }$
as in Section \ref{BlowupDoubleMainSection}. We extend $B_{\varepsilon
}\Sigma $ to a sphere bundle $S\Sigma \overset{\pi }{\longrightarrow }\Sigma 
$ by modifying the metric near the boundary and doubling. We wish to compute
the local supertrace of a heat operator corresponding to a $G$-equivariant,
transversally elliptic operator, restricted to sections of type $\rho $, on $%
\left. S\Sigma \right\vert _{\Sigma _{\alpha }}$. Let $\mathcal{O}_{x}$ be
an orbit in $\Sigma _{\alpha }\subseteq \Sigma $, with $x\in \Sigma ^{H}$,
and let $S_{\mathcal{O}_{x}}\rightarrow \mathcal{O}_{x}$, respectively $B_{%
\mathcal{O}_{x}}\rightarrow \mathcal{O}_{x}$, be the restriction of the
Euclidean sphere bundle $S\Sigma \overset{\pi }{\longrightarrow }\Sigma $ ,
respectively $B_{\varepsilon }\Sigma \rightarrow \Sigma $. By the
constructions in Section \ref{BlowupDoubleMainSection}, the group action,
tensor product structure, and the equivariant operator naturally extend from 
$B_{\varepsilon }\Sigma $ to $S\Sigma $, and the local supertraces of the
heat operators on $B_{\varepsilon }\Sigma $ and $S\Sigma $ agree in a
neighborhood of $\Sigma $, up to $\mathcal{O}\left( e^{-c/t}\right) $ for
some $c>0$. The space of sections $\Gamma \left( \ S_{\mathcal{O}%
_{x}},E\right) $ satisfies%
\begin{align*}
\Gamma \left( S_{\mathcal{O}_{x}},\left. E\right\vert _{\mathcal{O}%
_{x}}\right) & \cong \Gamma \left( \mathcal{O}_{x},\left. F_{N}\widehat{%
\otimes }E_{\Sigma }\right\vert _{\mathcal{O}_{x}}\right) , \\
\Gamma \left( \left. S\Sigma \right\vert _{\Sigma _{\alpha }},\left.
E\right\vert _{\Sigma _{\alpha }}\right) & \cong \Gamma \left( \Sigma
_{\alpha },F_{N}\widehat{\otimes }E_{\Sigma }\right) ,
\end{align*}%
where $F_{N}\rightarrow \Sigma $ is the infinite dimensional equivariant
vector bundle whose fiber is 
\begin{equation*}
F_{N,z}=\Gamma \left( S\Sigma _{z},\left. E_{N}\right\vert _{S\Sigma
_{z}}\right) ,~z\in \Sigma _{\alpha }.
\end{equation*}%
Observe that $F_{N,gx}$ and $E_{\Sigma ,gx}$ are $G_{gx}$-modules, where $%
G_{gx}=gG_{x}g^{-1}\cong H$. If $F_{N}^{b}$ is a fine component of $F_{N}$
associated to $\left( \alpha ,\left[ \sigma \right] \right) $ as in
Definition \ref{fineComponentDefinition}, let

\begin{itemize}
\item $W^{b}=W_{\alpha ,\left[ \sigma \right] }$ be the canonical isotropy $%
G $-bundle adapted to $F_{N}^{b}$ (Definition \ref%
{canonicalIsotropyBundleDefinition}, Lemma \ref{AdaptedToAnyBundleLemma}),

\item $m_{b}$ be the associated multiplicity of each irreducible isotropy
representation present in $W^{b}$,

\item $d_{b}$ be the associated representation dimension of each irreducible
isotropy representation present in $W^{b}$ (and thus $F_{N}^{b}$), and

\item $n_{b}$ be the inequivalence number, the number of inequivalent
irreducible representation types present in $W^{b}$ (and thus $F_{N}^{b}$).

\item Note $\mathrm{rank}\left( W^{b}\right) =m_{b}d_{b}n_{b}$.
\end{itemize}

We remark that at this time there are an infinite number of such components $%
F_{N}^{b}$, but soon we will restrict to a finite number.

With notation as above, let $\left[ \rho ,V_{\rho }\right] $ be an
irreducible unitary $G$-representation. We will restrict a transversally
elliptic, $G$-equivariant differential operator to the space $\Gamma \left(
\Sigma ,F_{N}\widehat{\otimes }E_{\Sigma }\right) ^{\rho }$ of sections of
type $\rho $. By Frobenius reciprocity, 
\begin{eqnarray}
\mathrm{Hom}_{G}\left( V_{\rho },\Gamma \left( \mathcal{O}_{x},F_{N}\widehat{%
\otimes }E_{\Sigma }\right) \right) &\cong &\mathrm{Hom}_{H}\left( V_{%
\mathrm{Res}\left( \rho \right) },\left( F_{N}\widehat{\otimes }E_{\Sigma
}\right) _{x}\right)  \notag \\
&\cong &\bigoplus\limits_{b}\mathrm{Hom}_{H}\left( V_{\mathrm{Res}\left(
\rho \right) },\left( F_{N}^{b}\widehat{\otimes }E_{\Sigma }\right)
_{x}\right) .  \label{reciprocityShowsFiniteNumberBs}
\end{eqnarray}%
If $\left\{ \sigma _{a}\right\} $ is the finite set of $H$-representation
types present in $E_{\Sigma ,x}$ and $\left\{ \sigma _{j}\right\} $ is the
finite set of $H$-representation types present in $V_{\mathrm{Res}\left(
\rho \right) }$, then by the remarks in the previous subsection, the
Clebsch-Gordan coefficients $\gamma _{ap}^{j}$ are nonzero only for a finite
number of irreducible representation types $\left[ \sigma _{p}\right] $.
Therefore, only a finite number of $b$ satisfy

\begin{equation*}
\mathrm{Hom}_{H}\left( V_{\mathrm{Res}\left( \rho \right) },\left( F_{N}^{b}%
\widehat{\otimes }E_{\Sigma }\right) _{x}\right) \neq 0.
\end{equation*}%
Moreover, if we consider all such $b$ over every $x\in \Sigma _{\alpha
}\subseteq \Sigma $, only a finite number of $b$ would yield a nonzero $%
\mathrm{Hom}_{H}\left( V_{\mathrm{Res}\left( \rho \right) },\left( F_{N}^{b}%
\widehat{\otimes }E_{\Sigma }\right) _{x}\right) $. Let $B$ be this finite
set of indices, and thus we have

\begin{equation*}
\Gamma \left( \Sigma _{\alpha },F_{N}\widehat{\otimes }E_{\Sigma }\right)
^{\rho }=\bigoplus\limits_{b\in B}\Gamma \left( \Sigma _{\alpha },F_{N}^{b}%
\widehat{\otimes }E_{\Sigma }\right) ^{\rho }.
\end{equation*}%
In the equation above, we may of course replace the smooth sections with the 
$L^{2}$ sections.

Let $K\left( t,y_{x},y_{x}\right) ^{\rho },$ $y\in S\Sigma _{x}$ denote the
heat kernel of a $G$-equivariant transversally elliptic first-order
symmetric operator $D^{\ast }D$ of the form $D=D_{N}\ast D_{\Sigma }$ on $%
\Gamma \left( \Sigma _{\alpha },F_{N}\widehat{\otimes }E_{\Sigma }\right) $
; we further restrict this kernel to $\Gamma \left( B_{\mathcal{O}%
_{x}},E\right) ^{\rho }$ or $\Gamma \left( S_{\mathcal{O}_{x}},E\right)
^{\rho }$. We assume that $D_{N}$ is a family of first order elliptic
operators on fibers $F_{N,x}$, and we also assume $D_{N}$ is $G$%
-equivariant, with corresponding local equivariant fiberwise heat kernel $%
K_{N}$. We likewise assume that $D_{\Sigma }$ is a first order transversally
elliptic operator on sections of $E_{\Sigma }$ over $\Sigma $ with
equivariant heat kernel $K_{\Sigma }$.

Consider a small $\varepsilon $-tubular neighborhood $T_{\varepsilon
}\subset \Sigma _{\alpha }\subseteq \Sigma $, $T_{\varepsilon }\cong G\times
_{H}D_{\varepsilon }$ around a fixed orbit $\mathcal{O}_{x}$ in $\Sigma $,
where $x\in D_{\varepsilon }\subset \Sigma ^{H}$ is a ball transverse to $%
\mathcal{O}_{x}$. The space of sections $L^{2}\left( T_{\varepsilon
},F_{N}^{b}\widehat{\otimes }E_{\Sigma }\right) ^{\rho }$ decomposes using
Lemma \ref{SectionSplittingLemma} as%
\begin{eqnarray}
&&L^{2}\left( T_{\varepsilon },F_{N}^{b}\widehat{\otimes }E_{\Sigma }\right)
^{\rho }\otimes \bigoplus\limits_{j=1}^{n_{b}}\mathrm{Hom}_{H}\left( \left.
W_{j}^{b}\right\vert _{D_{\varepsilon }},\left. W_{j}^{b}\right\vert
_{D_{\varepsilon }}\right)  \notag \\
&\cong &\bigoplus\limits_{j=1}^{n_{b}}\mathrm{Hom}_{H}\left( \left.
W_{j}^{b}\right\vert _{D_{\varepsilon }},\left. F_{N}^{b}\right\vert
_{D_{\varepsilon }}\right) \widehat{\otimes }L^{2}\left( T_{\varepsilon },%
\widetilde{W}_{j}^{b}\otimes E_{\Sigma }\right) ^{\rho }  \notag \\
&\cong &\bigoplus\limits_{j=1}^{n_{b}}\mathrm{Hom}_{H}\left( \left.
W_{j}^{b}\right\vert _{D_{\varepsilon }},W_{\sigma _{j}}\right) \otimes 
\mathrm{Hom}_{H}\left( W_{\sigma _{j}}~,\left. F_{N}^{b}\right\vert
_{D_{\varepsilon }}\right) \widehat{\otimes }L^{2}\left( T_{\varepsilon },%
\widetilde{W}_{j}^{b}\otimes E_{\Sigma }\right) ^{\rho }
\label{tubeSectionDecomposition}
\end{eqnarray}%
with the isomorphism given by evaluation, and then $D\otimes \mathbf{1}$
acts by $\mathbf{1}\otimes D_{N}\widehat{\otimes }\left( \mathbf{1}\otimes
D_{\Sigma }\right) ^{\rho }$. Here, corresponding to the
normalizer-conjugate representations $\left[ \sigma _{j},W_{\sigma _{j}}%
\right] $ is the isotypical decomposition $\left. W^{b}\right\vert
_{D_{\varepsilon }}=\bigoplus\limits_{j}W_{j}^{b}$ of the canonical isotropy 
$G$-bundle $W^{b}$ , which we may define over the small transversal
neighborhood $D_{\varepsilon }$, and $\widetilde{W}_{j}^{b}\rightarrow
T_{\varepsilon }$ is the bundle $G\times _{H}W_{j}^{b}\rightarrow G\times
_{H}D_{\varepsilon }\cong T_{\varepsilon }$. Note that $\left.
W^{b}\right\vert _{T_{\varepsilon }}=\bigoplus\limits_{j}\widetilde{W}%
_{j}^{b}$, and the multiplicity of the irreducible $H$-representation
actually present in each $\widetilde{W}_{j,p}^{b}$ is $m_{b}$, with $p\in
\Sigma ^{H}\cap T_{\varepsilon }$. Each bundle $\mathrm{Hom}_{H}\left(
\left. W_{j}^{b}\right\vert _{D_{\varepsilon }},\left. W_{j}^{b}\right\vert
_{D_{\varepsilon }}\right) \cong \mathbb{C}^{m_{b}^{2}}\times D_{\varepsilon
}$ is trivial.

Given an operator $P$, let $E_{t}\left( P\right) =\exp \left( -tP^{\ast
}P\right) $ denote the associated heat operator (defined through context).
We wish to calculate the local supertrace of $E_{t}\left( D\right) =\exp
\left( -tD^{\ast }D\right) $ on $L^{2}\left( T_{\varepsilon },F_{N}^{b}%
\widehat{\otimes }E_{\Sigma }\right) ^{\rho }$, which is the same as $\frac{1%
}{m_{b}^{2}n_{b}}$ times the local supertrace of $E_{t}\left( D\right)
\otimes \mathbf{1=}E_{t}\left( D\otimes \mathbf{1}\right) $ on $L^{2}\left(
T_{\varepsilon },F_{N}^{b}\widehat{\otimes }E_{\Sigma }\right) ^{\rho
}\otimes \bigoplus\limits_{j=1}^{n_{b}}\mathrm{Hom}_{H}\left( \left.
W_{j}^{b}\right\vert _{D_{\varepsilon }},\left. W_{j}^{b}\right\vert
_{D_{\varepsilon }}\right) $. According to the operator product assumption
and decomposition (\ref{tubeSectionDecomposition}), we then have the
pointwise supertrace formula%
\begin{eqnarray*}
\mathrm{\mathrm{str}}\left( E_{t}\left( D\right) ^{b,\rho }\text{ on }%
T_{\varepsilon }\right) _{y} &=&\frac{1}{m_{b}^{2}n_{b}}\sum_{j=1}^{n_{b}}%
\mathrm{tr}\left( \mathbf{1}^{\sigma _{j}}\right) _{x}\mathrm{str}\left(
E_{t}\left( D_{N}^{b,j}\right) \right) _{x}\mathrm{str}\left( E_{t}\left( 
\mathbf{1}^{b,j}\otimes D_{\Sigma }\right) ^{\rho }\right) _{y}, \\
x &\in &D_{\varepsilon },y\in \left[ \left( x,p\right) \right] \in
T_{\varepsilon },
\end{eqnarray*}%
since supertraces are multiplicative on graded tensor products and where $%
\mathbf{1}^{\sigma _{j}}$ is the identity bundle map on $\mathrm{Hom}%
_{H}\left( \left. W_{j}^{b}\right\vert _{D_{\varepsilon }},W_{\sigma
_{j}}\right) \rightarrow D_{\varepsilon }$ and $\mathbf{1}^{b,j}$ is the
identity bundle map on $\widetilde{W}_{j}^{b}\rightarrow T_{\varepsilon }$.
Here $D_{N}^{b,j}$ means the restriction of $D_{N}$ to $\mathrm{Hom}%
_{H}\left( W_{\sigma _{j}}~,F_{N}^{b}\right) $, and so the supertrace $%
\mathrm{str}\left( E_{t}\left( D_{N}^{b,j}\right) \right) _{x}$ is the same
as $\frac{1}{d_{b}}\mathrm{str}\left( E_{t}\left( D_{N}^{\sigma _{j}}\right)
\right) _{x}$, where $D_{N}^{\sigma _{j}}$ is the restriction of $D_{N}$ to 
\begin{equation*}
F_{N,x}^{\sigma _{j}}=\Gamma \left( S\Sigma _{x},\left. E_{N}\right\vert
_{S\Sigma _{x}}\right) ^{\sigma _{j}},x\in D_{\varepsilon }.
\end{equation*}%
Then%
\begin{eqnarray*}
\mathrm{\mathrm{str}}\left( E_{t}\left( D\right) ^{b,\rho }\text{ on }%
T_{\varepsilon }\right) _{y} &=&\frac{1}{m_{b}n_{b}d_{b}}\sum_{j=1}^{n_{b}}%
\mathrm{str}\left( E_{t}\left( D_{N}^{\sigma _{j}}\right) \right) _{x}%
\mathrm{str}\left( E_{t}\left( \mathbf{1}^{b,j}\otimes D_{\Sigma }\right)
^{\rho }\right) _{y}, \\
x &\in &D_{\varepsilon },y\in \left[ \left( x,p\right) \right] \in
T_{\varepsilon }.
\end{eqnarray*}

By (\ref{localSupertraceEta}) and (\ref{localSupertraceEtaCodim1}), the
integral of the pointwise supertrace $\mathrm{str}\left( E_{t}\left(
D_{N}^{\sigma _{j}}\right) \right) _{p}\left( z_{p},z_{p}\right) $ over a
small normal ball $z_{p}\in B_{\varepsilon }^{N}$ containing $p$ in $%
S_{p}\Sigma $ is the same as 
\begin{eqnarray*}
\int_{z_{p}\in B_{\varepsilon }^{N}}\mathrm{str}\left( E_{t}\left(
D_{N}^{\sigma _{j}}\right) \right) _{p}\left( z_{p},z_{p}\right) \sim
\int_{z_{p}\in \widetilde{B_{\varepsilon }^{N}}}\mathrm{str}\left(
E_{t}\left( \widetilde{D_{N}^{\sigma _{j}}}\right) \right) _{p}\left(
z_{p},z_{p}\right) \\
+\frac{1}{2}\left( -\eta \left( D^{S+,b,j}\right) +h\left( D^{S+,b,j}\right)
\right) ,
\end{eqnarray*}%
with the tilde corresponding to the equivariant desingularization of the
origin as in Section \ref{BlowupDoubleMainSection}.

We now discuss a general situation. Let $u$ be an eigenvector of a $G$%
-equivariant fiberwise operator $L$ on a fiber $E_{x}$ of an equivariant
bundle over a $G$-manifold \ $X$ with single orbit type $G\diagup H$. Assume
that $x\in X^{H}$, and $n\in N\left( H\right) $. Then $nu$ is an
eigensection of $L$ over $E_{nx}$ with the same eigenvalue. Moreover, if $u$
is an element of an $H$-irreducible subspace of type $\sigma $, then $nu$ is
an element of an irreducible subspace of type $\sigma ^{n}$. This implies
that the Schwarz kernel $K_{L}$ of $L$ at $x$ is mapped to the kernel $K_{L}$
of $L$ at $nx$ by conjugation by $n$, and thus the kernel $K_{L,x}^{\sigma }$
restricted to sections of type $\sigma $ at $x$ is conjugate to the kernel $%
K_{L,nx}^{\sigma ^{n}}$ at $nx$ (see also equations (\ref%
{heatKernelInvariance}), (\ref{traceHeatKernelInvariance})). Moreover, the
kernels $K_{L}$, $K_{L}^{\sigma }$, $K_{L}^{\sigma ^{n}}$ vary smoothly as a
function of $x\in X^{H}$, and in fact any fiberwise spectral invariant of $L$
must vary continuously with $x\in X^{H}$.

Next, consider the special case where $L:\Gamma \left( X,E\right)
\rightarrow \Gamma \left( X,E\right) $ has integer eigenvalues (as in the
case of the operators $D^{S+}$ -- see Corollary \ref%
{DSspectrumCorollaryHversion} and Remark \ref{stabilityRemarks}). Then the
eigenvalues of $L$, $L^{\sigma }$, and $L^{\sigma ^{n}}$ must be locally
constant on $X^{H}$. This implies that if $n\in N\left( H\right) $ maps a
connected component $X_{\alpha }^{H}\subset X^{H}$ to itself, then all
spectral invariants of $L^{\sigma }$ and $L^{\sigma ^{n}}$ are identical and
constant, when restricted to $X_{\alpha }^{H}$. Now, let $U$ be an open
subset of $X_{\alpha }^{H}$, and let $E^{b}$ be a fine component of $E$. If $%
\left[ \sigma \right] ,\left[ \sigma ^{\prime }\right] \in \widehat{H}$ are
different representation types present in a specific $\left.
E^{b}\right\vert _{U}$, then a consequence of the argument above is that all
spectral invariants of $\left. L^{\sigma }\right\vert _{U}$ and $\left.
L^{\sigma ^{\prime }}\right\vert _{U}$ are identical and constant.

As a result, in the displayed formula above, $\frac{1}{2}\left( -\eta \left(
D^{S+,b,j}\right) +h\left( D^{S+,b,j}\right) \right) $ is independent of $j$
and is a constant on each connected component $\Sigma _{\alpha }^{H}$ of $%
\Sigma ^{H}$ (and thus on $\Sigma _{\alpha }=G\Sigma _{\alpha }^{H}\subset
\Sigma $). In particular, 
\begin{equation*}
\frac{1}{2}\left( -\eta \left( D^{S+,b,j}\right) +h\left( D^{S+,b,j}\right)
\right) =\frac{1}{2n_{b}}\left( -\eta \left( D^{S+,b}\right) +h\left(
D^{S+,b}\right) \right) ,
\end{equation*}%
where $D^{S+,b}$ refers to the restriction of $D^{S+}$ to the sections in
the fine component $F_{N}^{b}$. We now integrate over $p\in T_{\varepsilon }$
to get

\begin{eqnarray*}
&&\int_{B_{\varepsilon }^{N}\left( T_{\varepsilon }\right) }\mathrm{\mathrm{%
str}}\left( E_{t}\left( D\right) ^{b,\rho }\right) \left( z_{p},z_{p}\right)
\\
&=&\frac{1}{m_{b}n_{b}d_{b}}\sum_{j=1}^{n_{b}}\int_{T_{\varepsilon
}}\int_{z_{p}\in B_{\varepsilon }^{N}}\mathrm{str}\left( E_{t}\left(
D_{N}^{\sigma _{j}}\right) \right) _{p}\left( z_{p},z_{p}\right) \mathrm{str}%
\left( E_{t}\left( \mathbf{1}^{b,j}\otimes D_{\Sigma }\right) ^{\rho
}\right) \left( p,p\right)
\end{eqnarray*}
\begin{eqnarray*}
&\sim &\frac{1}{m_{b}n_{b}d_{b}}\sum_{j=1}^{n_{b}}\int_{T_{\varepsilon
}}\int_{z_{p}\in \widetilde{B_{\varepsilon }^{N}}}\mathrm{str}\left(
E_{t}\left( \widetilde{D_{N}^{\sigma _{j}}}\right) \right) _{p}\left(
z_{p},z_{p}\right) \mathrm{str}\left( E_{t}\left( \mathbf{1}^{b,j}\otimes
D_{\Sigma }\right) ^{\rho }\right) \left( p,p\right) \\
&&+\frac{1}{m_{b}n_{b}d_{b}}\int_{T_{\varepsilon }}\frac{1}{2n_{b}}\left(
-\eta \left( D^{S+,b}\right) +h\left( D^{S+,b}\right) \right)
\sum_{j=1}^{n_{b}}\mathrm{str}\left( E_{t}\left( \mathbf{1}^{b,j}\otimes
D_{\Sigma }\right) ^{\rho }\right) \left( p,p\right) \\
&=&\int_{\widetilde{B_{\varepsilon }^{N}\left( T_{\varepsilon }\right) }}%
\mathrm{\mathrm{str}}\left( E_{t}\left( \widetilde{D}\right) ^{b,\rho }\text{
}\right) \left( z_{p},z_{p}\right)
\end{eqnarray*}
\begin{eqnarray*}
&&+\frac{1}{2m_{b}n_{b}^{2}d_{b}}\left( -\eta \left( D^{S+,b}\right)
+h\left( D^{S+,b}\right) \right) \int_{T_{\varepsilon }}\mathrm{str}\left(
E_{t}\left( \mathbf{1}^{b}\otimes D_{\Sigma }\right) ^{\rho }\right) \left(
p,p\right) ,
\end{eqnarray*}%
where $\mathbf{1}^{b}\otimes D_{\Sigma }$ is the differential operator on $%
W^{b}\otimes E_{\Sigma }$ over the $G$-manifold $\Sigma _{\alpha }$.

Thus, in calculating the small $t$ asymptotics of $\int_{B_{\varepsilon
}\left( U\right) }\mathrm{str}K\left( t,z_{p},z_{p}\right) ^{\rho }=\int 
\mathrm{\mathrm{str}}\left( E_{t}\left( D\right) ^{\rho }\right) \left(
z_{p},z_{p}\right) $ with $E_{t}\left( D\right) ^{\rho }=\exp \left(
-tD^{\ast }D\right) ^{\rho }$, it suffices to calculate the right hand side
of the formula above over the region $B_{\varepsilon }$ and sum over fine
components $b\in B$, using the heat kernel coming from the blown up
manifold. Thus, by integrating over any open saturated subset $U$ of $\Sigma
_{\alpha }\subseteq \Sigma $, we conclude that as $t\rightarrow 0$, 
\begin{gather}
\int_{B_{\varepsilon }\left( U\right) }\mathrm{str}K\left(
t,z_{p},z_{p}\right) ^{\rho }\sim \int_{\widetilde{B_{\varepsilon }\left(
U\right) }}\mathrm{str}K\left( t,z_{p},z_{p}\right) ^{\rho }  \notag \\
+\sum_{b}\frac{1}{2n_{b}\mathrm{rank}\left( W^{b}\right) }\left( -\eta
\left( D^{S+,b}\right) +h\left( D^{S+,b}\right) \right) \int_{p\in U}\mathrm{%
str}K_{\Sigma }^{b}\left( t,p,p\right) ^{\rho },  \label{heatSplitting}
\end{gather}%
with $\mathrm{str}K_{\Sigma }^{b}\left( t,p,p\right) ^{\rho }=\mathrm{str}%
\left( E_{t}\left( \mathbf{1}^{b}\otimes D_{\Sigma }\right) ^{\rho }\right)
\left( p,p\right) $ is the local heat supertrace corresponding to the
operator $\mathbf{1}^{b}\otimes D_{\Sigma }$ on $\Gamma \left(
U,W^{b}\otimes E_{\Sigma }\right) ^{\rho }$.

\section{The Equivariant Index Theorem\label{EquivariantIndexSection}}

\subsection{Reduction formula for the heat supertrace\label%
{ReductionHeatSupertraceSubsection}}

As before, let $E$ be a graded Hermitian vector bundle over a closed
Riemannian manifold $M$, such that a compact Lie group $G$ acts on $\left(
M,E\right) $ by isometries. Let $D=D^{+}:\Gamma \left( M,E^{+}\right)
\rightarrow \Gamma \left( M,E^{-}\right) $ be a first order, transversally
elliptic, $G$-equivariant differential operator. Let $\Sigma $ be a minimal
stratum of the $G$-manifold $M$. We modify the metrics and bundles
equivariantly so that there exists $\varepsilon >0$ such that the tubular
neighborhood $B_{\varepsilon }\Sigma $ of $\Sigma $ in $M$ is isometric to a
ball of radius $\varepsilon $ in the normal bundle $N\Sigma $, and so that
the $G$-equivariant bundle $E$ over $B_{\varepsilon }\left( \Sigma \right) $
is a pullback of the bundle $\left. E\right\vert _{\Sigma }\rightarrow
\Sigma $. We assume that $D^{+}$ can be written on each saturated open set $%
U $ contained in $B_{\varepsilon }\left( \Sigma \right) $ as the product 
\begin{equation*}
D^{+}=\left( D_{N}\ast D_{\Sigma }\right) ^{+},
\end{equation*}%
where $D_{\Sigma }\ $is a transversally elliptic, $G$-equivariant, first
order operator on the stratum $\Sigma $, and $D_{N}$ is a $G$-equivariant,
family of elliptic first order operators on the fibers of $B_{\varepsilon
}\left( \Sigma \right) $. In fact, we only require that $D^{+}$ is of this
form after an equivariant homotopy. Further, we may assume that $D_{N}$ is a
first order operator with coefficients frozen at $\Sigma $ so its
restriction to a normal exponential neighborhood $B_{\varepsilon }\Sigma
_{x} $, $x\in \Sigma $, is a constant coefficient operator. If $r$ is the
distance from $\Sigma $, we can write $D_{N}$ in polar coordinates as%
\begin{equation*}
D_{N}=Z\left( \nabla _{\partial _{r}}^{E}+\frac{1}{r}D^{S}\right)
\end{equation*}%
where $Z=-i\sigma \left( D_{N}\right) \left( \partial _{r}\right) $ is a
local bundle isomorphism and the map $D^{S}$ is a purely first order
operator that differentiates in the unit normal bundle directions tangent to 
$S_{x}\Sigma $. Note that we allow this to make sense if $\Sigma $ has
codimension one by setting $D^{S}=0$ and $\partial _{r}$ to be the
inward-pointing vector.

Let $\rho :G\rightarrow U\left( V_{\rho }\right) $ be an irreducible
representation. Let $U$ be a saturated (ie $G$-invariant) open subset of $M$
such that $G\diagdown U$ is connected. From Equation (\ref%
{heatKernelDecompStrata}), we wish to compute the quantity 
\begin{equation*}
I_{U}=\int_{x\in \Sigma \cap U}\int_{z_{x}\in B_{\varepsilon }\left( \Sigma
\right) _{x}\cap U}\mathrm{\mathrm{str}\,}\left( K\left( t,\left(
x,z_{x}\right) ,\left( x,z_{x}\right) \right) ^{\rho }\right)
\,\,\,\left\vert dz_{x}\right\vert ~\left\vert dx\right\vert ,
\end{equation*}%
where we choose coordinates $w=\left( x,z_{x}\right) $ for the point $\exp
_{x}^{\bot }\left( z_{x}\right) $, where $x\in \Sigma $, $z_{x}\in
B_{\varepsilon }=B_{\varepsilon }\left( N_{x}\left( \Sigma \right) ,0\right) 
$, and $\exp _{x}^{\bot }$ is the normal exponential map. In the expression
above, $\left\vert dz_{x}\right\vert $ is the Euclidean density on $%
N_{x}\left( \Sigma \right) $, and $\left\vert dx\right\vert $ is the
Riemannian density on $\Sigma $. Recall that $K\left( t,\left(
x,z_{x}\right) ,\left( x,z_{x}\right) \right) ^{\rho }$ is a map from $%
E_{\left( x,z_{x}\right) }^{\mathrm{Res}\left( \rho \right) }=E_{x}^{\mathrm{%
Res}\left( \rho \right) }$ to itself.

First, observe that the small $t$ asymptotics of this integral over $%
B_{\varepsilon }\left( \Sigma \right) \cap U$ may be computed using a heat
kernel for another differential on a neighborhood of a different manifold,
as long as all the local data (differential operator, vector bundles,
metrics) are the same. In particular we may use the manifold $S\Sigma $ as
in the previous section. Using Equation (\ref{heatSplitting}),

\begin{eqnarray*}
I_{i,U} &\sim &\int_{\widetilde{B_{\varepsilon }\left( \Sigma \right) \cap U}%
}\,\mathrm{\mathrm{str}\,}\left( \widetilde{K}\left( t,\widetilde{x},%
\widetilde{x}\right) ^{\rho }\right) \left\vert d\widetilde{x}\right\vert \\
&&-\frac{1}{2}\sum_{b\in B}\frac{1}{n_{b}\mathrm{rank~}W^{b}}\left( \eta
\left( D^{S+,b}\right) -h\left( D^{S+,b}\right) \right) \int_{x\in \Sigma
\cap U}\mathrm{str}\left( K_{\Sigma }^{b}\left( t,x,x\right) ^{\rho }\right)
\,~\left\vert dx\right\vert ~,
\end{eqnarray*}%
where $\sim $ means equal up to $\mathcal{O}\left( t^{k}\right) $ for large $%
k$, where $\widetilde{B_{\varepsilon }\left( \Sigma \right) \cap U}$ is the
fundamental domain in the desingularization of $B_{\varepsilon }\left(
\Sigma \right) \cap U$ and where $\widetilde{K}$ is the suitably modified
heat kernel. We note that the equation above remains true if $\Sigma $ is of
codimension $1$, because in this case $D^{S+,b}=0^{b}$, and we use the
definitions (\ref{equivariant eta definition dim zero}).

Let $\widetilde{M}$ be the equivariant desingularization of $M$ along $%
\Sigma $, and let $\widetilde{E^{\pm }}$ be the equivariant bundles over $%
\widetilde{M}$ induced from $E^{\pm }$ with transversally elliptic operators 
$\widetilde{D^{\pm }}$ (see Section \ref{BlowUpSection}). Let $\widetilde{K}%
^{\pm }\left( t,\cdot ,\cdot \right) ^{\beta }$ denote the equivariant heat
kernels corresponding to the extensions of $\widetilde{D^{\mp }}\widetilde{%
D^{\pm }}+C-\lambda _{\beta }$ to the double of $\widetilde{M}$ restricted
to sections of type $\beta $ (with reversed orientations of the induced
bundles $\widetilde{E^{\pm }}$ in the usual way). Then the heat kernel
supertrace formula becomes 
\begin{eqnarray*}
&&\int_{x\in U}\mathrm{\mathrm{str}\,}\left( K\left( t,x,x\right) ^{\rho
}\right) ~\left\vert dx\right\vert \sim \int_{\widetilde{U}}\,\mathrm{%
\mathrm{str}\,}\left( \widetilde{K}\left( t,\widetilde{x},\widetilde{x}%
\right) ^{\rho }\right) \left\vert d\widetilde{x}\right\vert \\
&&+\frac{1}{2}\sum_{b\in B}\frac{1}{n_{b}\mathrm{rank~}W^{b}}\left( -\eta
\left( D^{S+,b}\right) +h\left( D^{S+,b}\right) \right) \int_{G\diagdown
\left( U\cap \Sigma \right) }\mathrm{str}\left( \overline{K_{\Sigma }^{b}}%
\left( t,\overline{x},\overline{x}\right) ^{\rho }\right) \,~\left\vert d%
\overline{x}\right\vert ~.
\end{eqnarray*}%
The heat kernel $\overline{K_{\Sigma }^{b}}$ is that corresponding to the
operator $\left( \mathbf{1}\otimes D_{\Sigma }\right) ^{\rho }$ induced by
the operator $\mathbf{1}\otimes D_{\Sigma }$ on $\Gamma \left( \Sigma
,W^{b}\otimes E_{\Sigma }\right) $ on the quotient $G\diagdown \left( U\cap
\Sigma \right) $ as described in Section \ref{OneIsotropyTypeSection}.

We assemble the results above into the following theorem.

\begin{theorem}
(Reduction Formula for the Equivariant Supertrace)\label{supertraceReduction}
Let $E$ be a graded Hermitian vector bundle over a closed Riemannian
manifold $M$, such that a compact Lie group $G$ acts on $\left( M,E\right) $
by isometries. Let $\Sigma $ denote a minimal stratum. Let $D^{+}:\Gamma
\left( M,E^{+}\right) \rightarrow \Gamma \left( M,E^{-}\right) $ be a first
order, transversally elliptic, $G$-equivariant differential operator. We
modify the metrics and bundles equivariantly so that there exists $%
\varepsilon >0$ such that the tubular neighborhood $B_{\varepsilon }\Sigma $
of $\Sigma $ in $M$ is isometric to a ball of radius $\varepsilon $ in the
normal bundle $N\Sigma $, and so that the $G$-equivariant bundle $E$ over $%
B_{\varepsilon }\left( \Sigma \right) $ is a pullback of the bundle $\left.
E\right\vert _{\Sigma }\rightarrow \Sigma $. We assume that near $\Sigma $, $%
D$ is $G$-homotopic to the product 
\begin{equation*}
D^{+}=\left( D_{N}\ast D_{\Sigma }\right) ^{+},
\end{equation*}%
where $D_{\Sigma }\ $is a global transversally elliptic, $G$-equivariant,
first order operator on the stratum $\Sigma $, and $D_{N}$ is a $G$%
-equivariant, first order operator on $B_{\varepsilon }\Sigma $ that is
elliptic and has constant coefficients on the fibers. If $r$ is the distance
from $\Sigma $, we write $D_{N}$ in polar coordinates as%
\begin{equation*}
D_{N}=Z\left( \nabla _{\partial _{r}}^{E}+\frac{1}{r}D^{S}\right)
\end{equation*}

where $Z$ is a local bundle isomorphism and the map $D^{S}$ is a purely
first order operator that differentiates in the unit normal bundle
directions tangent to $S_{x}\Sigma $. Let $U$ be a saturated (ie $G$%
-invariant) open subset of $M$ such that $G\diagdown U$ is connected. Then
as $t\rightarrow 0$, 
\begin{eqnarray*}
&&\int_{x\in U}\mathrm{\mathrm{str}\,}\left( K\left( t,x,x\right) ^{\rho
}\right) ~\left\vert dx\right\vert \sim \int_{\widetilde{U}}\,\mathrm{%
\mathrm{str}\,}\left( \widetilde{K}\left( t,\widetilde{x},\widetilde{x}%
\right) ^{\rho }\right) \left\vert d\widetilde{x}\right\vert \\
&&+\frac{1}{2}\sum_{b\in B}\frac{1}{n_{b}\mathrm{rank~}W^{b}}\left( -\eta
\left( D^{S+,b}\right) +h\left( D^{S+,b}\right) \right) \int_{G\diagdown
\left( U\cap \Sigma \right) }\mathrm{str}\left( \overline{K_{\Sigma }^{b}}%
\left( t,\overline{x},\overline{x}\right) ^{\rho }\right) \,~\left\vert d%
\overline{x}\right\vert _{i}~.
\end{eqnarray*}%
Here, the superscript $W^{b}\rightarrow \Sigma _{\alpha }$ is a canonical
isotropy $G$-bundle over the component $\Sigma _{\alpha }$ relative to $G$
that contains $U$, $K_{\Sigma }^{b}$ refers to the equivariant heat kernel
corresponding to a twist of $D_{\Sigma }$ by $W^{b}\rightarrow \Sigma
_{\alpha }$. Since $K_{\Sigma }^{b}$ is an equivariant heat kernel on the
component $\Sigma _{\alpha }$, it induces a heat kernel $\overline{K_{\Sigma
}^{b}}=\left( K_{\Sigma }^{b}\right) ^{\rho }$ on the quotient $G\diagdown
\Sigma _{\alpha }$. Also, $\widetilde{U}$ is the equivariant
desingularization of $U$ along the stratum $\Sigma $, and the kernel $%
\widetilde{K}$ is the corresponding heat kernel of the double.
\end{theorem}

We refer the reader to Definitions \ref{componentRelGDefn} and \ref%
{canonicalIsotropyBundleDefinition}. Note that only a finite number ($b\in B$%
) of $W^{b}$ result in heat kernels $\overline{K_{\Sigma }^{b}}\left( t,%
\overline{x},\overline{x}\right) ^{\rho }$ that are not identically zero, by
the discussion following (\ref{reciprocityShowsFiniteNumberBs}).

\subsection{The main theorem\label{MainTheoremSubsection}}

To evaluate $\mathrm{ind}^{\rho }\left( D\right) $ as in Equation (\ref%
{heatKernelDecompStrata}), we apply Theorem \ref{supertraceReduction}
repeatedly, starting with a minimal stratum and then applying to each double
of the equivariant desingularization. After all the strata are blown up and
doubled, all of the resulting manifolds have a single stratum, and we may
apply the results of Section \ref{OneIsotropyTypeSection}. We obtain the
following result. In what follows, if $U$ denotes an open subset of a
stratum of the action of $G$ on $M$, $U^{\prime }$ denotes the equivariant
desingularization of $U$, and $\widetilde{U}$ denotes the fundamental domain
of $U$ inside $U^{\prime }$, as in Section \ref{BlowupDoubleMainSection}. We
also refer the reader to Definitions \ref{componentRelGDefn} and \ref%
{canonicalIsotropyBundleDefinition}.

\begin{theorem}
(Equivariant Index Theorem) \label{MainTheorem}Let $M_{0}$ be the principal
stratum of the action of a compact Lie group $G$ on the closed Riemannian $M$%
, and let $\Sigma _{\alpha _{1}}$,...,$\Sigma _{\alpha _{r}}$ denote all the
components of all singular strata relative to $G$. Let $E\rightarrow M$ be a
Hermitian vector bundle on which $G$ acts by isometries. Let $D:\Gamma
\left( M,E^{+}\right) \rightarrow \Gamma \left( M,E^{-}\right) $ be a first
order, transversally elliptic, $G$-equivariant differential operator. We
assume that near each $\Sigma _{\alpha _{j}}$, $D$ is $G$-homotopic to the
product $D_{N}\ast D^{\alpha _{j}}$, where $D_{N}$ is a $G$-equivariant,
first order differential operator on $B_{\varepsilon }\Sigma $ that is
elliptic and has constant coefficients on the fibers and $D^{\alpha _{j}}\ $%
is a global transversally elliptic, $G$-equivariant, first order operator on
the $\Sigma _{\alpha _{j}}$. In polar coordinates 
\begin{equation*}
D_{N}=Z_{j}\left( \nabla _{\partial _{r}}^{E}+\frac{1}{r}D_{j}^{S}\right) ~,
\end{equation*}%
where $r$ is the distance from $\Sigma _{\alpha _{j}}$, where $Z_{j}$ is a
local bundle isometry (dependent on the spherical parameter), the map $%
D_{j}^{S}$ is a family of purely first order operators that differentiates
in directions tangent to the unit normal bundle of $\Sigma _{j}$. Then the
equivariant index $\mathrm{ind}^{\rho }\left( D\right) $ satisfies 
\begin{eqnarray*}
\mathrm{ind}^{\rho }\left( D\right) &=&\int_{G\diagdown \widetilde{M_{0}}%
}A_{0}^{\rho }\left( x\right) ~\widetilde{\left\vert dx\right\vert }%
~+\sum_{j=1}^{r}\beta \left( \Sigma _{\alpha _{j}}\right) ~, \\
\beta \left( \Sigma _{\alpha _{j}}\right) &=&\frac{1}{2\dim V_{\rho }}%
\sum_{b\in B}\frac{1}{n_{b}\mathrm{rank~}W^{b}}\left( -\eta \left(
D_{j}^{S+,b}\right) +h\left( D_{j}^{S+,b}\right) \right) \int_{G\diagdown 
\widetilde{\Sigma _{\alpha _{j}}}}A_{j,b}^{\rho }\left( x\right) ~\widetilde{%
\left\vert dx\right\vert }~,
\end{eqnarray*}%
where

\begin{enumerate}
\item $A_{0}^{\rho }\left( x\right) $ is the Atiyah-Singer integrand (\cite%
{ABP}), the local supertrace of the ordinary heat kernel associated to the
elliptic operator induced from $D^{\prime }$ (blown-up and doubled from $D$)
on the quotient $M_{0}^{\prime }\diagup G$, where the bundle $E$ is replaced
by the bundle $\mathcal{E}_{\rho }$ (see Section \ref{OneIsotropyTypeSection}%
).

\item Similarly, $A_{i,b}^{\rho }$ is the local supertrace of the ordinary
heat kernel associated to the elliptic operator induced from $\left( \mathbf{%
1}\otimes D^{\alpha _{j}}\right) ^{\prime }$ (blown-up and doubled from $%
\mathbf{1}\otimes D^{\alpha _{j}}$, the twist of $D^{\alpha _{j}}$ by the
canonical isotropy bundle $W^{b}\rightarrow \Sigma _{\alpha _{j}}$ ) on the
quotient $\Sigma _{\alpha _{j}}^{\prime }\diagup G$, where the bundle is
replaced by the space of sections of type $\rho $ over each orbit.

\item $\eta \left( D_{j}^{S+,b}\right) $ is the eta invariant of the
operator $D_{j}^{S+}$ induced on any unit normal sphere $S_{x}\Sigma
_{\alpha _{j}}$, restricted to sections of isotropy representation types in $%
W_{x}^{b}$, defined in (\ref{equivariant eta definition}) and in (\ref%
{equivariant eta definition dim zero}) for the codimension one case when $%
D_{j}^{S+,b}=0_{j}^{+,b}$. This is constant on $\Sigma _{\alpha _{j}}$.

\item $h\left( D_{j}^{S+,b}\right) $ is the dimension of the kernel of $%
D_{j}^{S+,b}$, restricted to sections of isotropy representation types in $%
W_{x}^{b}$, again constant on on $\Sigma _{\alpha _{j}}$.

\item $n_{b}$ is the number of different inequivalent $G_{x}$-representation
types present in each $W_{x}^{b}$, $x\in \Sigma _{\alpha _{j}}$.
\end{enumerate}
\end{theorem}

\begin{remark}
Note that only a finite number ($b\in B$) of $W^{b}$ result in integrands $%
A_{j,b}^{\rho }\left( x\right) $ that are not identically zero, by the
discussion following (\ref{reciprocityShowsFiniteNumberBs}).
\end{remark}

\begin{remark}
\label{stratumIsSingleOrbit}If the stratum $\Sigma _{\alpha _{j}}$ is a
single orbit, then we have in the Theorem above that 
\begin{eqnarray*}
&&\frac{1}{2\dim V_{\rho }}\sum_{b\in B}\frac{1}{n_{b}\mathrm{rank~}W^{b}}%
\left( -\eta \left( D_{j}^{S+,b}\right) +h\left( D_{j}^{S+,b}\right) \right)
\int_{G\diagdown \widetilde{\Sigma _{\alpha _{j}}}}A_{j,b}^{\rho }\left(
x\right) ~\widetilde{\left\vert dx\right\vert } \\
&=&\frac{1}{2\dim V_{\rho }}\sum_{\sigma }\frac{n_{\sigma }^{\rho }}{\dim
\left( W_{\sigma }\right) }\left( -\eta \left( D^{S+,\sigma }\right)
+h\left( D^{S+,\sigma }\right) \right) ,
\end{eqnarray*}%
where the restriction of $\rho $ to $H$ is $\mathrm{Res}\left( \left[ \rho
,V_{\rho }\right] \right) =\bigoplus\limits_{\sigma }n_{\sigma }^{\rho }%
\left[ \sigma ,W_{\sigma }\right] $.
\end{remark}

\begin{remark}
Note that the quantities $\eta \left( D_{j}^{S+,b}\right) $ and $h\left(
D_{j}^{S+,b}\right) $ are invariant under equivariant stable homotopies of
the operator $D$, as explained in Remark \ref{stabilityRemarks}.
\end{remark}

\begin{theorem}
(Invariant Index Theorem) \label{InvariantIndexTheorem}With notation as in
the last theorem, if $\rho _{0}$ is the trivial representation, then%
\begin{eqnarray*}
\mathrm{ind}^{\rho _{0}}\left( D\right) &=&\int_{G\diagdown \widetilde{M_{0}}%
}A_{0}^{\rho _{0}}\left( x\right) ~\widetilde{\left\vert dx\right\vert }%
~+\sum_{j=1}^{r}\beta \left( \Sigma _{\alpha _{j}}\right) ~, \\
\beta \left( \Sigma _{\alpha _{j}}\right) &=&\frac{1}{2}\sum_{b\in B}\frac{1%
}{n_{b}\mathrm{rank~}W^{b}}\left( -\eta \left( D_{j}^{S+,b}\right) +h\left(
D_{j}^{S+,b}\right) \right) \int_{G\diagdown \widetilde{\Sigma _{\alpha _{j}}%
}}A_{j,b}^{\rho _{0}}\left( x\right) ~\widetilde{\left\vert dx\right\vert }~,
\end{eqnarray*}%
where $b\in B$ only if $W^{b}$ corresponds to irreducible isotropy
representations whose duals are present in $E^{\alpha _{j}}$, the bundle on
which $D_{\alpha _{j}}$ acts.
\end{theorem}

\begin{proof}
If $\left[ \sigma ,W_{\sigma }\right] ,\left[ \tau ,W_{\tau }\right] \in 
\widehat{H}$, we have $\left( W_{\sigma }\otimes W_{\tau }\right) ^{H}=0$
unless $\tau \simeq \sigma ^{\ast }$, and in this case $\left( W_{\sigma
}\otimes W_{\tau }\right) ^{H}\cong \mathbb{C}$.
\end{proof}

\section{Examples and Applications\label{ExamplesApplicationsSection}}

\subsection{Equivariant indices of elliptic operators}

Clearly, Theorem \ref{MainTheorem} and Theorem \ref{InvariantIndexTheorem}
apply when the $G$-equivariant differential operator $D:\Gamma \left(
M,E^{+}\right) \rightarrow \Gamma \left( M,E^{-}\right) $ is elliptic. In
this case, the kernel and cokernel are finite-dimensional representations of 
$G$, and thus we have the fact that the Atiyah-Singer index satisfies%
\begin{equation*}
\mathrm{ind}\left( D\right) =\sum_{\left[ \rho \right] }\mathrm{ind}^{\rho
}\left( D\right) \dim \left( V_{\rho }\right) ,
\end{equation*}%
and the equivariant indices satisfy%
\begin{eqnarray*}
\mathrm{ind}^{G}\left( D\right) &=&\sum_{\left[ \rho \right] }\mathrm{ind}%
^{\rho }\left( D\right) \left[ \rho \right] \in R\left( G\right) \\
\mathrm{ind}_{g}\left( D\right) &=&\sum_{\left[ \rho \right] }\mathrm{ind}%
^{\rho }\left( D\right) \chi _{\rho }\left( g\right) ,
\end{eqnarray*}%
with all sums, integers, and functions above being finite. Even in this
elliptic case, the result of the main theorem was not known previously.

For example, suppose that $D$ is a general Dirac operator over sections of a
Clifford bundle $E$ that is equivariant with respect to the action of a Lie
group $G$ of isometries of $M$. If we apply Theorem \ref{MainTheorem} to
this operator, all of the operators in the theorem are in fact Dirac
operators. That is, the operators $D_{N}$ and $D_{\Sigma }$ are Dirac
operators formed by restricting Clifford multiplication and the connections
to the normal and tangent bundles of the stratum $\Sigma _{i}$. The
spherical operators $D_{i}^{S^{+}}$ are also Dirac operators induced on the
unit normal bundles of the $\Sigma _{i}$. Further, in calculating the
Atiyah-Singer integrands on the quotients, note that the operators on the
orbit spaces are Dirac operators twisted by a bundle.

\subsection{The de Rham operator}

It is well known that if $M$ is a Riemannian manifold and $f:M\rightarrow M$
is an isometry that is homotopic to the identity, then the Euler
characteristic of $M$ is the sum of the Euler characteristics of the fixed
point sets of $f$. We generalize this result as follows. We consider the de
Rham operator 
\begin{equation*}
d+d^{\ast }:\Omega ^{\mathrm{even}}\left( M\right) \rightarrow \Omega ^{%
\mathrm{odd}}\left( M\right)
\end{equation*}%
on a $G$-manifold, and the invariant index of this operator is the
equivariant Euler characteristic $\chi ^{G}\left( M\right) $, the Euler
characteristic of the elliptic complex consisting of invariant forms. If $G$
is connected and the Euler characteristic is expressed in terms of its $\rho 
$-components, only the invariant part $\chi ^{G}\left( M\right) =\chi ^{\rho
_{0}}\left( M\right) $ appears. This is a consequence of the homotopy
invariance of de Rham cohomology. Thus $\chi ^{G}\left( M\right) =\chi
\left( M\right) $ for connected Lie groups $G$. In general the Euler
characteristic is a sum of components%
\begin{equation*}
\chi \left( M\right) =\sum_{\left[ \rho \right] }\chi ^{\rho }\left(
M\right) ,
\end{equation*}%
where $\chi ^{\rho }\left( M\right) $ is the alternating sum of the
dimensions of the $\left[ \rho \right] $-parts of the cohomology groups (or
spaces of harmonic forms). Since the connected component $G_{0}$ of the
identity in $G$ acts trivially on the harmonic forms, the only nontrivial
components $\chi ^{\rho }\left( M\right) $ correspond to representations
induced from unitary representations of the finite group $G\diagup G_{0}$.

In any case, at each stratum $\Sigma _{\alpha _{j}}$ of codimension at least
two in a $G$-manifold, we may write the de Rham operator (up to lower order
perturbations) as%
\begin{equation*}
d+d^{\ast }=D_{N}\ast D_{\Sigma },
\end{equation*}%
where $D_{N}$ and $D_{\Sigma }$ are both de Rham operators on the respective
spaces of forms. Further, the spherical operator $D^{S}$ is simply 
\begin{equation*}
D^{S}=-c\left( \partial _{r}\right) \left( d+d^{\ast }\right) ^{S},~c\left(
\partial _{r}\right) =dr\wedge -dr\lrcorner
\end{equation*}%
where $\left( d+d^{\ast }\right) ^{S}$ is a vector-valued de Rham operator
on the sphere. We will need to compute the kernel of this operator, which is
related to the harmonic forms on the sphere; it is spanned by $\left\{
1,dr,\omega ,dr\wedge \omega \right\} $, where $\omega $ is the sphere
volume form. The eigenforms of this operator are integers, and if $\alpha
+dr\wedge \beta $ is an eigenform of $D^{S}$ corresponding to eigenvalue $%
\lambda $, then%
\begin{eqnarray*}
\lambda \left( \alpha +dr\wedge \beta \right) &=&-c\left( \partial
_{r}\right) \left( d+d^{\ast }\right) ^{S}\left( \alpha +dr\wedge \beta
\right) \\
&=&-dr\wedge \left( d+d^{\ast }\right) ^{S}\alpha -\left( d+d^{\ast }\right)
^{S}\beta ,
\end{eqnarray*}%
so that 
\begin{eqnarray*}
\left( d+d^{\ast }\right) ^{S}\alpha &=&-\lambda \beta ;\left( d+d^{\ast
}\right) ^{S}\beta =-\lambda \alpha ,\text{ or} \\
\left( d+d^{\ast }\right) ^{S}\left( \alpha \pm \beta \right) &=&\mp \lambda
\left( \alpha \pm \beta \right) .
\end{eqnarray*}%
Also%
\begin{eqnarray*}
D^{S}\left( \alpha -dr\wedge \beta \right) &=&-c\left( \partial _{r}\right)
\left( d+d^{\ast }\right) ^{S}\left( \alpha -dr\wedge \beta \right) \\
&=&\lambda dr\wedge \beta -\lambda \alpha =-\lambda \left( \alpha -dr\wedge
\beta \right) ,\text{ and} \\
D^{S}\left( \beta \pm dr\wedge \alpha \right) &=&\pm \lambda \left( \beta
\pm dr\wedge \alpha \right) .
\end{eqnarray*}%
Thus the spectra of both $D^{S}$ and $\left( d+d^{\ast }\right) ^{S}$ are
symmetric. Let $\left[ G_{j}\right] $ be the isotropy type corresponding to
the stratum $\Sigma _{\alpha _{j}}$. If we restrict to forms of $G_{j}$%
-representation type $\sigma $, the forms $\alpha $,$\beta $ would
necessarily be of that type, but the symmetry of the spectra would remain.
Thus for the de Rham operator, 
\begin{equation*}
\eta \left( D^{S+,\sigma }\right) =0
\end{equation*}%
at each stratum. The dimension $h\left( D^{S+,\sigma }\right) $ of the
kernel is not always the same; the $+$ component of the forms contains at
most the span of $\left\{ 1,\omega \right\} $ if the codimension of the
stratum is odd and at most the span of $\left\{ 1,dr\wedge \omega \right\} $
if the codimension is even. Observe that $1$, $\omega $, and $dr\wedge
\omega $ are invariant under every orientation-preserving isometry, but $%
\omega $ and $dr\wedge \omega $ change sign under orientation reversing
isometries. Thus, the only representations of $G_{j}$ that appear are the
induced one-dimensional representations of $G_{j}$ on the transverse volume
form to $\Sigma _{\alpha _{j}}$. If some elements of $G_{j}$ reverse
orientation of the normal bundle, then let $\xi _{G_{j}}$ denote the
relevant one-dimensional representation of $G_{j}$ as $\pm 1$. Then 
\begin{equation}
h\left( D_{j}^{S+,\sigma }\right) =\left\{ 
\begin{array}{ll}
2 & \text{if }\sigma =\mathbf{1}\text{ and }G_{j}\text{ preserves orientation%
} \\ 
1 & \text{if }\sigma =\mathbf{1}\text{ and }G_{j}\text{ does not preserve
orientation} \\ 
1 & \text{if }\sigma =\xi _{G_{j}}\text{ and }G_{j}\text{ does not preserve
orientation} \\ 
0 & \text{otherwise}%
\end{array}%
\right.  \label{hFormulasEuler}
\end{equation}%
The orientation line bundle $\mathcal{L}_{N_{j}}\rightarrow \Sigma _{j}$ of
the normal bundle to $\Sigma _{j}$ is a pointwise representation space for
the representation $\xi _{G_{j}}$, and it is the canonical isotropy $G$%
-bundle $W^{b}$ corresponding to $\left( \alpha _{j},\left[ \xi _{G_{j}}%
\right] \right) $. We may also take it to be a representation bundle for the
trivial $G_{j}$-representation $\mathbf{1}$ (although the trivial line
bundle is the canonical one). The main theorem takes the form

\begin{eqnarray*}
\mathrm{ind}^{\rho }\left( d+d^{\ast }\right) &=&\frac{1}{\dim V_{\rho }}%
\int_{G\diagdown \widetilde{M_{0}}}A_{0}^{\rho }\left( x\right) ~\widetilde{%
\left\vert dx\right\vert }~+\sum_{j}\beta \left( \Sigma _{\alpha
_{j}}\right) ~, \\
\beta \left( \Sigma _{\alpha _{j}}\right) &=&\frac{1}{2\dim V_{\rho }}%
\sum_{j}\left( h\left( D_{j}^{S+,\xi _{G_{j}}}\right) +h\left( D_{j}^{S+,%
\mathbf{1}}\right) \right) \int_{G\diagdown \widetilde{\Sigma _{\alpha _{j}}}%
}A_{j}^{\rho }\left( x,\mathcal{L}_{N_{j}}\right) ~\widetilde{\left\vert
dx\right\vert }~ \\
&=&\frac{1}{\dim V_{\rho }}\sum_{j}\int_{G\diagdown \widetilde{\Sigma
_{\alpha _{j}}}}A_{j}^{\rho }\left( x,\mathcal{L}_{N_{j}}\right) ~\widetilde{%
\left\vert dx\right\vert }.
\end{eqnarray*}%
In this case it is useful to rewrite $\int_{G\diagdown \widetilde{\Sigma
_{\alpha _{j}}}}A_{j}^{\rho }\left( x,\mathcal{L}_{N_{j}}\right) ~\widetilde{%
\left\vert dx\right\vert }$ as $\int_{\widetilde{\Sigma _{\alpha _{j}}}%
}K_{j}^{\rho }\left( x,\mathcal{L}_{N_{j}}\right) ~\widetilde{\left\vert
dx\right\vert }$ before taking it to the quotient. We see that $K_{j}^{\rho
}\left( x,\mathcal{L}_{N_{j}}\right) $ is the Gauss-Bonnet integrand on the
desingularized stratum $\widetilde{\Sigma _{\alpha _{j}}}$, restricted to $%
\mathcal{L}_{N_{j}}$-twisted forms of type $\left[ \rho \right] $. The
result is the relative Euler characteristic 
\begin{equation*}
\int_{\widetilde{\Sigma _{\alpha _{j}}}}K_{j}^{\rho }\left( x,\mathcal{L}%
_{N_{j}}\right) ~\widetilde{\left\vert dx\right\vert }=\chi ^{\rho }\left( 
\overline{\Sigma _{\alpha _{j}}},\text{lower strata},\mathcal{L}%
_{N_{j}}\right) ,
\end{equation*}%
Here, the relative Euler characteristic is defined for $X$ a closed subset
of a manifold $Y$ as $\chi \left( Y,X,\mathcal{V}\right) =\chi \left( Y,%
\mathcal{V}\right) -\chi \left( X,\mathcal{V}\right) $, which is also the
alternating sum of the dimensions of the relative de Rham cohomology groups
with coefficients in a complex vector bundle $\mathcal{V}\rightarrow Y$. The
superscript $\rho $ denotes the restriction to the subgroups of these
cohomology groups of $G$-representation type $\left[ \rho \right] $. Note
that if $Y$ is a $G\diagup G_{j}$-bundle, $\chi \left( Y,\mathcal{V}\right)
^{\rho }$ is the alternating sum of the dimensions of subspaces of harmonic
forms with coefficients in the $G$-bundle $\mathcal{V}\rightarrow Y$ with
representation type $\left[ \rho \right] $. Further, since $\Sigma _{\alpha
_{j}}$ is a fiber bundle over $G\diagdown \Sigma _{\alpha _{j}}$ with fiber $%
G$-diffeomorphic to $G\diagup G_{j}$, we have%
\begin{equation*}
\int_{\widetilde{\Sigma _{\alpha _{j}}}}K_{j}^{\rho }\left( x,\mathcal{L}%
_{N_{j}}\right) ~\widetilde{\left\vert dx\right\vert }=\chi ^{\rho }\left(
G\diagup G_{j},\mathcal{L}_{N_{j}}\right) \chi \left( G\diagdown \overline{%
\Sigma _{\alpha _{j}}},G\diagdown \text{lower strata}\right) ,
\end{equation*}%
by the formula for the Euler characteristic on fiber bundles. A similar
formula holds for the principal stratum, with no orientation bundle. Putting
these facts together with Theorem \ref{MainTheorem} and $\mathrm{ind}^{\rho
}\left( d+d^{\ast }\right) =\frac{1}{\dim V_{\rho }}\chi ^{\rho }\left(
M\right) $, we have the following result.

\begin{theorem}
\label{EulerCharacteristicTheorem}Let $M$ be a compact $G$-manifold, with $G$
a compact Lie group and principal isotropy subgroup $H_{\mathrm{pr}}$. Let $%
M_{0}$ denote the principal stratum, and let $\Sigma _{\alpha _{1}}$,...,$%
\Sigma _{\alpha _{r}}$ denote all the components of all singular strata
relative to $G$. We use the notations for $\chi ^{\rho }\left( Y,X\right) $
and $\chi ^{\rho }\left( Y\right) $ as in the discussion above. Then 
\begin{eqnarray*}
\chi ^{\rho }\left( M\right) &=&\chi ^{\rho }\left( G\diagup H_{\mathrm{pr}%
}\right) \chi \left( G\diagdown M,G\diagdown \text{singular strata}\right) \\
&&+\sum_{j}\chi ^{\rho }\left( G\diagup G_{j}\text{~},\mathcal{L}%
_{N_{j}}\right) \chi \left( G\diagdown \overline{\Sigma _{\alpha _{j}}}%
,G\diagdown \text{lower strata}\right) ,
\end{eqnarray*}%
where $\mathcal{L}_{N_{j}}$ is the orientation line bundle of normal bundle
of the stratum component $\Sigma _{\alpha _{j}}$.
\end{theorem}

\begin{example}
Let $M=S^{n}$, let $G=O\left( n\right) $ acting on latitude spheres
(principal orbits, diffeomorphic to $S^{n-1}$). Then there are two strata,
with the singular strata being the two poles (each with Euler characteristic 
$1$). Without using the theorem, since the only harmonic forms are the
constants and multiples of the volume form, we expect that%
\begin{equation*}
\chi ^{\rho }\left( S^{n}\right) =\left\{ 
\begin{array}{ll}
\left( -1\right) ^{n} & \text{if }\rho =\xi \\ 
1 & \text{if }\rho =\rho _{0}=\mathbf{1}%
\end{array}%
\right. ,
\end{equation*}%
where $\xi $ is the induced one dimensional representation of $O\left(
n\right) $ on the volume forms. We have that%
\begin{equation*}
\chi ^{\rho }\left( G\diagup H_{\mathrm{pr}}\right) =\chi ^{\rho }\left(
S^{n-1}\right) =\left\{ 
\begin{array}{ll}
\left( -1\right) ^{n-1} & \text{if }\rho =\xi \\ 
1 & \text{if }\rho =\rho _{0}=\mathbf{1}%
\end{array}%
\right. ,
\end{equation*}%
and 
\begin{eqnarray*}
\chi \left( G\diagdown M,G\diagdown \text{singular strata}\right) &=&\chi
\left( \left[ -1,1\right] ,\left\{ -1,1\right\} \right) \\
&=&-1.
\end{eqnarray*}

At each pole, the isotropy subgroup is the full $O\left( n\right) $, which
may reverse the orientation of the normal space. Then%
\begin{equation*}
\chi ^{\rho }\left( G\diagup G_{j},\mathcal{L}_{N_{j}}\right) =\chi ^{\rho
}\left( \mathrm{pt}\right) =\left\{ 
\begin{array}{ll}
1 & \text{if }\rho =\rho _{0}=\mathbf{1}\text{,} \\ 
0 & \text{otherwise.}%
\end{array}%
\right.
\end{equation*}%
Also, for each pole $\Sigma _{\alpha _{j}}$, 
\begin{equation*}
\chi \left( G\diagdown \overline{\Sigma _{\alpha _{j}}},G\diagdown \text{%
lower strata}\right) =\chi \left( \mathrm{pt}\right) =1.
\end{equation*}%
By Theorem \ref{EulerCharacteristicTheorem}, we have%
\begin{eqnarray*}
\chi ^{\rho }\left( S^{n}\right) &=&\left\{ 
\begin{array}{ll}
\left( -1\right) ^{n} & \text{if }\rho =\xi \\ 
-1 & \text{if }\rho =\rho _{0}=\mathbf{1}%
\end{array}%
\right. \\
&&+\left\{ 
\begin{array}{ll}
0 & \text{if }\rho =\xi \\ 
2\cdot 1 & \text{if }\rho =\rho _{0}%
\end{array}%
\right. \\
&=&\left\{ 
\begin{array}{ll}
\left( -1\right) ^{n} & \text{if }\rho =\xi \\ 
1 & \text{if }\rho =\rho _{0}%
\end{array}%
\right. ,
\end{eqnarray*}%
which agrees with the predicted result.
\end{example}

\begin{example}
If instead the group $\mathbb{Z}_{2}$ acts on $S^{n}$ by the antipodal map,
note that%
\begin{equation*}
\chi ^{\rho }\left( S^{n}\right) =\left\{ 
\begin{array}{ll}
1-1=0 & \text{if }\rho =\rho _{0}\text{ and }n\text{ is odd} \\ 
1 & \text{if }\rho =\rho _{0}\text{ and }n\text{ is even} \\ 
1 & \text{if }\rho =\xi \text{ and and }n\text{ is even} \\ 
0 & \text{otherwise}%
\end{array}%
\right.
\end{equation*}%
since the antipodal map is orientation preserving in odd dimensions and
orientation reversing in even dimensions. On the other hand, by Theorem \ref%
{EulerCharacteristicTheorem}, since there are no singular strata, we have%
\begin{eqnarray*}
\chi ^{\rho }\left( S^{n}\right) &=&\chi ^{\rho }\left( G\diagup H_{\mathrm{%
pr}}\right) \chi \left( G\diagdown M,G\diagdown \text{singular strata}\right)
\\
&=&\left\{ 
\begin{array}{ll}
1\cdot \chi \left( \mathbb{R}P^{n}\right) & \text{if }\rho =\rho _{0}\text{
or }\xi \text{ and }n\text{ is odd} \\ 
1\cdot \chi \left( \mathbb{R}P^{n}\right) & \text{if }\rho =\rho _{0}\text{
or }\xi \text{ and }n\text{ is even} \\ 
0\cdot \chi \left( \mathbb{R}P^{n}\right) & \text{otherwise}%
\end{array}%
\right. \\
&=&\left\{ 
\begin{array}{ll}
0 & \text{if }\rho =\rho _{0}\text{ or }\xi \text{ and }n\text{ is odd} \\ 
1 & \text{if }\rho =\rho _{0}\text{ or }\xi \text{ and }n\text{ is even} \\ 
0~ & \text{otherwise}%
\end{array}%
\right. ,
\end{eqnarray*}%
which agrees with the direct computation.
\end{example}

\begin{example}
Consider the action of $\mathbb{Z}_{4}$ on the torus $T^{2}=\mathbb{R}%
^{2}\diagup \mathbb{Z}^{2}$, where the action is generated by a $\frac{\pi }{%
2}$ rotation. Explicitly, $k\in \mathbb{Z}_{4}$ acts on $\left( 
\begin{array}{c}
y_{1} \\ 
y_{2}%
\end{array}%
\right) $ by 
\begin{equation*}
\phi \left( k\right) \left( 
\begin{array}{c}
y_{1} \\ 
y_{2}%
\end{array}%
\right) =\left( 
\begin{array}{cc}
0 & -1 \\ 
1 & 0%
\end{array}%
\right) ^{k}\left( 
\begin{array}{c}
y_{1} \\ 
y_{2}%
\end{array}%
\right) .
\end{equation*}%
Endow $T^{2}$ with the standard flat metric. The harmonic forms have basis $%
\left\{ 1,dy_{1},dy_{2},dy_{1}\wedge dy_{2}\right\} $. Let $\rho _{j}$ be
the irreducible character defined by $k\in \mathbb{Z}_{4}\mapsto e^{ikj\pi
/2}$. Then the de Rham operator $\left( d+d^{\ast }\right) ^{\rho _{0}}$ on $%
\mathbb{Z}_{4}$-invariant forms has kernel $\left\{ c_{0}+c_{1}dy_{1}\wedge
dy_{2}:c_{0},c_{1}\in \mathbb{C}\right\} $. One also sees that $\ker \left(
d+d^{\ast }\right) ^{\rho _{1}}=\mathrm{span}\left\{ idy_{1}+dy_{2}\right\} $%
, $\ker \left( d+d^{\ast }\right) ^{\rho _{2}}=\left\{ 0\right\} $, and $%
\ker \left( d+d^{\ast }\right) ^{\rho _{3}}=\mathrm{span}\left\{
-idy_{1}+dy_{2}\right\} $. Then 
\begin{equation*}
\chi ^{\rho _{0}}\left( T^{2}\right) =2,\chi ^{\rho _{1}}\left( T^{2}\right)
=\chi ^{\rho _{3}}\left( T^{2}\right) =-1,\chi ^{\rho _{2}}\left(
T^{2}\right) =0.
\end{equation*}%
This illustrates the point that it is not possible to use the Atiyah-Singer
integrand on the quotient of the principal stratum to compute even the
invariant index alone. Indeed, the Atiyah-Singer integrand would be a
constant times the Gauss curvature, which is identically zero. In these
cases, the three singular points $a_{1}=\left( 
\begin{array}{c}
0 \\ 
0%
\end{array}%
\right) ~,a_{2}=\left( 
\begin{array}{c}
0 \\ 
\frac{1}{2}%
\end{array}%
\right) ~,a_{3}=\left( 
\begin{array}{c}
\frac{1}{2} \\ 
\frac{1}{2}%
\end{array}%
\right) $ certainly contribute to the index. The quotient $T^{2}\diagup 
\mathbb{Z}_{4}$ is an orbifold homeomorphic to a sphere.

We now compute the Euler characteristics $\chi ^{\rho }\left( T^{2}\right) $
using Theorem \ref{EulerCharacteristicTheorem}. First, $G\diagup H_{\mathrm{%
pr}}=G$ is four points on which $G$ acts transitively, and thus $\chi ^{\rho
}\left( G\diagup H_{\mathrm{pr}}\right) =1$ for each possible choice of $%
\rho $. We have $\chi \left( \mathbb{Z}_{4}\diagdown T^{2},\left\{ \text{3
points}\right\} \right) =\chi \left( S^{2},\left\{ \text{3 points}\right\}
\right) =-1$. Each of the singular points $a_{1}$, $a_{2}$, $a_{3}$ is one
of the strata components $\Sigma _{\alpha _{j}}$, and thus $\chi \left(
G\diagdown \overline{\Sigma _{\alpha _{j}}},G\diagdown \text{lower strata}%
\right) =1$ in each of these cases. We have the isotropy subgroup at $a_{j}$
is 
\begin{equation*}
G_{j}=G_{a_{j}}=\left\{ 
\begin{array}{ll}
\mathbb{Z}_{4} & \text{if }j=1\text{ or }3 \\ 
\mathbb{Z}_{2} & \text{if }j=2%
\end{array}%
\right.
\end{equation*}%
so that 
\begin{equation*}
\chi ^{\rho }\left( G\diagup G_{j}\text{~},\mathcal{L}_{N_{j}}\right) =\chi
^{\rho }\left( G\diagup G_{j}\text{~}\right) =\left\{ 
\begin{array}{ll}
1 & \text{if }j=1\text{ or }3~\text{and }\rho =\rho _{0} \\ 
1 & \text{if }j=2\text{ and }\rho =\rho _{0}\text{ or }\rho _{2} \\ 
0~ & \text{otherwise}%
\end{array}%
\right.
\end{equation*}%
Then Theorem \ref{EulerCharacteristicTheorem} implies 
\begin{eqnarray*}
\chi ^{\rho }\left( M\right) &=&1\cdot \left( -1\right) +\left\{ 
\begin{array}{ll}
3 & \text{if }\rho =\rho _{0} \\ 
1 & \text{if }\rho =\rho _{2} \\ 
0~ & \text{if }\rho =\rho _{1}\text{ or }\rho _{3}%
\end{array}%
\right. \\
&=&\left\{ 
\begin{array}{ll}
2 & \text{if }\rho =\rho _{0} \\ 
0 & \text{if }\rho =\rho _{2} \\ 
-1~ & \text{if }\rho =\rho _{1}\text{ or }\rho _{3}%
\end{array}%
\right. ~,
\end{eqnarray*}%
which agrees with the previous direct calculation.
\end{example}

\subsection{Transverse Signature Operator}

\medskip In this section we give an example of a transverse signature
operator that arises from an $S^{1}$ action on a $5$-manifold. This is
essentially a modification of an example from \cite[pp. 84ff]{A}, and it
illustrates the fact that the eta invariant term does not always vanish. Let 
$Z^{4}$ be a closed, oriented, $4$-dimensional Riemannian manifold on which $%
\mathbb{Z}_{p}$ ($p$ prime $>2$) acts by isometries with isolated fixed
points $x_{i}$, $i=1,...,N$. Let $M=S^{1}\times _{\mathbb{Z}_{p}}Z^{4}$,
where $\mathbb{Z}_{p}$ acts on $S^{1}$ by rotation by multiples of $\frac{%
2\pi }{p}$. Then $S^{1}$ acts on $M$, and $S^{1}\diagdown M\cong \mathbb{Z}%
_{p}\diagdown Z^{4}$. Observe that this group action can be viewed as a
Riemannian foliation by circles.

Next, let $D^{+}$ denote the signature operator $d+d^{\ast }$ from self-dual
to anti-self-dual forms on $Z^{4}$; this induces a transversally elliptic
operator (also denoted by $D^{+}$). Then the $S^{1}$-invariant index of $%
D^{+}$ satisfies%
\begin{equation*}
\mathrm{ind}^{\rho _{0}}\left( D^{+}\right) =\mathrm{Sign}\left(
S^{1}\diagdown M\right) =\mathrm{Sign}\left( \mathbb{Z}_{p}\diagdown
Z^{4}\right) .
\end{equation*}%
By the Invariant Index Theorem (Theorem \ref{InvariantIndexTheorem}), Remark %
\ref{stratumIsSingleOrbit}, and the fact that the Atiyah-Singer integrand is
the Hirzebruch $L$-polynomial $\frac{1}{3}p_{1}$,%
\begin{eqnarray*}
\mathrm{ind}^{\rho _{0}}\left( D\right) &=&\frac{1}{3}\int_{\widetilde{M}%
\diagup S^{1}}p_{1}~ \\
&&+\frac{1}{2}\sum_{j=1}^{N}\left( -\eta \left( D_{j}^{S+,\rho _{0}}\right)
+h\left( D_{j}^{S+,\rho _{0}}\right) \right) ,
\end{eqnarray*}%
where each $D_{j}^{S+,\rho _{0}}$ is two copies of the boundary signature
operator 
\begin{equation*}
B=\left( -1\right) ^{p}\left( \ast d-d\ast \right)
\end{equation*}%
on $2l$-forms ($l=0,1$) on the lens space $S^{3}\diagup \mathbb{Z}_{p}$. We
have $h\left( D_{j}^{S+,\rho _{0}}\right) =2h\left( B\right) =2$
(corresponding to constants), and in \cite{APS2} the eta invariant is
explicitly calculated to be%
\begin{equation*}
\eta \left( D_{j}^{S+,\rho _{0}}\right) =2\eta \left( B\right) =-\frac{2}{p}%
\sum_{k=1}^{p-1}\cot \left( \frac{km_{j}\pi }{p}\right) \cot \left( \frac{%
kn_{j}\pi }{p}\right) ,
\end{equation*}%
where the action of the generator $\zeta $ of $\mathbb{Z}_{p}$ on $S^{3}$ is%
\begin{equation*}
\zeta \cdot \left( z_{1},z_{2}\right) =\left( e^{\frac{2m_{j}\pi i}{p}%
}z_{1},e^{\frac{2n_{j}\pi i}{p}}z_{2}\right) ,
\end{equation*}%
with $\left( m_{j},p\right) =\left( n_{j},p\right) =1$. Thus,%
\begin{equation}
\mathrm{Sign}\left( S^{1}\diagdown M\right) =\frac{1}{3}\int_{\mathbb{Z}%
_{p}\diagdown \widetilde{Z^{4}}}p_{1}+\frac{1}{p}\sum_{j=1}^{N}%
\sum_{k=1}^{p-1}\cot \left( \frac{km_{j}\pi }{p}\right) \cot \left( \frac{%
kn_{j}\pi }{p}\right) +N  \label{signatureFormula}
\end{equation}%
Note that in \cite[pp. 84ff]{A} it is shown that%
\begin{equation*}
\mathrm{Sign}\left( S^{1}\diagdown M\right) =\frac{1}{3}\int_{\mathbb{Z}%
_{p}\diagdown Z^{4}}p_{1}+\frac{1}{p}\sum_{j=1}^{N}\sum_{k=1}^{p-1}\cot
\left( \frac{km_{j}\pi }{p}\right) \cot \left( \frac{kn_{j}\pi }{p}\right) ,
\end{equation*}%
which demonstrates that 
\begin{equation*}
\frac{1}{3}\int_{\mathbb{Z}_{p}\diagdown Z^{4}}p_{1}-\frac{1}{3}\int_{%
\mathbb{Z}_{p}\diagdown \widetilde{Z^{4}}}p_{1}=N,
\end{equation*}%
illustrating the difference in total curvature between the desingularization 
$\widetilde{M}$ and the original $M$.

\subsection{Basic index theorem for Riemannian foliations\label%
{BasicIndexSubsection}}

\medskip The content of this section will be proved, generalized, and
expanded in detail in \cite{BKR2}. Let $M$ be an $n$-dimensional, closed,
connected manifold, and let $\mathcal{F}$ be a codimension $q$ foliation on $%
M$. Let $Q$ denote the quotient bundle $TM\diagup T\mathcal{F}$ over $M$.
Such a foliation is called a \emph{Riemannian foliation} if it is endowed
with a metric on $Q$ (called the transverse metric) that is \emph{%
holonomy-invariant}; that is, the Lie derivative of that transverse metric
with respect to every leafwise tangent vector is zero. The metric on $Q$ can
always be extended to a Riemannian metric on $M$; the extended metric
restricted to the normal bundle $N\mathcal{F}=\left( T\mathcal{F}\right)
^{\bot }$ agrees with the transverse metric via the isomorphism $Q\cong N%
\mathcal{F}$. We refer the reader to \cite{Mo}, \cite{Re}, and \cite{T} for
introductions to the geometric and analytic properties of Riemannian
foliations.

Let $\widehat{M}$ be the transverse orthonormal frame bundle of $(M,\mathcal{%
F})$, and let $p$ be the natural projection $p:\widehat{M}\rightarrow M$.
The Bott connection is a natural connection on $Q$ that induces a connection
on $\widehat{M}$ (see \cite[pp. 80ff]{Mo} ). The manifold $\widehat{M}$ is a
principal $O(q)$-bundle over $M$. Given $\hat{x}\in \widehat{M}$, let $\hat{x%
}g$ denote the well-defined right action of $g\in G=O(q)$ applied to $\hat{x}
$. Associated to $\mathcal{F}$ is the lifted foliation $\widehat{\mathcal{F}}
$ on $\widehat{M}$; the distribution $T\widehat{\mathcal{F}}$ is the
horizontal lift of $T\mathcal{F}$. By the results of Molino (see \cite[%
pp.~105-108, p.~147ff]{Mo} ), the lifted foliation is transversally
parallelizable (meaning that there exists a global basis of the normal
bundle consisting of vector fields whose flows preserve $\widehat{\mathcal{F}%
}$), and the closures of the leaves are fibers of a fiber bundle $\widehat{%
\pi }:\widehat{M}\rightarrow \widehat{W}$. The manifold $\widehat{W}$ is
smooth and is called the basic manifold. Let $\overline{\widehat{\mathcal{F}}%
}$ denote the foliation of $\widehat{M}$ by leaf closures of $\widehat{%
\mathcal{\ F}}$, which is shown by Molino to be a fiber bundle. The leaf
closure space of $\left( M,\mathcal{F}\right) $ is denoted $W=M\diagup 
\overline{\mathcal{F}}=\widehat{W}\diagup G$. 
\begin{equation*}
\begin{array}{ccccccc}
p^{\ast }E &  &  &  & \mathcal{E} &  &  \\ 
& \searrow &  &  & \downarrow &  &  \\ 
SO\left( q\right) & \hookrightarrow & \left( \widehat{M},\widehat{\mathcal{F}%
}\right) & \overset{\widehat{\pi }}{\longrightarrow } & \widehat{W} &  &  \\ 
&  & \downarrow ^{p} & \circlearrowleft & \downarrow \,\, &  &  \\ 
E & \rightarrow & \left( M,\mathcal{F}\right) & \overset{\pi }{%
\longrightarrow } & W &  & 
\end{array}%
\end{equation*}

Endow $(\widehat{M},\widehat{\mathcal{F}})$ with the transverse metric $%
g^{Q}\oplus g^{O(q)}$, where $g^{Q}$ is the pullback of metric on $Q$, and $%
g^{O(q)}$ is the standard, normalized, biinvariant metric on the fibers. We
require that vertical vectors are orthogonal to horizontal vectors. This
transverse metric gives each of $(\widehat{M},\widehat{\mathcal{F}})$ and $(%
\widehat{M},\overline{\widehat{\mathcal{F}}})$ the structure of a Riemannian
foliation. The transverse metric on $(\widehat{M},\overline{\widehat{%
\mathcal{F}}})$ induces a well--defined Riemannian metric on $\widehat{W}$.
The action of $G=O(q)$ on $\widehat{M}$ induces an isometric action on $%
\widehat{W}$.

For each leaf closure $\overline{\widehat{L}}\in \overline{\widehat{\mathcal{%
F}}}$ and $\widehat{x}\in \overline{\widehat{L}}$, the restricted map $p:%
\overline{\widehat{L}}\rightarrow \overline{L}$ is a principal bundle with
fiber isomorphic to a subgroup $H_{\widehat{x}}\subset O(q)$, which is the
isotropy subgroup at the point $\widehat{\pi }(\hat{x})\in \widehat{W}$. The
conjugacy class of this group is an invariant of the leaf closure $\overline{%
L}$, and the strata of the group action on $\widehat{W}$ correspond to the
strata of the leaf closures of $(M,\mathcal{F})$.

A \emph{basic form} over a foliation is a global differential form that is
locally the pullback of a form on the leaf space; more precisely, $\alpha
\in \Omega ^{\ast }\left( M\right) $ is basic if for any vector tangent to
the foliation, the interior product with both $\alpha $ and $d\alpha $ is
zero. A\ \emph{basic vector field} is a vector field $V$ whose flow
preserves the foliation. In a Riemannian foliation, near any point it is
possible to choose a local orthonormal frame of $Q$ represented by basic
vector fields.

A vector bundle $E\rightarrow \left( M,\mathcal{F}\right) $ that is \emph{%
foliated} may be endowed with a basic connection $\nabla ^{E}$ (one for
which the associated curvature forms are basic -- see \cite{KT2}). An
example of such a bundle is the normal bundle $Q$. Given such a foliated
bundle, a section $s\in \Gamma \left( E\right) $ is called a \emph{basic
section} if for every $X\in T\mathcal{F}$, $\nabla _{X}^{E}s=0$. Let $\Gamma
_{b}\left( E\right) $ denote the space of basic sections of $E$. Note that
the basic sections of $Q$ correspond to basic normal vector fields.

An example of another foliated bundle over a component of a stratum $M_{j}$
is the bundle defined as follows. Let $E\rightarrow M$ be any foliated
vector bundle. Let $\Sigma _{\alpha _{j}}=\widehat{\pi }\left( p^{-1}\left(
M_{j}\right) \right) $ be the corresponding stratum on the basic manifold $%
\widehat{W}$, and let $W^{\tau }\rightarrow \Sigma _{\alpha _{j}}$ be a
canonical isotropy bundle (Definition \ref{canonicalIsotropyBundleDefinition}%
). Consider the bundle $\widehat{\pi }^{\ast }W^{\tau }\otimes p^{\ast
}E\rightarrow p^{-1}\left( M_{j}\right) $, which is foliated and basic for
the lifted foliation restricted to $p^{-1}\left( M_{j}\right) $. This
defines a new foliated bundle $E^{\tau }\rightarrow M_{j}$ by letting $%
E_{x}^{\tau }$ be the space of $O\left( q\right) $-invariant sections of $%
\widehat{\pi }^{\ast }W^{\tau }\otimes p^{\ast }E$ restricted to $%
p^{-1}\left( x\right) $. We call this bundle \textbf{the} $W^{\tau }$\textbf{%
-twist of} $E\rightarrow M_{j}$.

Suppose that $E$ is a foliated $\mathbb{C}\mathrm{l}\left( Q\right) $ module
with basic $\mathbb{C}\mathrm{l}\left( Q\right) $ connection $\nabla ^{E}$
over a Riemannian foliation $\left( M,\mathcal{F}\right) $. Then it can be
shown that Clifford multiplication by basic vector fields preserves $\Gamma
_{b}\left( E\right) $, and we have the operator%
\begin{equation*}
D_{b}^{E}:\Gamma _{b}\left( E^{+}\right) \rightarrow \Gamma _{b}\left(
E^{-}\right)
\end{equation*}%
defined for any local orthonormal frame $\left\{ e_{1},...,e_{q}\right\} $
for $Q$ by%
\begin{equation*}
D_{b}^{E}=\left. \sum_{j=1}^{q}c\left( e_{j}\right) \nabla
_{e_{j}}^{E}\right\vert _{\Gamma _{b}\left( E\right) }.
\end{equation*}%
Then $D_{b}^{E}$ can be shown to be well-defined and is called the basic
Dirac operator corresponding to the foliated $\mathbb{C}\mathrm{l}\left(
Q\right) $ module $E$ (see \cite{GlK}). We note that this operator is not
symmetric unless a zero$^{\mathrm{th}}$ order term involving the mean
curvature is added; see \cite{KTFol}, \cite{KT4}, \cite{KT}, \cite{GlK}, 
\cite{PrRi}, \cite{HabRi}, \cite{BKR2} for more information regarding
essential self-adjointness of the modified operator and its spectrum. In the
formulas below, any lower order terms that preserve the basic sections may
be added without changing the index.

\begin{definition}
The \emph{analytic basic index} of $D_{b}^{E}$ is 
\begin{equation*}
\mathrm{ind}_{b}\left( D_{b}^{E}\right) =\dim \ker D_{b}^{E}-\dim \ker
\left( D_{b}^{E}\right) ^{\ast }.
\end{equation*}
\end{definition}

As we will show in \cite{BKR2}, these dimensions are finite, and it is
possible to identify $\mathrm{ind}_{b}\left( D_{b}^{E}\right) $ with the
invariant index of a first order, $O\left( q\right) $-equivariant
differential operator $\widehat{D}$ over a vector bundle over the basic
manifold $\widehat{W}$. By applying the invariant index theorem (Theorem \ref%
{InvariantIndexTheorem}), we obtain the following formula for the index. In
what follows, if $U$ denotes an open subset of a stratum of $\left( M,%
\mathcal{F}\right) $, $U^{\prime }$ denotes the desingularization of $U$
very similar to that in Section \ref{BlowupDoubleMainSection}, and $%
\widetilde{U}$ denotes the fundamental domain of $U$ inside $U^{\prime }$.

\begin{theorem}
(Basic Index Theorem for Riemannian foliations, in \cite{BKR2}) Let $M_{0}$
be the principal stratum of the Riemannian foliation $\left( M,\mathcal{F}%
\right) $, and let $M_{1}$, ... , $M_{r}$ denote all the components of all
singular strata, corresponding to $O\left( q\right) $-isotropy types $\left[
G_{1}\right] $, ... ,$\left[ G_{r}\right] $ on the basic manifold. With
notation as in the discussion above, we have 
\begin{eqnarray*}
\mathrm{ind}_{b}\left( D_{b}^{E}\right) &=&\int_{\widetilde{M_{0}}\diagup 
\overline{\mathcal{F}}}A_{0,b}\left( x\right) ~\widetilde{\left\vert
dx\right\vert }+\sum_{j=1}^{r}\beta \left( M_{j}\right) ~ \\
\beta \left( M_{j}\right) &=&\frac{1}{2}\sum_{\tau }\frac{1}{n_{\tau }%
\mathrm{rank~}W^{\tau }}\left( -\eta \left( D_{j}^{S+,\tau }\right) +h\left(
D_{j}^{S+,\tau }\right) \right) \int_{\widetilde{M_{j}}\diagup \overline{%
\mathcal{F}}}A_{j,b}^{\tau }\left( x\right) ~\widetilde{\left\vert
dx\right\vert },
\end{eqnarray*}%
where the sum is over all components of singular strata and over all
canonical isotropy bundles $W^{\tau }$, only a finite number of which yield
nonzero $A_{j,b}^{\tau }$, and where

\begin{enumerate}
\item $A_{0,b}\left( x\right) $ is the Atiyah-Singer integrand, the local
supertrace of the ordinary heat kernel associated to the elliptic operator
induced from $\widetilde{D_{b}^{E}}$ (a desingularization of $D_{b}^{E}$) on
the quotient $\widetilde{M_{0}}\diagup \overline{\mathcal{F}}$, where the
bundle $E$ is replaced by the space of basic sections of over each leaf
closure;

\item $\eta \left( D_{j}^{S+,b}\right) $ and $h\left( D_{j}^{S+,b}\right) $
are defined in a similar way as in Theorem \ref{MainTheorem}, using a
decomposition $D_{b}^{E}=D_{N}\ast D_{M_{j}}$ at each singular stratum;

\item $A_{j,b}^{\tau }\left( x\right) $ is the local supertrace of the
ordinary heat kernel associated to the elliptic operator induced from $%
\left( \mathbf{1}\otimes D_{M_{j}}\right) ^{\prime }$ (blown-up and doubled
from $\mathbf{1}\otimes D_{M_{j}}$, the twist of $D_{M_{j}}$ by the
canonical isotropy bundle $W^{\tau }$) on the quotient $\widetilde{M_{j}}%
\diagup \overline{\mathcal{F}}$, where the bundle is replaced by the space
of basic sections over each leaf closure; and

\item $n_{\tau }$ is the number of different inequivalent $G_{j}$%
-representation types present in a typical fiber of $W^{\tau }$.
\end{enumerate}
\end{theorem}

An example of this result is the generalization of the Gauss-Bonnet Theorem
to the basic Euler characteristic. The basic forms $\Omega \left( M,\mathcal{%
F}\right) $ are preserved by the exterior derivative, and the resulting
cohomology is called basic cohomology $H^{\ast }\left( M,\mathcal{F}\right) $%
. It is known that the basic cohomology groups are finite-dimensional in the
Riemannian foliation case. See \cite{EKHS}, \cite{KT}, \cite{KT3}, \cite{KT4}%
, \cite{Gh} for facts about basic cohomology and Riemannian foliations. The
basic Euler characteristic is defined to be 
\begin{equation*}
\chi \left( M,\mathcal{F}\right) =\sum \left( -1\right) ^{j}\dim H^{j}\left(
M,\mathcal{F}\right) .
\end{equation*}%
We have two independent proofs of the following Basic Gauss-Bonnet Theorem;
one proof uses the result in \cite{BePaRi}, and the other proof is a direct
consequence of the basic index theorem stated above (proved in \cite{BKR2}).

In the theorem that follows, we express the basic Euler characteristic in
terms of the ordinary Euler characteristic, which in turn can be expressed
in terms of an integral of curvature. We extend the Euler characteristic
notation $\chi \left( Y\right) $ for $Y$ any open (noncompact without
boundary) or closed (compact without boundary) manifold to mean%
\begin{equation*}
\chi \left( Y\right) =%
\begin{array}{ll}
\chi \left( Y\right) & \text{if }Y\text{ is closed} \\ 
\chi \left( 1\text{-point compactification of }Y\right) -1~ & \text{if }Y%
\text{ is open}%
\end{array}%
\end{equation*}%
Also, if $\mathcal{L}$ is a foliated line bundle over a Riemannian foliation 
$\left( X,\mathcal{F}\right) $, we define the basic Euler characteristic $%
\chi \left( X,\mathcal{F},\mathcal{L}\right) $ as before, using the basic
cohomology groups with coefficients in the line bundle $\mathcal{L}$.

\begin{theorem}
(Basic Gauss-Bonnet Theorem, announced in \cite{RiLodz}, proved in \cite%
{BKR2}) \label{BasicGaussBonnet}Let $\left( M,\mathcal{F}\right) $ be a
Riemannian foliation. Let $M_{0}$,..., $M_{r}$ be the strata of the
Riemannian foliation $\left( M,\mathcal{F}\right) $, and let $\mathcal{O}%
_{M_{j}\diagup \overline{\mathcal{F}}}$ denote the orientation line bundle
of the normal bundle to $\overline{\mathcal{F}}$ in $M_{j}$. Let $L_{j}$
denote a representative leaf closure in $M_{j}$. With notation as above, the
basic Euler characteristic satisfies 
\begin{equation*}
\chi \left( M,\mathcal{F}\right) =\sum_{j}\chi \left( M_{j}\diagup \overline{%
\mathcal{F}}\right) \chi \left( L_{j},\mathcal{F},\mathcal{O}_{M_{j}\diagup 
\overline{\mathcal{F}}}\right) .
\end{equation*}
\end{theorem}

\begin{remark}
In \cite[Corollary 1]{GLott}, they show that in special cases the only term
that appears is one corresponding to a most singular stratum.
\end{remark}

\subsection{The orbifold index theorem}

\medskip A smooth orbifold or V-manifold is locally diffeomorphic to the
orbit space of discrete group action on a Euclidean ball. Globally, such a
space may always be rendered as the orbit space of a $O\left( n\right) $
action on a smooth manifold, such that all of the isotropy groups are finite
and the principal isotropy groups are trivial. The orbifold index theorem
was proved in \cite{Kawas1}, \cite{Kawas2}; but we may now use Theorem \ref%
{InvariantIndexTheorem} to express the index of an orbifold elliptic
operator in a slightly different way, because the orbifold index corresponds
to the invariant index of the induced transversally elliptic operator lifted
to the $O\left( n\right) $ action. The integrals over the quotients are
integrals over the orbifold strata. According to Theorem \ref%
{InvariantIndexTheorem}, the index of an orbifold elliptic operator $%
D:\Gamma \left( M,E^{+}\right) \rightarrow \Gamma \left( M,E^{-}\right) $
over a Riemannian orbifold $M$ is%
\begin{eqnarray*}
\mathrm{ind}\left( D\right) &=&\int_{\widetilde{M_{0}}}A_{0}\left( x\right) ~%
\widetilde{\left\vert dx\right\vert }~+\sum_{j=1}^{r}\beta \left( \Sigma
_{j}\right) ~, \\
\beta \left( \Sigma _{j}\right) &=&\frac{1}{2}\sum_{b\in B}\frac{1}{n_{b}%
\mathrm{rank~}W^{b}}\left( -\eta \left( D_{j}^{S+,b}\right) +h\left(
D_{j}^{S+,b}\right) \right) \int_{\widetilde{\Sigma _{j}}}A_{j,b}\left(
x\right) ~\widetilde{\left\vert dx\right\vert }~,
\end{eqnarray*}%
where $\widetilde{M_{0}}$ is the principal stratum of the orbifold (manifold
points) after equivariantly desingularizing along all singular strata and
taking the quotient, and each $\widetilde{\Sigma _{i}}$ is the orbit space
of the fundamental domain of a component of the singular stratum $\Sigma
_{j} $ of the orbifold after desingularizing $\Sigma _{j}$ along all
singular strata properly contained in the closure of $\Sigma _{j}$. Also, $%
A_{0}\left( x\right) $ is the Atiyah-Singer integrand for the operator $D$,
and $A_{j,b}\left( x\right) ~$is the corresponding Atiyah-Singer integrand
on the stratum $\Sigma _{j}$, where the restricted operator is modified by
twisting by the canonical isotropy bundle $W^{b}$ of the isotropy subgroup
of $O\left( n\right) $. The operators $D_{j}^{S+,b}$ on the sphere fibers
are defined as in the rest of the paper and in Theorem \ref%
{InvariantIndexTheorem}.

On the other hand, the orbifold index theorem in \cite{Kawas2} states that%
\begin{eqnarray*}
\mathrm{ind}\left( D\right) &=&\int_{M_{0}}A_{0}\left( x\right) ~\left\vert
dx\right\vert \\
&&+\sum_{i}\frac{1}{m_{i}}\int_{M_{i}}A_{i}\left( x\right) ~\left\vert
dx\right\vert ,
\end{eqnarray*}%
where $A_{0}\left( x\right) $ is the Atiyah-Singer integrand for the
operator $D$, and $A_{i}\left( x\right) ~$is the corresponding Atiyah-Singer
integrand on the set $M_{i}$. The set $M_{i}$ is a connected component of
the blown-up singular set $\Sigma M$ of the orbifold: if a neighborhood of $%
x $ is isometric to $U_{x}\diagup G_{x}$, with $G_{x}<O\left( n\right) $
being the isotropy, 
\begin{equation*}
\Sigma M=\left\{ \left( x,\left( h_{x}\right) \right) :x\in M,\text{ }\left(
h_{x}\right) \text{ is a nontrivial conjugacy class in }G_{x}\right\} .
\end{equation*}%
The centralizer $Z_{G_{x}}\left( h\right) $ of $h\in G_{x}$ is not effective
on $U_{x}$; the order of the trivially acting subgroup of $Z_{G_{x}}\left(
h\right) $ the integer $m_{\left( x,\left( h_{x}\right) \right) }$. This
integer is constant on the connected components $M_{i}$ of $\Sigma M$, so
that one may write $m_{i}=m_{\left( x,\left( h_{x}\right) \right) }$ for any 
$\left( x,\left( h_{x}\right) \right) \in M_{i}$. Clearly, the Kawasaki
result is a different formula than that resulting from our theorem, mainly
because in our case the singular strata are blown up. We note that in \cite%
{Kawas1}, a very similar formula to (\ref{signatureFormula}) above is
provided for the orbifold signature theorem.

\section{Appendix: Generalized spherical harmonics\label{SphereAppendix}}

Collected here are the necessary facts concerning eigenvalues of spherical
operators arising from writing constant coefficient, first order
differential operators in polar coordinates.

\subsection{Polar coordinate form of constant coefficient operator on $%
\mathbb{R}^{k}$\label{PolarCoordFormSection}}

We consider the following situation. Let $Q_{1}:C^{\infty }\left( \mathbb{R}%
^{k},\mathbb{C}^{d}\right) \rightarrow C^{\infty }\left( \mathbb{R}^{k},%
\mathbb{C}^{d}\right) $ be an elliptic, first order, constant coefficient
differential operator on $\mathbb{R}^{k}$ with $k\geq 2$. Let 
\begin{equation*}
Q_{1}=\sum_{j=1}^{k}A_{j}\partial _{j},
\end{equation*}%
where by assumption $A_{j}\in GL\left( d,\mathbb{C}\right) $, and for any $%
c=\left( c_{1},...,c_{k}\right) \in \mathbb{R}^{k}$, $\det \left( \sum
c_{j}A_{j}\right) =0$ implies $c=0$. Note that 
\begin{equation*}
Q_{1}^{\ast }=-\sum_{j=1}^{k}A_{j}^{\ast }\partial _{j}
\end{equation*}%
We form the symmetric operator $Q:C^{\infty }\left( \mathbb{R}^{k},\mathbb{C}%
^{d}\oplus \mathbb{C}^{d}\right) \rightarrow C^{\infty }\left( \mathbb{R}%
^{k},\mathbb{C}^{d}\oplus \mathbb{C}^{d}\right) $ 
\begin{equation*}
Q=\left( 
\begin{array}{cc}
0 & Q_{1}^{\ast } \\ 
Q_{1} & 0%
\end{array}%
\right) =\sum_{j=1}^{k}\left( 
\begin{array}{cc}
0 & -A_{j}^{\ast } \\ 
A_{j} & 0%
\end{array}%
\right) \partial _{j}
\end{equation*}

We write $Q$ in polar coordinates $\left( r,\theta \right) \in \left(
0,\infty \right) \times S^{k-1}$, with $x=r\theta =r\left( \theta
_{1},...,\theta _{k}\right) $, as 
\begin{eqnarray*}
Q_{1} &=&Z^{+}\left( \partial _{r}+\frac{1}{r}Q^{S+}\right) ,\text{ so that}
\\
Q &=&\left( 
\begin{array}{cc}
0 & \left( \partial _{r}+\frac{1}{r}Q^{S+}\right) ^{\ast }Z^{+\ast } \\ 
Z^{+}\left( \partial _{r}+\frac{1}{r}Q^{S+}\right) & 0%
\end{array}%
\right) \\
&=&Z\left( \partial _{r}+\frac{1}{r}Q^{S}\right) ,
\end{eqnarray*}%
where 
\begin{equation*}
Z_{\left( r,\theta \right) }^{+}=\sum_{j=1}^{k}\theta _{j}A_{j}
\end{equation*}%
and%
\begin{eqnarray*}
Q_{\left( r,\theta \right) }^{S+} &=&r\sum_{j=1}^{k}\left( Z^{+}\right)
^{-1}A_{j}\partial _{j}-\sum_{j=1}^{k}x_{j}\partial _{j} \\
&=&r\sum_{j=1}^{k}\left( Z^{+}\right) ^{-1}A_{j}\left( \theta _{j}\partial
_{r}+\frac{1}{r}\frac{\partial }{\partial \theta _{j}}\right)
-\sum_{j=1}^{k}r\theta _{j}\left( \theta _{j}\partial _{r}+\frac{1}{r}\frac{%
\partial }{\partial \theta _{j}}\right) \\
&=&r\left( Z^{+}\right) ^{-1}\left( \sum_{j=1}^{k}A_{j}\theta _{j}\right)
\partial _{r}+\left( Z^{+}\right) ^{-1}\sum_{j=1}^{k}A_{j}\frac{\partial }{%
\partial \theta _{j}}-r\partial _{r}-\sum_{j=1}^{k}\theta _{j}\frac{\partial 
}{\partial \theta _{j}}
\end{eqnarray*}%
\begin{equation}
Q_{\left( r,\theta \right) }^{S+}=\left( Z^{+}\right)
^{-1}\sum_{j=1}^{k}A_{j}\frac{\partial }{\partial \theta _{j}}%
-\sum_{j=1}^{k}\theta _{j}\frac{\partial }{\partial \theta _{j}}=\left(
Z^{+}\right) ^{-1}\sum_{j=1}^{k}A_{j}\frac{\partial }{\partial \theta _{j}}.
\label{DS+formula}
\end{equation}%
Note that $\frac{\partial }{\partial \theta _{j}}$ and $Q^{S+}$ may be
regarded as an operator on the sphere, and the matrix $Z$ is independent of
the radial parameter $r$. Thus the adjoint of $Z$ is also independent of $r$%
. Since $Q^{S+}$ does not differentiate in the $r$ direction, the formal
adjoint of $Q^{S+}$ restricted to the sphere is the same as the formal
adjoint of $Q^{S+}$ as an operator on $\mathbb{R}^{k}$, restricted to
functions that are locally constant in the radial directions. Note that the
formula above implies that if $u$ is a constant vector in $\mathbb{C}^{d}$,
then $Q^{S+}u=0$. Because the formal adjoint of $\partial _{r}$ is $%
-\partial _{r}-\frac{k-1}{r}$ on $\mathbb{R}^{k}$ and by the properties
indicated above, we have 
\begin{eqnarray*}
\left( \partial _{r}+\frac{1}{r}Q^{S+}\right) ^{\ast }Z^{+\ast } &=&\left(
-\partial _{r}-\frac{k-1}{r}+Q^{S+\ast }\circ \frac{1}{r}\right) Z^{+\ast }
\\
&=&-Z^{+\ast }\left( \partial _{r}+\frac{k-1}{r}-\left( Z^{+\ast }\right)
^{-1}Q^{S+\ast }\left( Z^{+\ast }\right) \circ \frac{1}{r}\right) .
\end{eqnarray*}

Thus, we have that 
\begin{equation*}
Z=\left( 
\begin{array}{cc}
0 & Z^{-} \\ 
Z^{+} & 0%
\end{array}%
\right) ,\text{ ~}Q^{S}=\left( 
\begin{array}{cc}
Q^{S+} & 0 \\ 
0 & Q^{S-}%
\end{array}%
\right) ,
\end{equation*}%
with 
\begin{equation*}
Z^{-}=-Z^{+\ast }\text{, }Q^{S-}=\left( k-1\right) I-r\left( Z^{-}\right)
^{-1}Q^{S+\ast }\left( Z^{-}\right) \circ \frac{1}{r}.
\end{equation*}%
Note that by the choice of polar coordinates, $Q^{S-}$ is an operator on the
sphere and does not differentiate in the radial direction, so that 
\begin{eqnarray*}
Q^{S-}\left( r\left( Z^{-}\right) ^{-1}\right) &=&\left[ \left( k-1\right)
I-r\left( Z^{-}\right) ^{-1}Q^{S+\ast }\left( Z^{-}\right) \circ \frac{1}{r}%
\right] \left( r\left( Z^{-}\right) ^{-1}\right) ,\text{ or} \\
rQ^{S-}\left( Z^{-}\right) ^{-1} &=&r\left( k-1\right) \left( Z^{-}\right)
^{-1}-r\left( Z^{-}\right) ^{-1}Q^{S+\ast }.
\end{eqnarray*}%
Multiplying on the left by $\frac{1}{r}Z^{-}$ and solving for $Q^{S+\ast }$,
we obtain 
\begin{eqnarray*}
Q^{S+\ast } &=&\left( k-1\right) I-Z^{-}Q^{S-}\left( Z^{-}\right) ^{-1},%
\text{ or} \\
Q^{S-} &=&\left( k-1\right) I-\left( Z^{-}\right) ^{-1}Q^{S+\ast }\left(
Z^{-}\right) .
\end{eqnarray*}%
In summary, our operator $Q:C^{\infty }\left( \mathbb{R}^{k},\mathbb{C}%
^{d}\oplus \mathbb{C}^{d}\right) \rightarrow C^{\infty }\left( \mathbb{R}%
^{k},\mathbb{C}^{d}\oplus \mathbb{C}^{d}\right) $ satisfies 
\begin{eqnarray}
Q &=&\sum_{j=1}^{k}\left( 
\begin{array}{cc}
0 & -A_{j}^{\ast } \\ 
A_{j} & 0%
\end{array}%
\right) \partial _{j}  \label{D operator polar coords} \\
&=&Z\left( \partial _{r}+\frac{1}{r}Q^{S}\right) =\left( 
\begin{array}{cc}
0 & Z^{-} \\ 
Z^{+} & 0%
\end{array}%
\right) \left( \partial _{r}+\frac{1}{r}\left( 
\begin{array}{cc}
Q^{S+} & 0 \\ 
0 & Q^{S-}%
\end{array}%
\right) \right)  \notag \\
&=&\left( 
\begin{array}{cc}
0 & -Z^{+\ast } \\ 
Z^{+} & 0%
\end{array}%
\right) \left( \partial _{r}+\frac{1}{r}\left( 
\begin{array}{cc}
Q^{S+} & 0 \\ 
0 & \left( k-1\right) I-\left( Z^{+\ast }\right) ^{-1}Q^{S+\ast }\left(
Z^{+\ast }\right)%
\end{array}%
\right) \right)  \notag \\
&=&\left( 
\begin{array}{cc}
0 & Z^{-} \\ 
-Z^{-\ast } & 0%
\end{array}%
\right) \left( \partial _{r}+\frac{1}{r}\left( 
\begin{array}{cc}
\left( k-1\right) I-\left( Z^{-\ast }\right) ^{-1}Q^{S-\ast }\left( Z^{-\ast
}\right) & 0 \\ 
0 & Q^{S-}%
\end{array}%
\right) \right) ,  \notag
\end{eqnarray}%
and 
\begin{equation}
Q^{S\ast }=\left( k-1\right) I-ZQ^{S}Z^{-1}  \label{DSstarFormula}
\end{equation}%
In the case of a Dirac type operator, $Q^{S}$ restricts to a Dirac-type
operator on the sphere that does not inherit the grading.

Note also that%
\begin{equation}
Z_{\left( r,\theta \right) }^{+}=\frac{1}{r}\sum_{j=1}^{k}x_{j}A_{j},\text{
and }Z_{\left( r,\theta \right) }^{+\ast }=\frac{1}{r}%
\sum_{j=1}^{k}x_{j}A_{j}^{\ast }  \label{Z+formula}
\end{equation}%
If we assume that $Z^{+}$ is an isometry, then 
\begin{equation}
\left( Z_{\left( r,\theta \right) }^{+}\right) ^{-1}=\frac{1}{r}%
\sum_{j=1}^{k}x_{j}A_{j}^{\ast }.  \label{Z+InverseFormula}
\end{equation}%
Then%
\begin{equation*}
Q_{\left( r,\theta \right) }^{S+}=r\sum_{j=1}^{k}\left( Z^{+}\right)
^{-1}A_{j}\partial _{j}-\sum_{j=1}^{k}x_{j}\partial
_{j}=\sum_{j,l=1}^{k}x_{l}A_{l}^{\ast }A_{j}\partial
_{j}-\sum_{j=1}^{k}x_{j}\partial _{j},
\end{equation*}%
so that $Q^{S+}$ fixes the vector space of polynomials in $x$ that are
homogeneous of degree $m$, for $m\geq 0$. A similar fact is true for $\left(
Q^{S+}\right) ^{\ast }$. In the following section, it will be shown that in
fact $\left( Q^{S+}\right) ^{\ast }=Q^{S+}$ on appropriate subspaces of
sections. For future reference, note that $Q^{S+}$ is a continuous function
of the matrices $A_{j}$.

\subsection{Eigenvalues of $Q^{S+}$\label{eigenvaluesDSsection}}

Let 
\begin{eqnarray*}
\Delta &=&Q_{1}^{\ast }Q_{1}=\sum -A_{i}^{\ast }A_{j}\partial _{i}\partial
_{j} \\
&=&\left( -\partial _{r}-\frac{k-1}{r}+\frac{1}{r}\left( Q^{S+}\right)
^{\ast }\right) \left( Z^{+}\right) ^{\ast }Z^{+}\left( \partial _{r}+\frac{1%
}{r}Q^{S+}\right) ,\text{ and} \\
\Delta &=&\left( -\partial _{r}-\frac{k-1}{r}+\frac{1}{r}\left(
Q^{S+}\right) ^{\ast }\right) \left( \partial _{r}+\frac{1}{r}Q^{S+}\right)
\end{eqnarray*}%
if we let $Z^{+}$ be an isometry. Let%
\begin{equation*}
B\left( x\right) =\sum -x_{i}x_{j}A_{i}^{\ast }A_{j},
\end{equation*}%
which is, up to a sign, the principal symbol of $\Delta $. Let $\mathcal{P}%
_{m}$ denote the vector space of polynomials in $x$ that are homogeneous of
degree $m$ with coefficients that are $d$-dimensional complex vectors. We
define the following Hermitian inner product $\left\langle ~\cdot ~,~\cdot
~\right\rangle $ on $\mathcal{P}_{m}$. Given any $P,Q\in \mathcal{P}_{m}$,
we write%
\begin{equation*}
P\left( x\right) =\sum P_{\alpha }x^{\alpha },~Q\left( x\right) =\sum
Q_{\alpha }x^{\alpha }
\end{equation*}%
where $\alpha =\left( \alpha _{1},...,\alpha _{k}\right) $ is a multiindex
and $x^{\alpha }=x_{1}^{\alpha _{1}}x_{2}^{\alpha _{2}}...x_{k}^{\alpha
_{k}} $. We define%
\begin{equation*}
\left\langle P,Q\right\rangle =\sum \alpha !P^{\alpha }\cdot \overline{%
Q^{\alpha }},
\end{equation*}%
where $\alpha !=\alpha _{1}!...\alpha _{k}!$. Observe that%
\begin{equation*}
\left\langle P,Q\right\rangle =P\left( \partial \right) \cdot \overline{%
Q\left( x\right) },
\end{equation*}%
where $P\left( \partial \right) $ means $\sum P^{\alpha }\partial _{\alpha }$%
, and where $P\left( \partial \right) \cdot \overline{Q\left( x\right) }$
means $\sum \alpha !P^{\alpha }\cdot \overline{Q^{\beta }}\partial _{\alpha
}x^{\beta }$. Observe that $\left\langle ~\cdot ~,~\cdot ~\right\rangle $ is
a positive definite Hermitian inner product on $\mathcal{P}_{m}$.

\begin{proposition}
For $m\geq 2$, the linear map $\Delta :\mathcal{P}_{m}\rightarrow \mathcal{P}%
_{m-2}$ is onto. Also, letting 
\begin{equation*}
\mathcal{H}_{m}=\ker \Delta \cap \mathcal{P}_{m}=\ker Q_{1}\cap \mathcal{P}%
_{m},
\end{equation*}%
the space $\mathcal{P}_{m}$ is the orthogonal direct sum 
\begin{equation*}
\mathcal{P}_{m}=\bigoplus_{j=0}^{\left\lfloor m/2\right\rfloor }\left(
B\left( x\right) \right) ^{2j}\mathcal{H}_{m-2j}.
\end{equation*}
\end{proposition}

\begin{proof}
Observe that for every $Q\in \mathcal{P}_{m-2}$ and every $P\in \mathcal{P}%
_{m}$,%
\begin{eqnarray*}
\left\langle Q,\Delta P\right\rangle &=&Q\left( \partial \right) \cdot 
\overline{B\left( \partial \right) P\left( x\right) } \\
&=&\sum_{\alpha ,\beta ,i,j}\alpha !Q_{\alpha }\cdot \overline{\left(
-A_{i}^{\ast }A_{j}P_{\beta }\right) }\partial _{\alpha }\left( \partial
_{i}\partial _{j}x^{\beta }\right) \\
&=&\sum_{\alpha ,\beta ,i,j}\alpha !\left( -A_{j}^{\ast }A_{i}Q_{\alpha
}\right) \cdot \overline{P_{\beta }}\partial _{\alpha }\left( \partial
_{j}\partial _{i}x^{\beta }\right) \\
&=&\left( BQ\right) \left( \partial \right) \cdot \overline{P\left( x\right) 
}=\left\langle BQ,P\right\rangle .
\end{eqnarray*}%
If $Q\in \mathcal{P}_{m-2}$ is orthogonal to $\Delta \left( \mathcal{P}%
_{m}\right) $, then the computation above implies that every $P\in \mathcal{P%
}_{m}$ is orthogonal to $BQ\in \mathcal{P}_{m}$, so $BQ=0$ and thus $Q=0$.
(Here we use the fact that $Q_{1}$ and thus $\Delta $ is elliptic.) Thus $%
\Delta :\mathcal{P}_{m}\rightarrow \mathcal{P}_{m-2}$ is onto. Also, $\Delta
P=0$ if and only if $\left\langle BQ,P\right\rangle =0$ for every $Q\in 
\mathcal{P}_{m-2}$, so $\mathcal{P}_{m}=\mathcal{H}_{m}\oplus B\left(
x\right) \mathcal{P}_{m-2}$. The result follows.
\end{proof}

\begin{corollary}
Let $d_{m}$ denote the dimension of $\mathcal{P}_{m}$, which is $d\left( 
\begin{array}{c}
m+k-1 \\ 
m%
\end{array}%
\right) $. The dimension $h_{m}$ of the space $\mathcal{H}_{m}$ is 
\begin{equation*}
h_{m}=\left\{ 
\begin{array}{ll}
d_{m}-d_{m-2} & \text{if }m\geq 2 \\ 
d_{m} & \text{if }m=0,1%
\end{array}%
\right.
\end{equation*}
\end{corollary}

Next, let 
\begin{equation*}
Q=\left( 
\begin{array}{cc}
0 & Q_{1}^{\ast } \\ 
Q_{1} & 0%
\end{array}%
\right) =\sum_{j=1}^{k}\left( 
\begin{array}{cc}
0 & -A_{j}^{\ast } \\ 
A_{j} & 0%
\end{array}%
\right) \partial _{j}
\end{equation*}%
as above, so that 
\begin{equation*}
Q^{2}=\left( 
\begin{array}{cc}
\Delta & 0 \\ 
0 & \left( Z^{+}\right) \widetilde{\Delta }\left( Z^{+}\right) ^{-1}%
\end{array}%
\right) ,
\end{equation*}%
with%
\begin{equation*}
\widetilde{\Delta }=\left( Z^{+}\right) ^{-1}Q_{1}Q_{1}^{\ast }Z^{+}.
\end{equation*}%
Writing $Q$ in polar form 
\begin{equation*}
Q=Z\left( \partial _{r}+\frac{1}{r}Q^{S}\right) ,
\end{equation*}%
note that a polynomial $p_{m}$ of degree $m$ is in the kernel $\mathcal{H}%
_{m}=\ker \Delta $ implies

\begin{eqnarray}
\left( -\partial _{r}-\frac{k-1}{r}+\frac{1}{r}\left( Q^{S+}\right) ^{\ast
}\right) \left( \partial _{r}+\frac{1}{r}Q^{S+}\right) p_{m} &=&0,\text{ or}
\notag \\
\left( -\frac{m+k-2}{r}+\frac{1}{r}\left( Q^{S+}\right) ^{\ast }\right)
\left( \frac{m}{r}+\frac{1}{r}Q^{S+}\right) p_{m} &=&0,\text{ so that} 
\notag \\
\left( -\left( m+k-2\right) +\left( Q^{S+}\right) ^{\ast }\right) \left(
m+Q^{S+}\right) p_{m} &=&0.  \label{Ds+equation}
\end{eqnarray}%
Since the vector space ~$\mathcal{H}_{m}$ of polynomials that are
homogeneous of degree $m$ in $\ker $ $\Delta $ is finite-dimensional and is
fixed by $Q^{S+}$ and $\left( Q^{S+}\right) ^{\ast }$, the following linear
algebra fact is relevant (thanks to George Gilbert and Igor Prokhorenkov):

\begin{lemma}
\label{linAlgLemma}Suppose that $L$ is an $r\times r$ complex matrix that
satisfies%
\begin{equation*}
\left( L-aI\right) \left( L^{\ast }-bI\right) =0,
\end{equation*}%
where $L^{\ast }$ is the adjoint (which is the conjugate transpose), $a,b\in 
\mathbb{R}$, and $I$ is the identity matrix. Then $L$ is Hermitian, and each
eigenvalue is $a$ or $b$.
\end{lemma}

\begin{proof}
If $a=b$, then $L=aI$, since if $M$ is an $r\times r$ matrix with $MM^{\ast
}=0$, then $M=0$. If $a\neq b$, then by taking adjoints we also have that $%
\left( L-bI\right) \left( L^{\ast }-aI\right) =0$. Subtracting the two
equations, we obtain%
\begin{equation*}
\left( a-b\right) L+\left( b-a\right) L^{\ast }=0,
\end{equation*}%
so that $L=L^{\ast }$ and is thus Hermitian, and since its minimal
polynomial is a factor of $\left( x-a\right) \left( x-b\right) $, the result
follows.
\end{proof}

\begin{proposition}
\label{DsEigenvalues}For any $m\in \mathbb{Z}_{\geq 0}$ the restriction $%
\mathcal{H}_{m}^{S}$ of $\mathcal{H}_{m}=\mathcal{P}_{m}\cap \ker \Delta $
to the unit sphere $S^{k-1}\subset \mathbb{R}^{k}$ is the $L^{2}$-orthogonal
direct sum of the eigenspaces of $Q^{S+}$ corresponding to eigenvalues $-m$
and $m+k-2$.
\end{proposition}

\begin{proof}
Lemma \ref{linAlgLemma} and Formula \ref{Ds+equation}.
\end{proof}

Observe that $d_{0}=\dim \mathcal{P}_{0}=\dim \mathcal{H}_{0}=h_{0}=d$, and $%
d_{1}=\dim \mathcal{P}_{1}=\dim \mathcal{H}_{1}=h_{1}=dk$. Letting $%
E_{\lambda }$ denote the eigenspace of $Q^{S+}$ corresponding to eigenvalue $%
\lambda $, by formula (\ref{DS+formula}) we have%
\begin{eqnarray*}
d &=&\dim E_{0}+\dim E_{k-2}\geq d+\dim E_{k-2}\text{, so} \\
\dim E_{0} &=&d;~\dim E_{k-2}=0.
\end{eqnarray*}

Note that we do have similar facts for $\ker \widetilde{\Delta }$, because
also maps $\mathcal{P}_{m}$ to $\mathcal{P}_{m-2}$. If $p_{m}\in \mathcal{P}%
_{m}\cap \ker \widetilde{\Delta }$, we have 
\begin{eqnarray*}
\widetilde{\Delta }p_{m} &=&\left( \partial _{r}+\frac{1}{r}Q^{S+}\right)
\left( -\partial _{r}-\frac{k-1}{r}+\frac{1}{r}\left( Q^{S+}\right) ^{\ast
}\right) p_{m} \\
&=&\left( \frac{m-1}{r}+\frac{1}{r}Q^{S+}\right) \left( -\frac{m+k-1}{r}+%
\frac{1}{r}\left( Q^{S+}\right) ^{\ast }\right) p_{m}=0,
\end{eqnarray*}%
which implies%
\begin{equation*}
\left( m-1+Q^{S+}\right) \left( -\left( m+k-1\right) +\left( Q^{S+}\right)
^{\ast }\right) p_{m}=0.
\end{equation*}%
Then, Lemma \ref{linAlgLemma} implies the following.

\begin{proposition}
\label{HarmonicsDSEigenspaces}For any $m\in \mathbb{Z}_{\geq 0}$ the
restriction $\widetilde{\mathcal{H}_{m}^{S}}$ of $\mathcal{P}_{m}\cap \ker 
\widetilde{\Delta }$ to the unit sphere $S^{k-1}\subset \mathcal{R}^{k}$ is
the $L^{2}$-orthogonal direct sum of the eigenspaces of $Q^{S+}$
corresponding to eigenvalues $-m+1$ and $m+k-1$.
\end{proposition}

\begin{proposition}
\label{harmonicSectionsSpanPolys}Let $\mathcal{P}_{m}^{S}$ denote the
restriction of $\mathcal{P}_{m}$ to the unit sphere $S^{k-1}\subset \mathcal{%
R}^{k}$, and let $\mathcal{H}_{m}^{S}$ be defined similarly. Let $E_{\lambda
}$ denote the eigenspace of $Q^{S+}$ corresponding to eigenvalue $\lambda $.
Then 
\begin{eqnarray*}
\mathcal{P}_{m}^{S} &=&\mathcal{H}_{m}^{S}+\mathcal{H}_{m-2}^{S}+...+%
\mathcal{H}_{m-2\left\lfloor m/2\right\rfloor }^{S} \\
&=&E_{-m}+...+E_{0}+E_{k-1}+...+E_{m+k-2}.
\end{eqnarray*}%
That is, $\mathcal{P}_{m}^{S}$ has a basis consisting of elements of $\ker
\Delta $, restricted to $S^{k-1}$. Similarly, letting $\widetilde{\mathcal{H}%
_{m}^{S}}$ denote the restriction of $\mathcal{P}_{m}\cap \ker \widetilde{%
\Delta }$ to $S^{k-1}$, we have that%
\begin{eqnarray*}
\mathcal{P}_{m}^{S} &=&\widetilde{\mathcal{H}_{m}^{S}}+\widetilde{\mathcal{H}%
_{m-2}^{S}}+...+\widetilde{\mathcal{H}_{m-2\left\lfloor m/2\right\rfloor
}^{S}} \\
&=&E_{-m+1}+...+E_{0}+E_{k-1}+...+E_{m+k-1}.
\end{eqnarray*}%
In particular, we have that for all $m\geq 0$,%
\begin{equation*}
\dim E_{-m}=\dim E_{m+k-1}
\end{equation*}
\end{proposition}

\begin{proof}
Clearly $\mathcal{H}_{m-2j}^{S}\subset \mathcal{P}_{m}^{S}$ for $0\leq j\leq
\left\lfloor m/2\right\rfloor $, since $\left\vert x\right\vert ^{2j}%
\mathcal{H}_{m-2j}\subset \mathcal{P}_{m}$. Also, the subspaces $\mathcal{H}%
_{m}^{S},\mathcal{H}_{m-2}^{S},...,\mathcal{H}_{m-2\left\lfloor
m/2\right\rfloor }^{S}$ are linearly independent. Otherwise, a nontrivial
polynomial in $\ker \Delta $ would be identically zero on $S^{k-1}$. By
Proposition \ref{DsEigenvalues}, we could write this vector-valued
polynomial as a sum of eigenvectors of $Q^{S+}$ and thus each eigenvector
would have to be zero. Since eigenvalues of $Q^{S+}$ corresponding to
distinct $m\in \mathbb{Z}_{\geq 0}$ are distinct, each such eigenvector is
written as a vector of polynomials of pure degree that is identically zero
on the sphere. Because any such polynomial is the zero polynomial, this
implies that there is no nontrivial polynomial in $\ker \Delta $ that is
identically zero. Further, By counting dimensions, the result follows. The
second part of the corollary is proved in a similar way.
\end{proof}

The denseness of polynomials in $L^{2}\left( S^{k-1}\right) $ and the
formulas above imply the following result.

\begin{proposition}
\label{DSspectrumCorollary}The spectrum of $Q^{S+}$ is discrete, and there
is a basis of $L^{2}\left( S^{k-1},\mathbb{C}^{d}\right) $ consisting of
eigenvectors of $Q^{S+}$. The set of eigenvalues of $Q^{S+}$ is a subset of $%
Q=\left\{ ...,-5,-4,-3,-2,-1,0,k-1,k,...\right\} $. Let $\mu _{j}$ denote
the multiplicity of the eigenvalue $j\in Q$. Then the nonnegative integers $%
\mu _{j}$ satisfy the equations%
\begin{eqnarray*}
\mu _{0} &=&d~,\text{ }\mu _{-j}+\mu _{k-2+j}=h_{j}~,\text{ and} \\
\mu _{-j} &=&\mu _{j+k-1}=\sum_{m=0}^{j}\left( -1\right)
^{j+m}h_{m}=d_{j}-d_{j-1} \\
&=&d\binom{j+k-2}{j}\text{ for all }j\geq 0.
\end{eqnarray*}
\end{proposition}

\begin{remark}
Observe that this gives a universal formula for the eigenvalues of any such $%
Q^{S+}$ that depends only on the rank $d$ and the dimension $k$.
\end{remark}

\subsection{The equivariant case}

Suppose that we have the setup as in Section \ref{PolarCoordFormSection} and %
\ref{eigenvaluesDSsection}, but that in addition we are given a subgroup $H$
of $O\left( k\right) $ that acts on $\mathbb{R}^{k}$ in the obvious way and
also acts unitarily on each summand of $\mathbb{C}^{d}\oplus \mathbb{C}^{d}$%
, such that the operator $Q_{1}$ (and thus $Q$) commutes with the $H$%
-action. Then, given an irreducible representation $\alpha :H\rightarrow
GL\left( V\right) $, we consider the restriction of each of the operators in
the sections above to the space of vector-valued functions of type $\alpha $%
. Then the following generalizations of Propositions \ref{DsEigenvalues}, %
\ref{HarmonicsDSEigenspaces}, \ref{harmonicSectionsSpanPolys}, and \ref%
{DSspectrumCorollary} are immediate.

Let the superscript $\alpha $ denote the restriction to the space of
vector-valued functions of type $\alpha $.

Let $d_{m}^{\alpha }$ denote the dimension of $\left( \mathcal{P}_{m}\right)
^{\alpha }=\mathcal{P}_{m}^{\alpha }$. As before, the dimension $%
h_{m}^{\alpha }$ of the space $\mathcal{H}_{m}^{\alpha }=\mathcal{P}_{m}\cap
\left( \ker \Delta \right) ^{\alpha }$ is then 
\begin{equation*}
h_{m}^{\alpha }=\left\{ 
\begin{array}{ll}
d_{m}^{\alpha }-d_{m-2}^{\alpha } & \text{if }m\geq 2 \\ 
d_{m}^{\alpha } & \text{if }m=1,2%
\end{array}%
\right.
\end{equation*}

\begin{proposition}
\label{DsEigenvaluesHversion}For any $m\in \mathbb{Z}_{\geq 0}$ the
restriction $\mathcal{H}_{m}^{S,\alpha }$ of $\mathcal{H}_{m}^{\alpha }=%
\mathcal{P}_{m}\cap \left( \ker \Delta \right) ^{\alpha }$ to the unit
sphere $S^{k-1}\subset \mathcal{R}^{k}$ is the $L^{2}$-orthogonal direct sum
of the eigenspaces of $Q^{S+,\alpha }$ corresponding to eigenvalues $-m$ and 
$m+k-2$.
\end{proposition}

\begin{proposition}
\label{HarmonicsDSEigenspacesHversion}For any $m\in \mathbb{Z}_{\geq 0}$ the
restriction $\widetilde{\mathcal{H}_{m}^{S,\alpha }}$ of $\mathcal{P}%
_{m}\cap \left( \ker \widetilde{\Delta }\right) ^{\alpha }$ to the unit
sphere $S^{k-1}\subset \mathcal{R}^{k}$ is the $L^{2}$-orthogonal direct sum
of the eigenspaces of $Q^{S+,\alpha }$ corresponding to eigenvalues $-m+1$
and $m+k-1$.
\end{proposition}

\begin{proposition}
\label{harmonicSectionsSpanPolysHversion}Let $\mathcal{P}_{m}^{S,\alpha }$
denote the restriction of $\mathcal{P}_{m}^{\alpha }$ to the unit sphere $%
S^{k-1}\subset \mathcal{R}^{k}$, and let $\mathcal{H}_{m}^{S,\alpha }$ be
defined similarly. Let $E_{\lambda }^{\alpha }$ denote the eigenspace of $%
Q^{S+,\alpha }$ corresponding to eigenvalue $\lambda $. Then 
\begin{eqnarray*}
\mathcal{P}_{m}^{S,\alpha } &=&\mathcal{H}_{m}^{S,\alpha }+\mathcal{H}%
_{m-2}^{S,\alpha }+...+\mathcal{H}_{m-2\left\lfloor m/2\right\rfloor
}^{S,\alpha } \\
&=&E_{-m}^{\alpha }+...+E_{0}^{\alpha }+E_{k-1}^{\alpha
}+...+E_{m+k-2}^{\alpha }.
\end{eqnarray*}%
That is, $\mathcal{P}_{m}^{S,\alpha }$ has a basis consisting of elements of 
$\left( \ker \Delta \right) ^{\alpha }$, restricted to $S^{k-1}$. Similarly,
letting $\widetilde{\mathcal{H}_{m}^{S,\alpha }}$ denote the restriction of $%
\mathcal{P}_{m}\cap \left( \ker \widetilde{\Delta }\right) ^{\alpha }$ to $%
S^{k-1}$, we have that%
\begin{eqnarray*}
\mathcal{P}_{m}^{S,\alpha } &=&\widetilde{\mathcal{H}_{m}^{S,\alpha }}+%
\widetilde{\mathcal{H}_{m-2}^{S,\alpha }}+...+\widetilde{\mathcal{H}%
_{m-2\left\lfloor m/2\right\rfloor }^{S,\alpha }} \\
&=&E_{-m+1}^{\alpha }+...+E_{0}^{\alpha }+E_{k-1}^{\alpha
}+...+E_{m+k-1}^{\alpha }.
\end{eqnarray*}%
In particular, we have that for all $m\geq 0$,%
\begin{equation*}
\dim E_{-m}^{\alpha }=\dim E_{m+k-1}^{\alpha }
\end{equation*}
\end{proposition}

\begin{proposition}
\label{DSspectrumCorollaryHversion}The spectrum of $Q^{S+,\alpha }$ is
discrete, and there is a basis of $L^{2}\left( S^{k-1},\mathbb{C}^{d}\right)
^{\alpha }$ consisting of eigenvectors of $Q^{S+,\alpha }$. The set of
eigenvalues of $Q^{S+,\alpha }$ is a subset of $Q=\left\{
...,-5,-4,-3,-2,-1,0,k-1,k,...\right\} $. Let $\mu _{j}^{\alpha }$ denote
the (possibly zero) multiplicity of the eigenvalue $j\in Q$. Then the
nonnegative integers $\mu _{j}^{\alpha }$ satisfy the equations%
\begin{eqnarray*}
\mu _{0}^{\alpha } &=&\dim \mathbb{C}^{d,\alpha }~,\text{ }\mu _{-j}^{\alpha
}+\mu _{k-2+j}^{\alpha }=h_{j}^{\alpha }~,\text{ and} \\
\mu _{-j}^{\alpha } &=&\mu _{j+k-1}^{\alpha }=\sum_{m=0}^{j}\left( -1\right)
^{j+m}h_{m}^{\alpha }=d_{j}^{\alpha }-d_{j-1}^{\alpha }~\text{for all }j\geq
0.
\end{eqnarray*}
\end{proposition}

\begin{remark}
\label{stabilityRemarks}Observe that this gives a universal formula for the
eigenvalues of any such $Q^{S+,\alpha }$ that depends only on the
irreducible representation $\alpha $ and the actions of $H$ on $\mathbb{R}%
^{k}$ and on $\mathbb{C}^{d}$. Furthermore, since $Q^{S+,\alpha }$ varies
continuously with the entries of the matrices $A_{j}$, we note that any
continuous function of the set of eigenvalues (such as the eta invariant or
the dimension of an eigenspace) remains constant as the matrices $A_{j}$
vary continuously (and equivariantly). Furthermore, any spectral function of
the form 
\begin{equation*}
\sum_{\lambda \in \sigma \left( Q^{S+,\alpha }\right) }f\left( \lambda
\right) ,
\end{equation*}%
such as the trace of the heat kernel or the zeta or eta functions, is
invariant under equivariant stable homotopies of the original operator $Q$.
\end{remark}

\end{document}